\documentclass{amsart}

\usepackage{amssymb}
\usepackage{amsmath}
\usepackage{amsthm}
\usepackage{color}
\usepackage{euscript}
\usepackage[linktoc=page,colorlinks,citecolor = green!50!black,linkcolor= blue,urlcolor = blue]{hyperref}
\usepackage{enumitem}
\usepackage{tikz-cd} 
\usepackage{mathrsfs}
\usepackage[all]{xy}
\usepackage{verbatim}

\theoremstyle{plain}
\newtheorem{thm}{Theorem}[section]
\newtheorem{lem}[thm]{Lemma}

\newtheorem{cor}[thm]{Corollary}
\newtheorem{prop}[thm]{Proposition}
\newtheorem{condition}[thm]{Condition}
\newtheorem{assump}[thm]{Assumption}

\theoremstyle{definition}
\newtheorem{defn}[thm]{Definition}
\newtheorem{exmp}[thm]{Example}

\theoremstyle{definition}
\newtheorem{note}[thm]{Note}
\newtheorem{rk}[thm]{Remark}
\newtheorem*{notation}{Notation}

\newcommand{\C}{\mathbb{C}}
\newcommand{\Z}{\mathbb{Z}}

\newcommand{\Om}{\mathcal{O} }
\newcommand{\A}{\mathcal{A}}

\newcommand{\T}{\EuScript{T}}

\newcommand{\cG}{\mathcal{G}}

\newcommand{\tr}{\mathfrak{tr}}

\newcommand{\Tt}{\tilde{\EuScript{T}}}
\newcommand{\B}{\mathcal{B}}

\newcommand{\fM}{\mathfrak{M}}

\newcommand{\scrC}{\EuScript{C}}
\newcommand{\scrF}{\EuScript{F}}

\newcommand{\cW}{\mathcal{W}}

\newcommand{\cQ}{\mathcal{Q}}
\newcommand{\cS}{\mathcal{S}}

\newcommand{\wh}[1]{\widehat{#1}}

\numberwithin{equation}{section}
\numberwithin{figure}{section}
\title[Distinguishing open symplectic mapping tori]{Distinguishing open symplectic mapping tori via their wrapped Fukaya categories}
\author{Yusuf Bar{\i}\c{s} Kartal}
\date{}
\begin{document}
\maketitle
\begin{abstract}
In this paper, we present partial results towards a classification of symplectic mapping tori using dynamical properties of wrapped Fukaya categories. More precisely, we construct a symplectic manifold $T_\phi$ associated to a Weinstein domain $M$, and an exact, compactly supported symplectomorphism $\phi$. $T_\phi$ is another Weinstein domain and its contact boundary is independent of $\phi$. In this paper, we 
distinguish $T_\phi$ from $T_{1_M}$, under certain assumptions (Theorem \ref{mainsymp}). As an application, we obtain pairs of diffeomorphic Weinstein domains with the same contact boundary and whose symplectic cohomology groups are the same, as vector spaces, but that are different as Liouville domains. To our knowledge, this is the first example of such pairs that can be distinguished by their wrapped Fukaya category. 

Previously, we have suggested a categorical model $M_\phi$ for the wrapped Fukaya category $\mathcal{W}(T_\phi)$, and we have distinguished $M_\phi$ from the mapping torus category of the identity. In this paper, we prove $\mathcal{W}(T_\phi)$ and $M_\phi$ are derived equivalent (Theorem \ref{comparison}); hence, deducing the promised Theorem \ref{mainsymp}. Theorem \ref{comparison} is of independent interest as it preludes an algebraic description of wrapped Fukaya categories of locally trivial symplectic fibrations as twisted tensor products.
\end{abstract}
\tableofcontents
\parskip1em
\parindent0em
\section{Introduction}\label{sec:intro2}
\subsection{Statement of results}
One way of studying a symplectic manifold is via its dynamics. In particular, there is a classical invariant of symplectic manifolds called flux, which roughly measures the amount of periodic symplectic isotopies up to Hamiltonians. This invariant has algebraic incarnations as studied by Seidel in \cite{flux}, which can be used to distinguish $A_\infty$-categories associated to symplectic manifolds as shown by us in \cite{ownpaperalg}. 

In \cite{ownpaperalg}, we constructed a category $M_\phi$ that was expected to model the wrapped Fukaya category of an open symplectic mapping torus and we exploited the dynamics of these categories to distinguish them. In this paper, we prove the equivalence of $M_\phi$ with the wrapped Fukaya category and give example applications of the main theorem, such as construction of non-deformation equivalent Liouville domains with the same topology and whose symplectic cohomology groups are the same, as vector spaces. To our knowledge, these are the first examples of different fillings of a contact manifold that cannot be distinguished by the symplectic cohomology groups (as vector spaces), but that can be distinguished by their wrapped Fukaya categories. 

One way to motivate the main construction is as follows: given a symplectic manifold $M$ and a symplectomorphism $\phi$, one defines the (closed) symplectic mapping torus as\begin{equation}
\mathbb{R}\times S^1\times  M/(s,\theta,x)\sim (s+1,\theta,\phi (x)) 
\end{equation}
This is a locally trivial symplectic fibration over the two torus $\mathbb{R}\times S^1 /(s,\theta)\sim (s+1,\theta)$, and one can see it as a construction that turns discrete dynamics into continuous dynamics, i.e. one can realize the action of $\phi$ on the $M$-component as the time $1$-flow of the symplectic vector field obtained by descending $-\partial_s$. On the other hand, the Fukaya category of this manifold is hard to study, as it contains few Lagrangians. For instance, to construct a Lagrangian fibered over a circle in the two torus $\mathbb{R}\times S^1 /(s,\theta)\sim (s+1,\theta)$ in the $s$-direction, one needs a $\phi$-invariant Lagrangian of $M$, which does not necessarily exist. However, if one removes a fiber of this fibration, one obtains an open symplectic manifold, which admits several Lagrangians fibered over the non-compact curves in the punctured torus. This open manifold is roughly what we call the open symplectic mapping torus, and as we will see in this paper its (wrapped) Fukaya category admits a simple algebraic model in terms of the (wrapped) Fukaya category of $M$--- the category $M_\phi$ mentioned above. Heuristically, one can see the category of the closed mapping torus as a deformation of the category associated to open one. One can now study algebraic deformations of the wrapped Fukaya category of the open symplectic mapping torus and show that some of the dynamical properties of the closed mapping torus are inherited by these algebraic deformations. For instance, one can show that the algebraic incarnations of the flux we mentioned above are different for $\phi=1_M$ and for $\phi$ acting non-trivially on the wrapped Fukaya category of $M$ (under some assumptions on $M$). See also Remark \ref{rk:dynamicalprops}. 

The rigorous construction of open symplectic mapping torus is as follows: let $M$ be a Weinstein domain with vanishing first and second Betti numbers and let $\phi$ be a compactly supported (i.e. $\phi$ acts trivially on a neighborhood of $\partial M$), exact symplectomorphism of $M$. Let $\widehat M$ denote the Liouville completion of $M$ (see Seidel \cite[(2.3)]{biasedview}). One can define \textbf{the open symplectic mapping torus of $\phi$} as \begin{equation}
\widehat{T}_\phi:=(\mathbb{R}\times S^1\setminus \Z\times\{1\})\times \widehat M/(s,\theta,x)\sim (s+1,\theta,\phi (x)) 
\end{equation}
There is an obvious projection map $\pi:\widehat T_\phi\rightarrow \widehat T_0$, where $\widehat T_0$ is the punctured 2-torus. $\pi$ is a symplectic fibration with a flat symplectic connection and with fibers isomorphic to $\widehat M$. The symplectic form is
\begin{equation}
\{\omega_{\widehat  M}\}+\pi^*\omega_{\widehat T_0}
\end{equation}
Here, $\omega_{\widehat M}$, resp. $\omega_{\widehat T_0}$ is the symplectic form on $\widehat M$, resp. $\widehat T_0$, and $\{\omega_{\widehat M} \}$ denotes fiberwise $\omega_{\widehat M}$. 

$\widehat T_0$ can be seen as the completion of a torus with one boundary component. This domain will be denoted by $T_0$, and its $\Z$-fold covering space corresponding to covering \begin{equation}
(\mathbb R\times S^1\setminus \Z\times \{1\}) 
\to \widehat{T}_0\end{equation} will be denoted by $\tilde T_0$. One can build a Weinstein domain \begin{equation}\label{eq:firstdefn}
T_\phi:=\tilde T_0\times M/ (s,\theta,x)\sim (s+1,\theta,\phi(x) )
\end{equation}
whose completion gives $\widehat T_\phi$. See Figure \ref{figure:sympmt}. 

We will later prove (Proposition \ref{prop:weinstein}) that $T_\phi$ carries a natural Liouville structure that is deformation equivalent to a Weinstein structure. Moreover, $\partial T_\phi=\partial (T_0\times M)$ as contact manifolds. Our main result is about distinguishing the fillings $T_\phi$ and $T_0\times M$. More precisely:
\begin{thm}\label{mainsymp}
Suppose $M$ satisfies Assumption \ref{assumption:symp} below, and $\phi$ induces a non-trivial action on $\cW(M)$. Then, $T_\phi$ and $T_0\times M$ have inequivalent wrapped Fukaya categories. In particular, they are not graded symplectomorphic.
\end{thm}
By assumptions on $H^1(M)$ and $H^2(M)$, $K_M=\bigwedge^n_\C T^*M$ has a canonical trivialization (that is unique up to homotopy); hence, $\cW(M)$ and $SH^*(M)$ can be $\Z$-graded. Moreover, $K_{T_0}$ can be trivialized using the double cover $\mathbb{R}^2\setminus \mathbb{Z}^2\to T_0$, and this induces a natural trivialization on $K_{T_\phi}$ and a $\Z$-grading on $\cW(T_\phi)$. Theorem \ref{mainsymp} distinguishes the wrapped Fukaya categories with this particular grading. The conclusion of this is that there is no exact symplectomorphism between two domains that preserves the homotopy class of trivializations (i.e. they are not graded symplectomorphic).

$\cW(M)$ and $SH^*(M)$ can be defined with coefficients in $\Z$, but we assume they are defined over $\C$. The assumption we need for Theorem \ref{mainsymp} is:
\begin{assump}\label{assumption:symp}
$\cW(M)$ is cohomologically proper and bounded below in each degree (see Assumption \ref{assumption:alg}), $SH^*(M)$ vanishes for $*<0,*=1,*=2$, and $SH^0(M)=\C$.
\end{assump}
There are many examples of symplectic manifolds satisfying Assumption \ref*{assumption:symp}. 
For instance:
\begin{exmp}\label{exmp:1}
Let $X$ be a smooth hypersurface in $\C\mathbb P^7$ of degree greater than or equal to $9$ and $D\subset X$ be a transverse hyperplane section. Let $M=X\setminus D$ and let $\phi$ be the square of a Dehn twist along a spherical Lagrangian (one can find such Lagrangians easily by considering degenerations of $M$ into varieties with quadratic singularities). Then by Theorem \ref*{mainsymp}, $T_\phi$ and $T_0\times M$ are not graded symplectomorphic. On the other hand, $\phi$ is smoothly isotopic to identity by an unpublished work of Giroux (see Maydanskiy \cite[\S 5.3]{maydlef} and Siegel \cite{kylersquare}). Thus, $T_\phi$ and $T_0\times M$ are diffeomorphic. 
\end{exmp}
\begin{exmp}\label{exmp:2}
Similarly, let $X$ be a smooth hypersurface in $\C\mathbb P^5$ of degree greater than or equal to $7$, and $M$ be complement of a transverse hyperplane section. Let $\phi$ be the eighth power of a Dehn twist. One can show using Krylov \cite{krylov1} and Kauffman and Krylov \cite{krylovkauffman} that $\phi$ is smoothly isotopic to identity (see remarks at the end of \cite[Section 3.1]{krylovkauffman}). Hence, again we obtain a Weinstein domain $T_\phi$ that is different from $T_0\times M$ as a graded Liouville domain, but they are the same as smooth manifolds. 
\end{exmp}
\begin{rk}
As we will show later (Lemma \ref{lem:samesh}), it is possible to give examples that cannot be distinguished by their symplectic cohomology groups, as vector spaces, either. Indeed, this is true for Example \ref{exmp:1} and Example \ref{exmp:2} if we assume the degree of the hypersurface $X$ is at least $14$, resp. $10$.
\end{rk}
That these manifolds satisfy Assumption \ref{assumption:symp} is proven in Section \ref{sec:examples} (see Corollary \ref{cor:hypersurface}). The reason $\phi$ acts non-trivially in either case is that when $\cW(M)$ is $\Z$-graded, $\tau_L$ , the Dehn twist along a spherical Lagrangian $L$, acts on $L$ (considered as an object of $\cW(M)$) as shift by $1-n$. One can consider the situations where an $A_k$-configuration of Lagrangian spheres is embedded into $M$ to produce more sophisticated examples (i.e. examples where action of $\phi$ is different from a shift). 

As these two examples demonstrate, $T_\phi$ is a construction that turns exotic symplectomorphisms (i.e. symplectomorphisms that are isotopic to $1$ in $Diff(M,\partial M) $ but act non-trivially on $\cW(M)$) into exotic Liouville structures. In particular, we can use Theorem \ref{mainsymp} to obtain pairs of diffeomorphic, but not (graded) symplectomorphic Liouville domains for every even $n\geq 4$. Indeed, as we explain now, it is possible to produce non-symplectomorphic examples as well.

Assume $\pi_1(M)=1$ and $n>1$. One can attach subcritical handles along the same isotropic spheres on the boundary of $T_\phi$ and $T_0\times M$ to obtain Weinstein manifolds $M_1$ and $M_2$ satisfying $\pi_1(M_1)=\pi_1(M_2)=1$.
Moreover, attaching subcritical handles does not change the derived equivalence class of the wrapped Fukaya category. A proof of this statement can be found in Ganatra, Pardon and Shende \cite[Cor 1.21]{GPS2}, where one uses the Weinstein property for generation as in \cite[Theorem 1.20]{GPS2} (see also Cieliebak \cite{cieliebakhandle}, Irie \cite{iriehandle} and Bourgeois, Ekholm, Eliashberg \cite{BEE}).  
Combining Theorem \ref{mainsymp} with this fact, we obtain:
\begin{cor}\label{handlecor}
	$M_1$ and $M_2$ give different exact fillings of $\partial M_1=\partial M_2$.
\end{cor}
Notice, after handle attachment, the trivialization of the canonical bundle is unique up to homotopy. Hence, Corollary \ref{handlecor} produces non-symplectomorphic fillings (which is stronger than not being graded symplectomorphic). 

The proof of Theorem \ref{mainsymp} is in two steps. The first is to define an algebraic model $M_\phi$ for $\cW(T_\phi)$, and prove an analogue of Theorem \ref{mainsymp}. This is achieved in \cite{ownpaperalg}. More precisely, we have proven:
\begin{thm}\label{mainthmalg}\cite[Theorem 1.3]{ownpaperalg}
Suppose Assumption \ref{assumption:alg} is satisfied. Assume $\phi$ is not equivalent to the identity functor $1_\A$. Then, $M_\phi$ and $M_{1_\A}$ are not Morita equivalent. In particular, they are not derived equivalent.
\end{thm}
In the statement of Theorem \ref{mainthmalg}, $\A$ is an $A_\infty$-category over $\C$, $\phi$ denotes an auto-equivalence of $\A$, and $M_\phi$ is constructed based on this data (one can assume $\A$ is dg and $\phi$ is strict for the construction of $M_\phi$ and for the proof of the theorem). Theorem \ref{mainthmalg} is an obvious algebraic analogue of Theorem \ref{mainsymp}. 
\begin{assump}\label{assumption:alg}	$\A$ is (homologically) smooth (see Kontsevich and Soibelman \cite{koso} for a definition), proper in each degree and bounded below, i.e. $H^*(hom_\A(x,y))$ is finite dimensional in each degree and vanishes for $*\ll 0$ for any $x,y\in Ob(\A)$. Moreover, $HH^i(\A)$, the $i^{th}$ Hochschild cohomology group of $\A$, is $0$ for $i<0, i=1, i=2$ and is isomorphic to $\C$ for $i=0$.
\end{assump}
\begin{figure}\centering
	\includegraphics[height=4 cm]{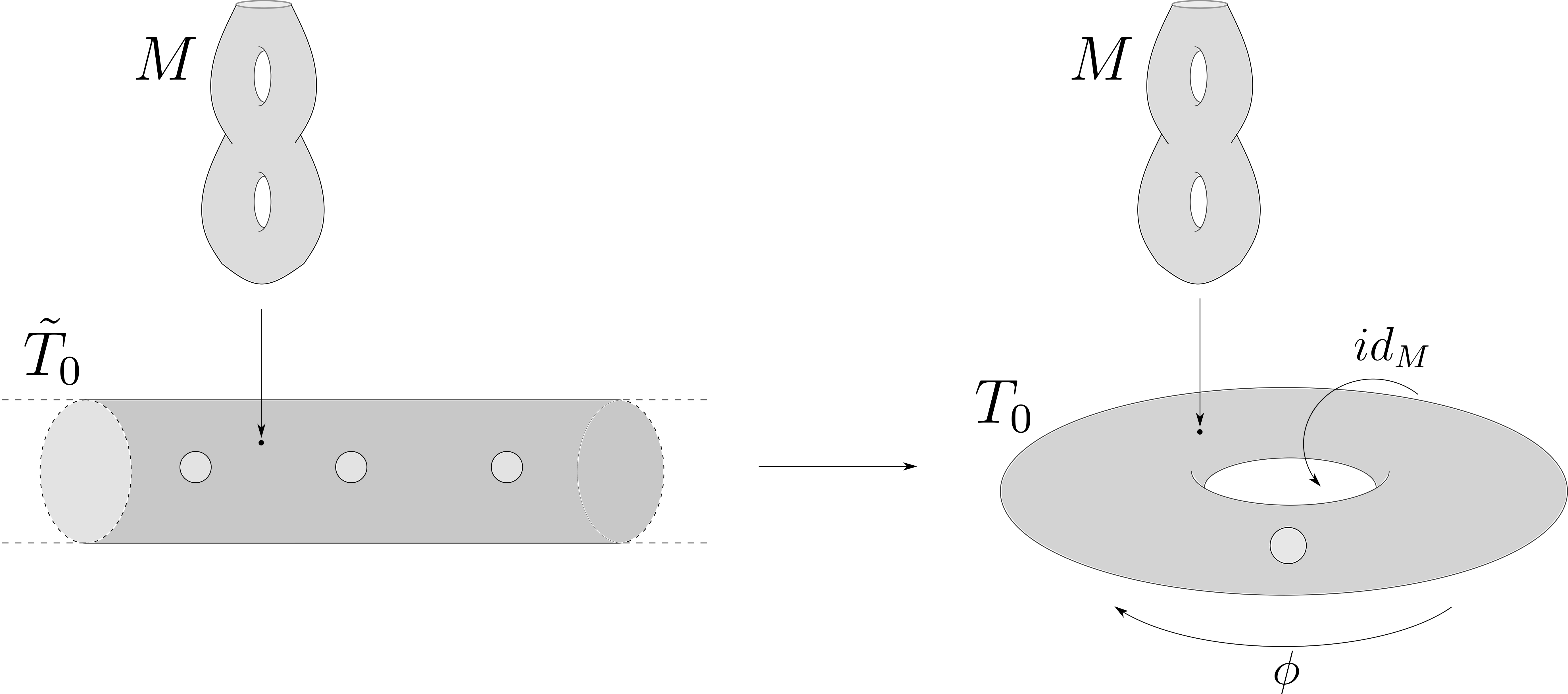}
	\caption{$T_\phi$ and its $\Z$-fold cover $\tilde T_0\times M$}
	\label{figure:sympmt}
\end{figure}

The second step in the proof of Theorem \ref{mainsymp} is the comparison of $M_\phi$ with $\cW(T_\phi)$, and this is the goal of this paper. In other words, we prove:
\begin{thm}\label{comparison}
	$M_\phi$ is Morita equivalent to $\cW(T_\phi)$, if $\A=\cW(M)$ and the auto-equivalence $\phi$ is induced by the given symplectomorphism (which was also denoted by $\phi$).
\end{thm} 
Theorem \ref{mainsymp} is clearly implied by Theorem \ref{mainthmalg} and Theorem \ref{comparison}. More precisely, since $M$ is Weinstein it is non-degenerate in the sense of Ganatra \cite[Definition 1.1]{sheelthesis} (i.e. the generation criterion from Abouzaid \cite{generation} holds). Therefore, it is smooth by \cite[Theorem 1.2]{sheelthesis} and its Hochschild cohomology is isomorphic to $SH^*(M)$ by \cite[Theorem 1.1]{sheelthesis}. Hence, if $\A=\cW(M)$ and the auto-equivalence $\phi$ is induced by the given symplectomorphism, then Assumption \ref{assumption:alg} follows from Assumption \ref{assumption:symp}. Hence, $\cW(T_\phi)\simeq M_\phi$ is different from $\cW(T_0\times M)=\cW(T_{1_M})\simeq M_{1_\A}$ by Theorem \ref{mainthmalg}.
\begin{rk}\label{rk:dynamicalprops}
The proof of Theorem \ref{mainthmalg} uses dynamical properties of (deformations) of categories $M_\phi$. From one perspective, it may be seen as the comparison of a categorical version of Flux groups of $M_\phi$ and $M_{1_\A}$. However, the dynamics is not visible at a geometric level alone; hence, one has to exploit dynamics of Fukaya categories. Moreover, Corollary \ref{handlecor} gives examples of simply connected diffeomorphic fillings that are distinguished by categorical dynamics. As they have vanishing first cohomology, one cannot expect to use any form of flux. Therefore, the dynamics is only visible at the level of Fukaya categories.
\end{rk}
\begin{rk}
Given a pair of strictly commuting auto-equivalences $\phi$ and $\phi'$ on $\A$, one can generalize the mapping torus category $M_\phi$ to the double monodromy mapping torus category $M_{\phi,\phi'}$ (as a twisted tensor product with respect to group action/extra grading by $\Z\times\Z$, see Section \ref{subsec:generalizationtogeneraltwisted}). In this case, $\cW(T_\phi)$ with different choice of gradings will correspond to $M_{\phi[m],[n]}$, where $[m]$ and $[n]$ denote the shift functors. Moreover, one can presumably generalize Theorem \ref{mainthmalg} to distinguish $M_{\phi[m],[n]}$ from $M_{[m'],[n']}$ unless $\phi$ is quasi-isomorphic to a shift functor 
(this requires a simple modification of the technique of \cite{ownpaperalg}, as well as some other minor technical checks, see the remark at \cite[Section 1]{ownpaperalg}). Hence, this would imply that $T_\phi$ and $T_0\times M$ are not symplectomorphic if $\phi$ does not act as a shift, which is stronger than not symplectomorphic as graded symplectic manifolds. 

Similarly, a generalization of Theorem \ref{mainthmalg} that distinguish $M_{\phi}$, resp. $M_{\phi[m],[n]}$ by the order of $\phi$, resp. the order of $\phi$ modulo shifts (see the speculation at the end of \cite[Section 1]{ownpaperalg}) would produce infinitely many different Liouville domains (that are diffeomorphic in the examples above).
\end{rk}
\subsection{Summary of the proof of Theorem \ref*{comparison}}\label{sec:summary}
To prove Theorem \ref{comparison}, we need to give a simpler description of $M_\phi$. We claim it is a ``twisted tensor product'' of $\Om(\T_0)_{dg}$ and $\A$. Here, $\Om(\T_0)_{dg}$ is a dg model for the derived category of coherent sheaves on the nodal elliptic curve $\T_0$ over $\C$ and $\A$ is the $A_\infty$-category used to construct $M_\phi$. This claim is proven in Section \ref{sec:twtensor}. After introducing the notion of twisted tensor products, the claim follows from the definition of $M_\phi$. We recall the definition of $M_\phi$ from \cite{ownpaperalg} for the convenience of the reader.

By the results of Lekili and Polishchuk \cite{lekpol}, $\Om(\T_0)_{dg}$ is also quasi-equivalent to $\cW(T_0)$. On the other hand, the notion of twisted tensor product requires extra gradings on $\Om(\T_0)_{dg}$ and $\cW(T_0)$. We define the extra gradings, and we show in Section \ref{subsec:extragr} that $\Om(\T_0)_{dg}$ and $\cW(T_0)$ are quasi-equivalent as categories with extra grading as well (by reproving the equivalence of these categories via the gluing formula of \cite{GPS2}).

Hence, Theorem \ref{comparison} reduces to the following:
\begin{thm}\label{mainthmsymp}
$\cW(T_\phi)$ is quasi-equivalent to twisted tensor product of $\cW(T_0)$ and $\cW(M)$.
\end{thm}
Note that the notions of twisted tensor product and twisted bimodules are similar to the notion that can be found in Bergh, Oppermann \cite{twistedproduct} and Grimley, Nguyen, Witherspoon \cite{twisted} and perhaps its appearance should not be as surprising for K\"unneth type theorems for symplectic fibrations. 
%
This is a twisted version of K\"unneth theorem for wrapped Fukaya categories. Untwisted versions of this theorem are proven in Ganatra \cite{sheelthesis}, Gao \cite{gaofunctor}, and more recently in Ganatra, Pardon, Shende \cite{GPS2}. 

For simplicity, first ignore compactness issues and describe the main TQFT argument for the proof of fully faithfulness. For this, one can adapt the more analytic definition of wrapped Fukaya categories given in \cite{generation}. To define the ``K\"unneth'' functor, we use the same count of pseudo-holomorphic quilts as in \cite{sheelthesis} and \cite{GPS2} (see Figures \ref{figure:labelledquilt} and \ref{figure:labelledtwisted}), but with target $\wh{\tilde T_0}\times \wh{M}$, and obtain a functor
\begin{equation}\label{eq:kunnethfunctorintro}
\cW(T_\phi)\to Bimod_{tw}(\cW(T_0),\cW(M) )
\end{equation}
where the right hand side denotes the category of twisted bimodules defined in Section \ref{sec:twtensor}. This is simply the category of $\cW(T_0)$-$\cW(M)$-bimodules when $\phi=1_M$. This category is (weakly) generated by ``twisted Yoneda bimodules'' and we show that a generating set of Lagrangian branes (that look like product type) in $\cW(T_\phi)$ map into twisted Yoneda bimodules. The span of twisted Yoneda bimodules is equivalent to the twisted tensor product (in fact, their span can be taken to be the definition of twisted tensor product). We then use a version of Yoneda lemma (Lemma \ref{lem:twyoneda}) and a geometric description of Yoneda map (i.e. the unit insertion) to prove fully faithfulness. 

Unfortunately, applying this idea to the definition of wrapped Fukaya categories via quadratic Hamiltonians given in \cite{generation} comes with analytic difficulties. We have shown in our thesis \cite{kartalthesis} how to solve these problems. However, for this paper we have chosen to switch to the definition of wrapped Fukaya categories given in Ganatra, Pardon and Shende \cite{GPS1}. The latter definition does not use Cauchy-Riemann equations with Hamiltonian terms, instead an auxiliary category $\Om(M)$ is defined first and the ``wrapping'' is done via an algebraic process called localization. This description realizes wrapped Fukaya categories as a colimit. In Section \ref{subsec:algmodify}, we imitate the argument in Section \ref{subsec:defnkunnethearly} for these auxiliary categories. It does not give us a fully faithful functor, but it becomes fully faithful after one passes to colimit.

One still needs to find a class of almost complex structures such that Gromov compactness holds for the moduli spaces of pseudo-holomorphic curves in $\wh{\tilde T_0}\times\wh{M}$. For this, one has to prevent the curves from escaping to
\begin{enumerate}
	\item conical end
	\item left/right ends of the infinite type domain $\wh{\tilde T_0}$ 
\end{enumerate}
To take care of the first, we prefer (integrated) maximum principle of Abouzaid and Seidel \cite{abousei}. More precisely, we define a category $\Om^2(T_\phi)$, that is similar to $\Om^{prod}(X\times Y)$ of \cite{GPS2}. $\Om^{prod}(X\times Y)$ is defined using split type almost complex structures on the product. $T_\phi$ is not a product; however, its conical end looks like that of the product and using almost complex structures that are of split type on the conical end suffices for our purposes. Second issue is solved by choosing infinitely many annuli on $\tilde T_0$ that are placed periodically. We choose our almost complex structures so that each time a curve passes one of these annuli, its energy increases at least by a fixed amount. Therefore, a curve with fixed energy cannot cross infinitely many annuli and has to remain in a finite type subdomain of $\wh{\tilde T_0}\times \wh{M}$.

Therefore, a similar count of ``quilted strips'' gives us a functor 
\begin{equation}
\Om^2(T_\phi)\to Bimod_{tw}(\Om(T_0),\Om(M))
\end{equation}
which gives us a fully faithful functor (\ref{eq:kunnethfunctorintro}) in the limit (more precisely, this functor goes from a category $\cW^2(T_\phi)$ obtained by localizing $\Om^2(T_\phi)$, but one can show it is equivalent to $\cW(T_\phi)$ using the same proof in \cite{GPS2} for the untwisted case).

Finally, we would like to mention another possible proof of Theorem \ref{mainthmsymp}. Recently, a gluing formula for wrapped Fukaya categories appeared in \cite{GPS2}. One can also cut $T_\phi$ into Liouville sectors (that are products of $M$ with simpler sectors). Then, one can use the gluing formula (\cite[Theorem 1.20]{GPS2}), ordinary K\"unneth theorem for sectors (proven in \cite{GPS2}) and the framework of twisted tensor products given in Section \ref{sec:twtensor} to give another proof of Theorem \ref{mainthmsymp}. We sketch this in Appendix \ref{sec:appendixgluing}.
\subsection{Outline}\label{sec:outline}
In Section \ref{sec:structures}, we start by investigating Liouville and Weinstein structures on $T_\phi$. In other words, we show $T_\phi$ carries a natural Liouville structure that is deformation equivalent to a Weinstein structure. We give a description of the cocores of this Weinstein manifold, giving us generators by Chantraine, Dimitrioglou Rizell, Ghiggini, Golovko \cite{Weinsteinsector}, or by \cite[Theorem 1.10]{GPS2}.

In Section \ref{sec:twtensor}, we set up the algebra of twisted tensor products and bimodules. 
We then demonstrate how one can realize $M_\phi$ as a twisted tensor product and comment on the extra gradings on $\cW(T_0)$. In particular, we prove equivalence of $\cW(T_0)$ and $\Om(\T_0)_{dg}$ as extra graded categories.

Section \ref{sec:kunneth} is devoted to proof of Theorem \ref{mainthmsymp}. We start by giving an exposition that ignores analytic difficulties to better illustrate the idea. Then, we switch the definitions, and show how to solve compactness issues and how to modify the idea to fit into this algebraic definition.

In Section \ref{sec:examples}, we give a large class examples of symplectic manifolds satisfying Assumption \ref{assumption:symp}, which let us apply Theorem \ref{mainsymp}, and construct exotic Liouville manifolds as in Corollary \ref{handlecor}.

\subsection{Wrapped Fukaya categories of more general locally trivial fibrations}\label{subsec:generalizationtogeneraltwisted}
It is easy to generalize the notion of twisted tensor product to general discrete groups: if $G$ is a discrete group acting strictly on $\B'$ and $\B$ is an $A_\infty$-category with an extra $G$-grading, then one can define twisted bimodules, Yoneda embeddings etc. similar to Section \ref{sec:twtensor} (note non-commutative groups require some care). This suggests an algebraic description of wrapped Fukaya categories of more general locally trivial symplectic fibrations, i.e. if $G$ acts on the Weinstein domain $M$ by compactly supported exact symplectomorphisms, and if $\tilde B\to B$ is a $G$-fold covering of a Weinstein domain $B$, then one can construct the symplectic manifold $\tilde B\times_G M$ that is a locally trivial symplectic fibration over $B$. Then, one can try to prove the generalization of Theorem \ref{mainthmsymp} that states the wrapped Fukaya category of $\tilde B\times_G M$ is equivalent to the twisted tensor product of the wrapped Fukaya categories of $B$ and $M$. We believe slight modifications of the arguments in Section \ref{sec:structures} and \ref{sec:kunneth} should be sufficient to prove this. For instance, to define a class of almost complex structures for which the moduli of pseudo-holomorphic maps into $\tilde B\times M$ are compact, one can start with a triangulation of $B$ and a fixed neighborhood of codimension $1$-faces of this triangulation. Then, one considers almost complex structures that are of product type over each $b\in \tilde B$ that map to a point in this neighborhood under $\tilde B\to B$ (the $M$-component of the almost complex structure may vary). By a similar argument to Section \ref{subsubsec:noescapelr}, one can arrange this so that the energy of the curve increases by a fixed amount each time it crosses the codimension $1$ skeleton. 
\subsection*{Acknowledgments}
This work is part of a doctoral thesis written
under the supervision of Paul Seidel. I would like to thank him for suggesting the problem and numerous discussions. I would also like to thank Jingyu Zhao, Vivek Shende, Dmitry Tonkonog,  Yank{\i} Lekili, Sheel Ganatra, John Pardon and Zack Sylvan for many helpful conversations and/or explaining their work. I would like to thank Vivek Shende for pointing out to push-out preserving property of $Coh$. Finally, I would like to thank the referee for several useful suggestions. 
This work was partially supported by NSF grant DMS-1500954 and by the Simons Foundation (through a Simons Investigator award).
\section{Structures on the mapping torus}\label{sec:structures}
\subsection{Liouville structure}\label{subsec:liouv}
Let $\wh M$ and $\wh T_0$ denote the completions of $ M$ and $T_0$. Let $\lambda_M$ and $Z_M$ denote the Liouville form and vector field on the completion $\wh M$ as well. We assume $\phi$ is exact, i.e. there exists a smooth function $K$ with compact support in the interior of the Liouville domain $M$ such that $\phi_*(\lambda_M)=\lambda_M+dK$. 
\begin{prop}\label{prop:liouvtphi}
$\wh T_\phi$ has a natural Liouville structure.
\end{prop}
\begin{proof}
We first try the following ansatz for the Liouville form: \begin{equation}\label{eq:ansatz1}
\lambda_{T_\phi}=\{\lambda_s \}+\pi^*\lambda_{T_0} 
\end{equation} where $\lambda_{T_0}$ is a choice of Liouville form on $T_0$ and $\{\lambda_s \}$ refers to a family of Liouville forms on $M$ parametrized by the first coordinate of $\wh T_0$ (and extend them to completions). We will construct $\lambda_s$ as $\lambda_M+dK_s$, where $\{K_s \}$ is a smooth family of compactly supported functions on $M$ as above (extended by $0$ to $\wh M$). If we take the parameter $s$ in $\mathbb{R}$, we have to show that the ansatz (\ref{eq:ansatz1}) induces a 1-form on $\wh T_\phi$, i.e.
\begin{equation}\label{invform}
\phi_*\lambda_s=\lambda_{s+1}
\end{equation}  We will choose $\{\lambda_s\}$ to be constant near every $s\in\Z$ (indeed over $(s-\epsilon,s+\epsilon)$ for a fixed $\epsilon$ such that the hole of the domain $T_0\subset \mathbb{R}\times S^1/(s,\theta)\sim (s+1,\theta)$ have $s$-component in this interval).  

For (\ref{invform}) to hold, we need \begin{equation}d(\phi_*(K_s)+K )=dK_{s+1} 
\end{equation} on $\wh M$ which would be implied by  \begin{equation}\label{eq:invcond} \phi_*(K_s)+K=K_{s+1} \end{equation}(\ref{eq:invcond}) gives us enough data to define family $\{K_s\}$. Namely, fix a small $\epsilon> 0$ (one may assume it is large enough to cover the $s$-component of the hole of $T_0$). Let $\rho:(-\epsilon,1+\epsilon)\rightarrow[0,1]$ be a function such that $\rho(s)=0$ for $s\in(-\epsilon,\epsilon)$ and $\rho(s)=1$ for $s\in(1-\epsilon,1+\epsilon)$. Define $K_s=\rho(s)K$, for $s\in (-\epsilon,1+\epsilon)$. The equality $\phi_*(K_s)+K=K_{s+1}$ holds for $s\in (-\epsilon,\epsilon)$ and we can extend $K_s$ to all $s\in \mathbb{R}$ using (\ref{eq:invcond}). 

Unfortunately, $\lambda_{T_\phi}$ is not a primitive for the original symplectic form. More explicitly \begin{equation}
d(\lambda_{T_\phi})=\{d_M(\lambda_s) \}+d(\lambda_{T_0})+\rho'(s)ds\wedge dK=\omega_{T_\phi}+\rho'(s)ds \wedge dK
\end{equation}
Here, $d_M$ is the exterior derivative along the fiber direction and $d(\lambda_{T_0})$ is used to mean $\pi^*d(\lambda_{T_0})$. We are implicitly using the coordinates $s\in (0,1)$ and the fact that $\rho'(s)$ vanishes near $s=0,1$. We can correct the form $\lambda_{T_\phi}$ as \begin{equation}\label{eq:corrlf}
\lambda_{T_\phi}+\rho'(s)Kds 
\end{equation} and its derivative is clearly $\omega_{T_\phi}$. Moreover, (\ref{eq:corrlf}) looks like $\lambda_s$ near the conical ends of fibers and $\lambda_M+\lambda_{T_0}$ near the puncture; hence, it is a Liouville form.
\end{proof}
\subsection{Weinstein structures on $T_\phi$}\label{subsec:weinstein}
\begin{defn}
A triple $(M,\lambda_M,f_M)$ is called Weinstein if $(M,\lambda_M)$ is a Liouville manifold with Liouville vector field $Z_M$ and $f_M$ is a proper (generalized) Morse function on $M$ such that 
\begin{equation}\label{Lyapunov}
Z_M(f_M)\geq \epsilon(|Z_M|^2+|df_M|^2) 
\end{equation}
for some $\epsilon>0$ (and for some Riemannian metric). If a pair $(Z_M,f_M)$ satisfies (\ref{Lyapunov}), $Z_M$ is called gradient-like for $f_M$ and $f_M$ is called Lyapunov for $Z_M$ (see \cite{cieliebakeliashbergbook} for more details).
\end{defn}
Assume $(M,\omega_M)$ is Weinstein, with Weinstein structure $(M,\lambda_M,Z_M,f_M)$. We construct a Weinstein structure on $T_\phi$ that is Liouville deformation equivalent to the Liouville structure constructed in Proposition  \ref{prop:liouvtphi}. 

Fix a Weinstein structure on $T_0$ such that the handlebody decomposition is as in Figure \ref{figure:handletorus}. The yellow and orange strips (i.e. the vertical and horizontal strips respectively) are the $1$-handles, and the blue and yellow curves (i.e. the vertical and horizontal curves) are the cocores. We denote this Weinstein structure by $(\lambda_{T_0},Z_{T_0},f_{T_0})$. 
\begin{figure}\centering
\includegraphics[height=4 cm]{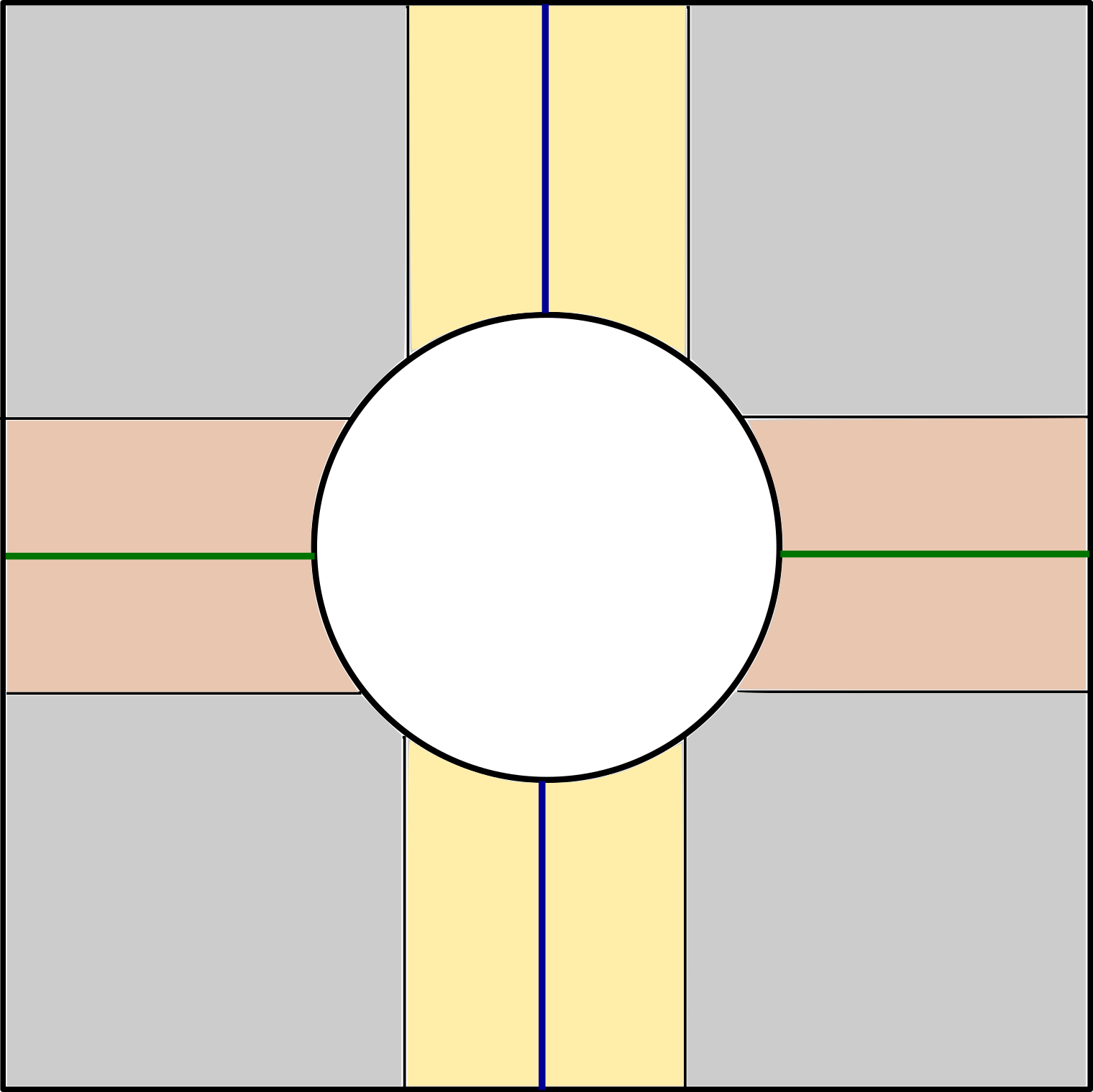}
\caption{Handlebody decomposition of $T_0$}
\label{figure:handletorus}
\end{figure}
The ansatz for the Weinstein structure on $T_\phi$ is the following: Let $\{\lambda_z:z\in \tilde T_0\}$ be a family of Liouville structures on $M$ that descend to mapping torus (in other words $\lambda_{(s+1,t)}=(\phi^{-1})^*\lambda_{(s,t)}$). Assume \begin{equation}\lambda_z=\lambda_M+dK_z
\end{equation} for a family of functions $K_z$ with support uniformly contained in a compact subset of $M\setminus \partial M$. Let $\{f_z:z\in \tilde T_0\}$ be a family of functions on the fibers of $T_\phi\to T_0$ (i.e. a family of functions on $M$ parametrized by $\tilde T_0$ that descend to $T_\phi$). Assume near the critical points of $f_{T_0}$, $\lambda_z$ and $f_z$ do not depend on $z\in \tilde T_0$ and form a Weinstein structure on $M$. Then, \begin{align}
\label{ansatz1}\lambda_{T_0}+C^{-1}\lambda_z\\ f_{T_0}+C^{-1}f_z
\end{align} 
is the ansatz for the Weinstein structure. First, notice (\ref{ansatz1}) is a Liouville form for large enough $C$. To see this consider its differential: 
\begin{equation}
\omega_{T_0}+C^{-1}\omega_M+C^{-1}\nabla \lambda_z
\end{equation}
where $\nabla$ is the natural flat symplectic connection of local system of symplectic manifolds $T_\phi\to T_0$ (in other words, locally it is differentiation in the base direction). Take $(n+1)^{th}$ exterior power to obtain \begin{equation}
C^{-n}\omega_{T_0}\wedge\omega_M^n+O(C^{-n-1}) 
\end{equation}
To ensure it is Liouville, we could assume that $\lambda_z$ is constant near $\partial T_0$ (the middle circle in Figure \ref{figure:handletorus}). However, this is not the best option for other purposes. Instead, we arrange it to be constant over a neighborhood of the part of $\partial T_0$ bounding gray and orange areas. We enlarge this area slightly to include part of yellow strip as well. That it is pointing outward over the rest will follow from the computation below. 

Consider the mapping torus as a fibration over this Weinstein domain. Let $T$ denote the yellow (i.e. vertical) middle strip in Figure \ref{figure:handletorus}. The monodromy $\phi$ is forgotten if we take out the pre-image of $T$. In other words, the complement is a product $(T_0\setminus T)\times M$. Hence, 
the mapping torus can be constructed topologically by gluing $\overline{T_0\setminus T}\times M$ and $T\times M$. We identify the left boundary of $T$ (times $M$) by $id_M$, but we need to twist the right boundary by $\phi$.

Now, we demand the family $(\lambda_z,f_z)$ to be constant and equal to $(\lambda_M,f_M)$ over a small neighborhood of the orange and gray area in Figure \ref{figure:handletorus} (i.e. in a neighborhood of $T_0\setminus T$). Here, we use a trivialization of the local system of symplectic manifolds over this area. To construct a $1$-form and a function over the $1$-handle $T$, we need to construct a family $\{(\lambda_z,f_z):z\in T\}$ that is constant near right and left boundary of $T$ and that interpolates between $(\lambda_M,f_M)$ and $(\phi^{-1})^*(\lambda_M,f_M)$.

The $1$-handle $T$ can be identified with $[-1,1]\times [-1,1]$ in the $qp$-plane such that \begin{align}
f_T=(p^2-\epsilon q^2)/2,\omega_T=dpdq\\
\lambda_T=\frac{pdq+\epsilon qdp}{1-\epsilon}, Z_T=\frac{p\partial_p-\epsilon q\partial_q}{1-\epsilon}
\end{align}
Note that we can simply assume $\epsilon\in (0,1)$ is $1/2$ as we will not let it vary. 
\begin{figure}\centering
	\includegraphics[height=4 cm]{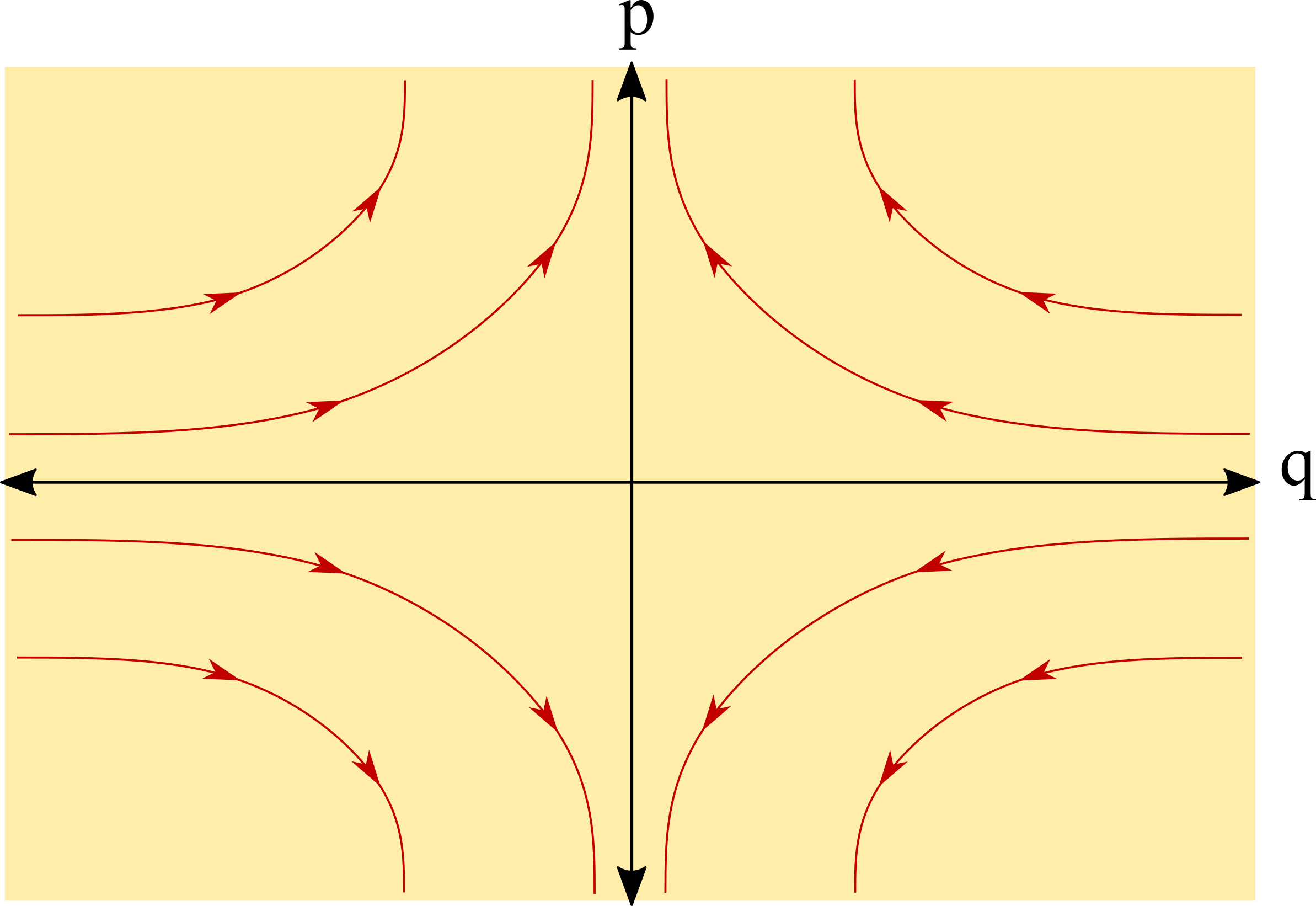}
	\caption{One handle $T$ in $qp$-coordinates}
	\label{figure:onehandle}
\end{figure}

As $(\phi^{-1})^*\lambda_M=\lambda_M+dK$, for some compactly supported function $K$, we can interpolate between $\lambda_M$ and $(\phi^{-1})^*\lambda_M$ by a family $\lambda_q$ (it only depends on the $q$-coordinate) that only differ by compactly supported exact $1$-forms on $M$ such that $\lambda_q=\lambda_M$ on $q\in[-1,\delta]$ for some small $\delta$ and $\lambda_q=(\phi^{-1})^*\lambda_M$ on $q\in [1-\delta,1]$. Similarly, we choose a family $f_q$ of functions on $M$ such that $f_q=f_M$ for $q\in[-1,\delta]$ and $f_q=(\phi^{-1})^*f_M$ for $q\in [1-\delta,1]$. Define 
\begin{align}
\lambda_C=\label{ansatz2}\lambda_{T}+C^{-1}\lambda_q\\ f_C=f_{T}+C^{-1}f_q
\end{align} 
Then we have \begin{equation}\omega_C:=d\lambda_C=dpdq+C^{-1}\omega_M+C^{-1}dq\wedge \nabla_q(\lambda_q)
\end{equation}
which is symplectic for large $C$ as we remarked before. 
We need to compute symplectic dual $Z_C$ of $\lambda_C$. Write \begin{equation}
Z_C=Z_T+Z_q+Z_{corr}
\end{equation}
where $Z_q$ is the Liouville vector field corresponding to $\lambda_q$. Note that $Z_{corr}$ is a correction term not a Liouville vector field. 
Then \begin{multline}\label{zcorr}
i_{Z_{corr}}\omega_C=i_{Z_C}\omega_C-i_{Z_T}\omega_C-i_{Z_q}\omega_C=\\ \lambda_C-\lambda_T+\frac{\epsilon}{1-\epsilon}C^{-1}q\nabla_q(\lambda_q)-C^{-1}\lambda_q+C^{-1}i_{Z_q}\nabla_q(\lambda_q)dq=\\
\frac{\epsilon}{1-\epsilon}C^{-1}q\nabla_q(\lambda_q)+C^{-1}i_{Z_q}\nabla_q(\lambda_q)dq
\end{multline}
Clearly, \begin{equation}
C^{-1}i_{Z_q}\nabla_q(\lambda_q)dq=C^{-1}i_{Z_q}\nabla_q(\lambda_q) i_{\partial_p}\omega_C
\end{equation}
Define \begin{equation}
Z_{corr,1}=Z_{corr}-C^{-1}i_{Z_q}\nabla_q(\lambda_q) \partial_p
\end{equation}
Then \begin{equation}\label{eq:son}
i_{Z_{corr,1}}\omega_C=\frac{\epsilon}{1-\epsilon}C^{-1}q\nabla_q(\lambda_q)
\end{equation}
Write $Z_{corr,1}=g\partial_p+v_f$, where $v_f$ is in fiber direction (there is no $g_1\partial_q$ as this would produce $-g_1dp$ term, which cannot be eliminated). (\ref{eq:son}) implies
\begin{equation}
gi_{\partial_p}\omega_C+i_{v_f}\omega_C=\frac{\epsilon}{1-\epsilon}C^{-1}q\nabla_q(\lambda_q)
\end{equation}
In other words, \begin{equation}
gdq+C^{-1}i_{v_f}\omega_M-C^{-1}i_{v_f}\nabla_q(\lambda_q)dq=\frac{\epsilon}{1-\epsilon}C^{-1}q\nabla_q(\lambda_q)
\end{equation}
Using the natural splitting of tangent spaces into horizontal and vertical directions, we conclude \begin{equation}
	g=C^{-1}i_{v_f}\nabla_q(\lambda_q)\text{ and }i_{v_f}\omega_M=\frac{\epsilon}{1-\epsilon}q\nabla_q(\lambda_q)
\end{equation}
The symplectic dual of $\nabla_q\lambda_q$ is clearly $\nabla_qZ_q$. Hence, \begin{equation}
v_f=\frac{\epsilon}{1-\epsilon}q\nabla_qZ_q\text{ and }g=C^{-1}q\frac{\epsilon}{1-\epsilon}i_{\nabla_qZ_q}(\nabla_q\lambda_q)
\end{equation}
(the latter term actually vanishes). To sum up
\begin{equation}
Z_{corr}=v_f+O(C^{-1})\partial_p
\end{equation}
and thus \begin{equation}Z_C=Z_T+Z_q+\frac{\epsilon}{1-\epsilon}q\nabla_qZ_q+O(C^{-1})\partial_p
\end{equation}
As we mentioned, we do not let $\epsilon$ vary, but for sufficiently large $C$, $Z_T$ dominates $O(C^{-1})\partial_p$ near the upper and lower boundary of $T$. Hence, it is pointing outward there. On the other hand, $Z_q$ are all the same near $\partial M$, so it is pointing outward on $T\times \partial M$ as well. In short, the form is Liouville over the $1$-handle $T$. 

Now, let us examine $f_C$. First, the only critical point of $f_T$ is at $(0,0)$. For large enough $C$, $df_T$ dominates $C^{-1}df_q+C^{-1}dq\wedge \partial_q f_q$ away from $(0,0)$. Near $(0,0)$, $f_q$ is constant and equal to the Morse function $f_M$. Hence, the only critical points of $f_C$ live over $q=p=0$ and they are all non-degenerate. 

Moreover, \begin{equation}Z_C(f_C)=Z_T(f_T)+O(C^{-1})
\end{equation}since $v_f(f_T)=0$. Hence, away from the critical point $(0,0)$, $Z_T(f_T)$ dominates other terms and the Lyapunov property (\ref{Lyapunov}) is satisfied.

Near $(0,0)$, $\lambda_q,f_q$ are constant in $q$; hence, by (\ref{zcorr}) $Z_{corr}=0$, $Z_q=Z_M,f_q=f_M$. This implies \begin{equation}Z_C(f_C)=Z_T(f_T)+C^{-1}Z_M(f_M)
\end{equation}
From this, Lyapunov property is clear. 

By gluing the ``Weinstein structures'' on $\overline{T_0\setminus T}\times M$ and $T\times M$, we obtain:
\begin{prop}\label{iniweiprop}
There exist a Weinstein structure on $T_\phi$ that is of the form \begin{align}
\label{babla}\lambda_{T_\phi}=\lambda_{T_0}+C^{-1}\lambda_z\\ 
f_{T_\phi}=f_{T_0}+C^{-1}f_z\label{iniweimorse}
\end{align}
where $\lambda_z$ is a family of Liouville forms, $f_z$ is a family of functions on $M$, both are locally constant (in $z$) outside one handle $T$ and around the critical point of $T$. This is Weinstein for all sufficiently large $C$.
\end{prop}
Recall how we made the original symplectic structure on $T_\phi$ Liouville. We found a primitive of the form \begin{equation}\label{zort}
\lambda'_{T_\phi}=\lambda_{T_0}+\lambda_s+\rho'(s)Kds
\end{equation}
where $\lambda_s=\lambda_M+\rho(s)d_M(K )$.

Turning $C$ parameter on would effect these only by 
\begin{equation}\label{blaba}
\lambda'_{T_\phi,C}=\lambda_{T_0}+C^{-1}\lambda_s+C^{-1}\rho'(s)Kds
\end{equation}
Now, for large enough $C$, the Liouville structure (\ref{babla}) and (\ref{blaba}) are linearly interpolated by Liouville forms. 
Hence, they are deformation equivalent. (\ref{blaba}) is clearly deformation equivalent to  (\ref{zort}). In summary:
\begin{prop}\label{prop:weinstein}
$T_\phi$ with its standard symplectic structure is Liouville and the corresponding Liouville form is deformation equivalent to Liouville form of a Weinstein structure.
\end{prop}
\subsection{Generators for $\cW(T_\phi)$}\label{subsec:generators}
Now, we will write an explicit set of generators for $\cW(T_\phi)$. As shown in \cite{Weinsteinsector} and \cite[Theorem 1.10]{GPS2}, the cocores of a Weinstein manifold generate its wrapped Fukaya category. The cocores of the Weinstein structure in Proposition \ref{iniweiprop} can be described as follows: The cocores of $T_0$ with the chosen structure are given by green and purple curves in Figure \ref{figure:handletorus} (i.e. the dividing horizontal and vertical curves), which we denote by $L_{gr}$, and $L_{pur}$ respectively. Fix lifts of these curves to $\Z$-fold cover $\tilde T_0\to T_0$, and denote them by $\tilde L_{gr},\tilde L_{pur}$. 
\begin{defn}\label{defn:twlagr}
Let $L'\subset M$ and $L\subset T_0$ be cylindrical Lagrangians with a fixed lift $\tilde L\subset\tilde T_0$ of the latter. Let $L\times_\phi L'$ denote the image of $\tilde L\times L'$ under the projection map $\tilde T_0\times M\to T_\phi$.
\end{defn}
It is easy to see the cocores of critical handles of (\ref{iniweimorse}) are among the Lagrangians $L_{gr}\times_\phi L'$, $L_{pur}\times_\phi L'$. More precisely, if $L'$ is a cocore disc for $M$, moving it along green and purple curves in Figure \ref{figure:handletorus} gives us the cocores of $T_\phi$. 

It is easy to see that by careful choices $L_{gr}\times_\phi L'$ and $L_{pur}\times_\phi L'$ can be forced to stay as exact Lagrangians throughout the Liouville deformations involved. Hence, we have proven
\begin{cor}\label{torusgenerators}$\cW(T_\phi)$ is generated by objects of the form $L_{gr}\times_\phi L'$ and $L_{pur}\times_\phi L'$, where $L'$ is a cocore for $M$.
\end{cor}
\section{Mapping torus categories and twisted tensor products}\label{sec:twtensor}
\subsection{Twisted tensor product, twisted bifunctors and bimodules}
Let $A$ and $A'$ be two ordinary algebras. Assume $A$ carries an extra $\Z$-grading and $A'$ carries an automorphism $\phi$. Following \cite{twistedproduct}, we can define $A\otimes_{tw} A'$ as the algebra with underlying vector space $A\otimes A'$ and with multiplication \begin{equation}\label{eq:commtw}(a_1\otimes a_1').(a_2\otimes a_2')=a_1a_2\otimes \phi^{-|a_2|}(a_1')a_2'
\end{equation}
where $|a_2|$ is the degree of $a_2$ in the extra grading. Hence, one can describe a right module over $A\otimes_{tw}A'$ as a vector space $M$ with a right $A$-module structure 
\begin{equation}
(m,a)\mapsto \mu^{1|1;0}(m|a;)
\end{equation}
and a right $A'$-module structure, 
\begin{equation}
(m,a')\mapsto \mu^{1|0;1}(m|;a')
\end{equation} satisfying
\begin{equation}
\mu^{1|0;1}\big(   \mu^{1|1;0}(m|a;)|  ;a'\big)-
\mu^{1|1;0}\big(   \mu^{1|0;1}(m|;\phi^{|a| }(a') )| a ;\big)=0
\end{equation}
for any $m\in M,a\in A,a'\in A'$.
This is the same as saying $(m.a).a'=(m.\phi^{|a|}(a') ).a  $. 

The definition of such bimodules extends to $A_\infty$-categories immediately. Namely, let $\B$ and $\B'$ be two $A_\infty$-categories. Assume $\B$ carries an extra $\Z$-grading such that the $A_\infty$-structure maps preserve the degree, and $\B'$ is endowed with a strict automorphism $\phi$ without higher maps. 
\begin{figure}\centering
\includegraphics[height= 4 cm]{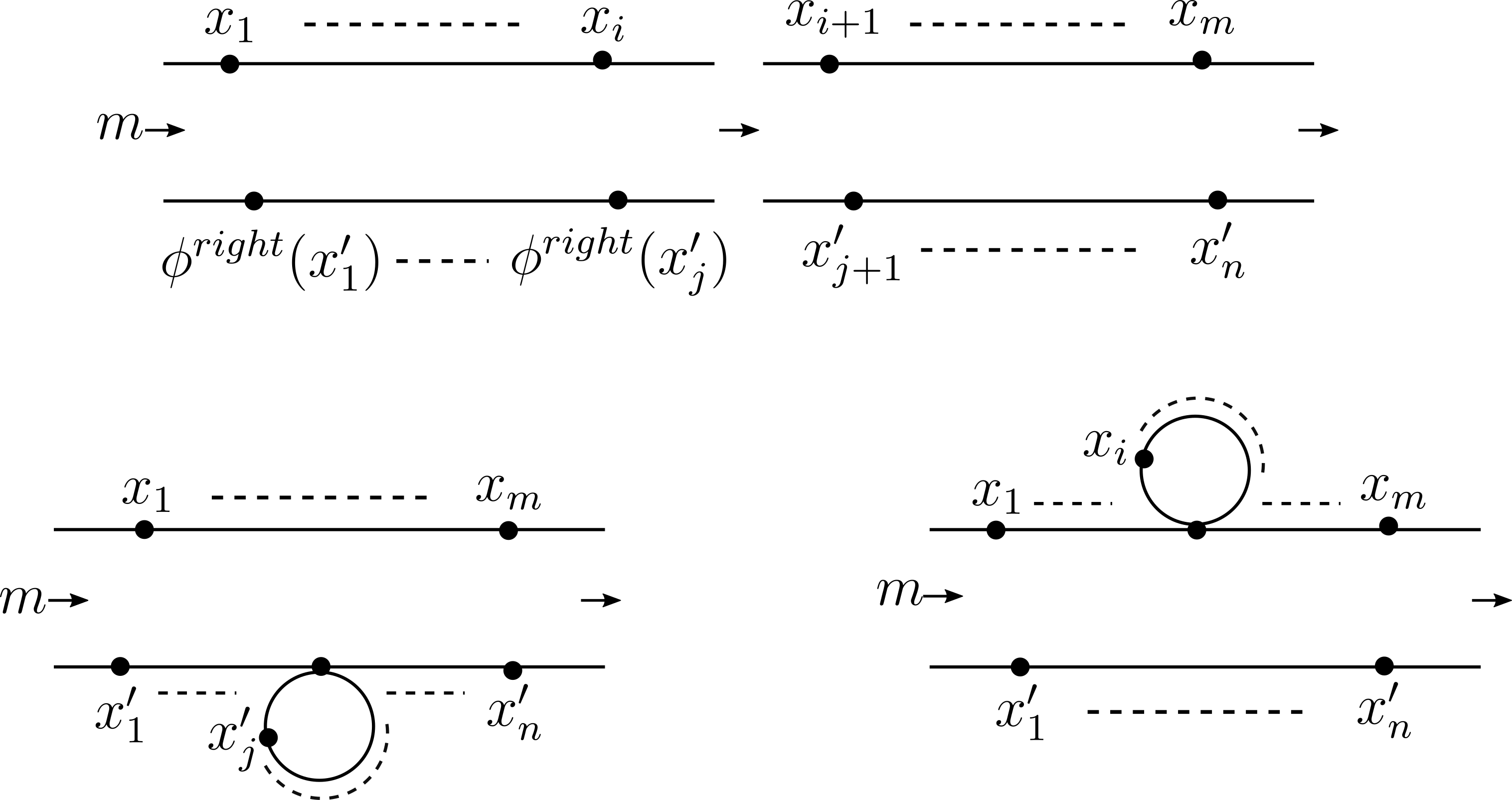}
\caption{A pictorial representation of twisted $A_\infty$-bimodule equations}
\label{figure:maubimoduleequations}
\end{figure}
\begin{defn}A (right-right) twisted $A_\infty$-bimodule $\fM$ over $\B$-$\B'$ is given by an assignment \begin{equation}
(L,L')\in ob(\B)\times ob(\B')\mapsto \fM(L,L')	
\end{equation}
of a $\Z$-graded vector space $\fM(L,L')$ and maps 
\begin{equation}\label{twbimodmaps}
\fM(L_0,\phi^gL'_0)\otimes \B(L_1,L_0)^{g_1}\otimes \B(L_2,L_1)^{g_2}\otimes \dots \otimes \B(L_m,L_{m-1})^{g_m}\otimes\atop \B'(L'_1,L'_0)\otimes\dots \otimes \B'(L'_n,L'_{n-1})\to \fM(L_m,L'_n)[1-m-n]
\end{equation}
where $\B(L_1,L_0)^{g_1}$ denotes the degree $g_1$-part of $\B(L_1,L_0)$ in the extra grading, and $g=\sum g_i$. We will denote these maps by $\mu_\fM=\mu_\fM^{1|m;n}$, omitting $L_i,L'_j$ and the degrees from the notation. These maps are required to satisfy 
\begin{align*}
\sum\pm \mu_\fM(\mu_\fM(m|x_1,\dots x_i;\phi^{right}(x_1',\dots x_j') )| x_{i+1}, \dots x_m; x_{j+1}',\dots ,x_n')+\\\sum\pm \mu_\fM(
m| x_1\dots \mu_\B(\dots)\dots, x_m; x_1'\dots x_n')+\\\sum\pm
\mu_\fM(m| x_1\dots x_m; x_1', \dots\mu_{\B'}(\dots) \dots x_n')=0
\end{align*}
where $\phi^{right}$ denotes $\phi^{|x_{i+1}|+|x_{i+2}|+\dots+|x_m|}$, i.e. $\phi$ applied as many times as the total degree of $x_s$'s on the right, and $\phi^{right}(x_1',\dots x_j')$ means $\phi^{right}$ is applied to each $x_1',\dots, x_j'$ separately, rather than a higher component of $\phi$ (we use this notation in order to shorten the expression). See Figure \ref{figure:maubimoduleequations} for a pictorial representation of the terms in the twisted $A_\infty$-bimodule equations, adapted by \cite{quiltedstrip}. 

A (pre-bimodule) homomorphism $f$ from a twisted bimodule $\fM$ to another $\fM'$ is defined to be a collection of maps 
\begin{equation}\label{twbimodmorphisms}
f^{1|m;n}:\fM(L_0,\phi^gL'_0)\otimes \B(L_1,L_0)^{g_1}\otimes \B(L_2,L_1)^{g_2}\otimes \dots \otimes \B(L_m,L_{m-1})^{g_m}\otimes\atop \B'(L'_1,L'_0)\otimes\dots \otimes \B'(L'_n,L'_{n-1})\to \fM'(L_m,L'_n)[-m-n]
\end{equation}
The differential of the pre-bimodule homomorphism $f$ is 
\begin{align*}
\sum\pm \mu_{\fM'}(f(m|x_1,\dots x_i;\phi^{right}(x_1',\dots x_j') )| x_{i+1}, \dots x_m; x_{j+1}',\dots ,x_n')+\\
\sum\pm f(\mu_\fM(m|x_1,\dots x_i;\phi^{right}(x_1',\dots x_j') )| x_{i+1}, \dots x_m; x_{j+1}',\dots ,x_n')+\\\sum\pm f(
m| x_1\dots \mu_\B(\dots)\dots, x_m; x_1'\dots x_n')+\\\sum\pm
f(m| x_1\dots x_m; x_1', \dots\mu_{\B'}(\dots) \dots x_n')
\end{align*}
(when this is $0$, we say $f$ is a homomorphism of twisted bimodules). Twisted bimodules form a dg category denoted by $Bimod_{tw}(\B,\B')$. The composition is similar to composition of pre-bimodule homomorphisms (see \cite{seidelbook}); however, with a similar twisting rule (see Note \ref{note:rule}).
\end{defn}
\begin{note}\label{note:rule}
One can see ``a twisted $A_\infty$-object'' as an $A_\infty$-object in a symmetric monoidal category of graded vector spaces with an extra $\Z$-grading and a $\Z$-action preserving the extra grading. While the monoidal structure of this category is defined in the obvious way, the braiding map becomes 
\begin{equation}
x\otimes x'\mapsto \pm g^{|x|}.x'\otimes x.g^{|x'|}:= \pm g^{|x|}.x'\otimes g^{-|x'|}.x
\end{equation}
where $|x|$ denotes the extra degree, $g^{|x|}$ denote the corresponding automorphism, and the sign comes from the Koszul convention for the original grading. In our situation, we assume that the extra $\Z$-grading is trivial on one vector space and the $\Z$-action is trivial on the other. Therefore, the general rule to define a twisted $A_\infty$-object is that when swapping morphisms $x$ of $\B$ and $x'$ of $\B'$, one acts on $x'$ by $\phi^{|x|}$ (i.e. $x\otimes x'\mapsto \phi^{|x|}(x')\otimes x$) (for instance, this is the case with composition of pre-bimodule maps etc.). 
\end{note}
\begin{note}\label{note:sign}
We have not specified signs here, but any set of sign conventions for ordinary $A_\infty$-bimodules can be used (in particular signs that one can obtain by unfolding Koszul signs in bar constructions). 
\end{note}
\begin{rk}\label{rk:definetrimodule}
Given another $A_\infty$-category $\B''$, we define a left-right-right $\B''$-$\B$-$\B'$-trimodule with a twisting between the last two components 
as an $A_\infty$-functor $\B''\to Bimod_{tw}(\B,\B')$. In Section \ref{sec:kunneth}, we will define such a functor \begin{equation}
\cW(T_\phi)\to Bimod_{tw}(\cW(T_0),\cW(M) )
\end{equation} to prove Theorem \ref{mainthmsymp}. 
\end{rk}
\begin{rk}
One can weaken the assumption that $\phi$ is a strict auto-equivalence without higher maps. Namely, one can assume $\phi$ is bijective on objects and $\phi^1$ (the first component of $\phi$) is bijective on hom-sets. In other words, the action of $\phi$ induces an automorphism on the coalgebra $T\B'[1]$. In this case, the rule in swapping morphisms $x$ and $x'$ becomes 
\begin{equation}
 x\otimes (x_1'\otimes\dots x_n')\longmapsto\atop \sum_{i_1+\dots+i_k=n}\pm ( (\phi^{|x|})^{i_1}(x_1',\dots x_{i_1}')\otimes \dots \otimes (\phi^{|x|})^{i_k}(x_{n-i_k+1}'\dots x_n') )\otimes x 
\end{equation}
In other words, as before $\phi^{|x|}$ is applied to $x_1'\otimes\dots x_n'$ while moving $x$ to its right; however, this time $\phi^{|x|}$ is considered to be an automorphism of the dg coalgebra $T\B'[1]$. While this allows us to work with minimal models, the definition of twisted Yoneda bimodules (see Example \ref{twistedyoneda}) becomes much more complicated. 
%
We therefore prefer to assume $\phi$ has no higher maps. 
\end{rk}
\begin{exmp}\label{twistedyoneda}(Twisted Yoneda bimodule) Let $(L,L')\in ob(\B)\times ob(\B')$. Define the twisted Yoneda bimodule as \begin{equation}\label{eq:twyon}
\bigoplus_{r\in \Z}\B(\cdot,L)^r\otimes \B'(\cdot,\phi^{-r}(L'))	
\end{equation}
To define structure maps, one uses $A_\infty$-structures of $\B$ and $\B'$, but twists $\B'$ component by the degree of elements of $\B$ on its right. 
More explicitly, the bimodule structure is given by:
\begin{equation}\label{eq:yonedabistructure}
(y\otimes y'|x_1,\dots,x_m;x_1',\dots x_n')\longmapsto\begin{cases}
\pm\mu^1_\B(y)\otimes y'\pm y\otimes \mu^1_{\B'}(y') &m=n=0 \\
\pm\mu_\B(y,x_1,\dots,x_m)\otimes \phi^{-left}(y'),&m\neq 0,n= 0\\ 
\pm y\otimes \mu_{\B'}(y',x_1',\dots,x_n'),&m=0,n\neq 0\\
0,& m\neq 0, n\neq 0
\end{cases}
\end{equation}
where $\phi^{-left}=\phi^{-|x_1|-\dots-|x_m|}$. We will denote twisted Yoneda bimodules by $h_L\otimes_{tw} h_{L'}$, or simply by $h_L\otimes h_{L'}$ when the twisting is trivial. Notice that \begin{equation}
h_L\otimes_{tw}h_{L'}\simeq h_{L\langle 1\rangle }\otimes_{tw}h_{\phi^{-1}(L')}
\end{equation}
where $L\langle 1\rangle$ is the ``shift'' of $L$ defined by $\B(\cdot,L\langle 1 \rangle )^r=\B(\cdot,L)^{r+1}$ (one may enlarge $\B$ by adding these objects, and make it closed under such shifts).
Also, note that $h_L\otimes_{tw} h_{L'}$ only depends on the right Yoneda modules $h_L$, $h_{L'}$ and the compatible extra grading on $h_L$. To show this, identify (\ref{eq:twyon}) with 
\begin{equation}\label{eq:twyon2}
\bigoplus_{r\in \Z}\B(\cdot,L)^r\otimes B'(\phi^r(\cdot),L')=	\bigoplus_{r\in \Z} h_L(\cdot)^r\otimes h_{L'}(\phi^r(\cdot))
\end{equation} and describe structure maps for the latter. 
\end{exmp}
\begin{defn}\label{defn:twdga}
Assume $\B$ and $\B'$ are dg categories. Define the twisted tensor product 	$\B\otimes_{tw}\B'$ to be the dg category satisfying
\begin{enumerate}
\item $ob(\B\otimes_{tw}\B')=ob(\B)\times ob(\B')$
\item $hom(L_1\times L_1',L_2\times L'_2)=\bigoplus_{r\in\Z}\B(L_1,L_2)^r\otimes \B'(L_1',\phi^{-r} (L_2') ) $ as chain complexes
\end{enumerate}
with the composition defined by (\ref{eq:commtw}) (but with Koszul signs). 
\end{defn}
\begin{rk}
Given a model for the tensor product of $A_\infty$ algebras such as the model in \cite{loday}, one can presumably define its twisted version and a Yoneda embedding. However, we will bypass this by considering the full subcategory of $Bimod_{tw}(\B,\B')$ spanned by twisted Yoneda bimodules. This is equivalent to giving an explicit model by Yoneda Lemma (see Lemma \ref{lem:twyoneda}).
\end{rk}
From now on, assume the categories $\B$ and $\B'$ are cohomologically unital with units denoted by $e_L$ and $e_{L'}$. Further assume $\phi$ acts freely on objects of $\B'$, $e_L$ is homogeneous of degree $0$ (in the extra grading) and $\phi$ sends $e_{L'}$ to $e_{\phi(L')}$ for all $L'$. Once we have cohomological unitality, these can be arranged easily by passing to quasi-equivalent models. For instance, given cohomologically unital $\B'$, one can replace it with the category whose objects are pairs $[L',n]$, where $L'\in Ob(\B'),n\in\Z $, and whose homomorphisms from $[L_1',n_1]$ to $[L_2',n_2]$ are given by $\B'(L_1',L_2')$. Then, we let $\phi$ act by $[L',n]\mapsto [\phi(L'),n+1]$. 

We would like to investigate the structure of the category $Bimod_{tw}(\B,\B')$. Most of the following proofs are standard (up to remembering the rule of twisting $x\otimes x'\to \phi^{|x|}(x')\otimes x$ and $x'\otimes x\to x\otimes \phi^{-|x|}(x')$). Nevertheless, we will include them for the convenience of the reader. 
 
First, let us prove something for graded twisted bimodules over graded algebras:
\begin{lem}\label{lem:standardbar}
Let $A$ and $A'$ be graded algebras equipped with an extra grading and an automorphism $\phi$ respectively. Let $M$ be a twisted graded (right-right) $A$-$A'$-bimodule. Then there exists a bar type resolution of $M$ as a twisted bimodule consisting of shifted direct sums of $M\otimes A^{\otimes m}\otimes A'^{\otimes n}$.
\end{lem}
\begin{proof}
First, consider the bar resolution of $M$ with respect to $A$. It is given by \begin{equation}
\{M\otimes A^{\otimes m} \}\to M
\end{equation} with the map $M\otimes A\to M$ being $m\otimes a\mapsto ma$. One can endow this resolution with an $A'$-action making it $A\otimes_{tw} A'$-linear. For instance, define $(m\otimes a).a':=m\phi^{|a|}(a')\otimes a$. Now, apply the standard bar construction as $A'$-modules to each term in the resolution to obtain a double resolution of type $M\otimes A^{\otimes m}\otimes A'^{\otimes n}$. Likewise, one can equip each of these terms with an $A$-action making the double complex $A\otimes_{tw} A'$-linear. By taking the total complex of this double resolution, we obtain what we desire.
\end{proof}
We would also like to prove independence of $Bimod_{tw}(\B,\B')$ from the quasi-equivalence type. Namely:
\begin{lem}\label{lem:eqbimod}
Let $f:\B\to \underline{\B}$ be a quasi-equivalence of extra graded $A_\infty$ categories. Let $f':\B'\to\underline{\B'}$ be a quasi-equivalence of $A_\infty$-categories that are equipped with strict auto-equivalences $\phi$ and $\underline{\phi}$ without higher maps. Assume $f$ is compatible with the extra grading and $f'$ strictly commutes with given auto-equivalences. Then, there is an induced dg quasi-equivalence 
\begin{equation}
F:Bimod_{tw}(\underline{\B},\underline{\B'})\to Bimod_{tw}(\B,\B')
\end{equation} 
\end{lem}
\begin{proof}
For simplicity assume all $A_\infty$-categories have only one object (hence, they are $A_\infty$-algebras). The induced map is standard and twisting does not effect the definition. Namely, if $\fM\in Bimod_{tw}(\underline{\B},\underline{\B'})$, then define $F(\fM)$ to be the bimodule with the same underlying chain complex, and with structure maps
\begin{equation}
(m|x_1,\dots,x_m;x_1',\dots,x_n')\mapsto\atop \sum\pm\mu_\fM(m|f^{i_1}(x_1,\dots, x_{i_1}), f^{i_2}(x_{i_1+1},\dots)\dots;f'^{j_1}(\dots),\dots  )
\end{equation}
Likewise, for a pre-bimodule homomorphism $g$, $F(g)$ is defined by 
\begin{equation}
(m|x_1,\dots,x_m;x_1',\dots,x_n')\mapsto\atop \sum\pm g(m|f^{i_1}(x_1,\dots, x_{i_1}), f^{i_2}(x_{i_1+1},\dots)\dots;f'^{j_1}(\dots),\dots  )
\end{equation}
It is easy to see that $F$ is a dg functor. To see $F$ is cohomologically fully faithful, filter the hom-complexes by the total length. Then, the induced map between associated graded complexes is clearly a quasi-isomorphism. 

To see it is essentially surjective, one can construct a quasi inverse $G$ as 
\begin{equation}\label{eq:mbbar}
``\fM\otimes_\B \underline{\B}\otimes_{\B'}\underline{\B'}"
\end{equation}
In other words, as a complex, $G(\fM)$ is obtained as shifted sums of \begin{equation}\fM\otimes \B^{\otimes m}\otimes \underline{\B}\otimes \B'^{\otimes n}\otimes \underline{\B'}
\end{equation} similar to untwisted case. The only difference is, when defining the structure maps (and action of the functor $G$ on morphisms), one has to take twisting into account. For instance, $\mu_\fM(m|;b')\otimes b$ term in the differential of $m\otimes b\otimes b'$ is replaced by  $\mu_\fM(m|;\phi^{|b|}(b'))\otimes b$. Then,
$FG(\fM)$ is quasi isomorphic to bimodule given by \begin{equation}\label{eq:barmodfg}``\fM\otimes_\B \B\otimes_{\B'}\B'"
\end{equation}as a twisted $\B$-$\B'$-bimodule, where the quasi-isomorphism is induced by $f$ and $f'$ seen as maps of $\B$-$\B$, resp. $\B'$-$\B'$ bimodules $\B\to\underline{\B}$, resp. $\B'\to\underline{\B'}$. There exists a natural map
\begin{equation}\label{eq:barmodcone}
\fM\otimes_\B \B\otimes_{\B'}\B'\to\fM
\end{equation}
of twisted bimodules and one can filter the cone of (\ref{eq:barmodcone}) by the total length. The $E_1$-page of the corresponding spectral sequence is (a union of the summands of) the standard bar resolutions of Lemma \ref{lem:standardbar}, for $A=H^*(\B)$, $A'=H^*(\B')$ and $M=H^*(\fM)$; hence, it is acyclic (to make sure the $E_1$-page agrees with the standard bar resolution in Lemma \ref{lem:standardbar}, one can construct (\ref{eq:mbbar}) and (\ref{eq:barmodfg}) by first taking $\otimes_{\B}\underline{\B}$ and then $\otimes_{\B'}\underline{\B'}$). This implies, the natural map from (\ref{eq:barmodfg}) to $\fM$ is a quasi-isomorphism and we are done.
\end{proof}
The bimodules we will encounter in Section \ref{sec:kunneth} will fall into span of twisted Yoneda bimodules. However, the following is a natural corollary of the proof of Lemma \ref{lem:eqbimod} and we include it here:
\begin{cor}\label{cor:genbimodtw}
The category $Bimod_{tw}(\B,\B')$ is generated by twisted Yoneda bimodules (in the sense that every object is quasi-isomorphic to a (homotopy) colimit of finite complexes of twisted Yoneda bimodules).
\end{cor}
\begin{proof}
The resolution (\ref{eq:barmodfg}) can be seen as an infinite resolution by twisted Yoneda bimodules. More precisely, let $(\fM\otimes_\B \B\otimes_{\B'}\B')^{\leq n}$ denote the submodule of $\fM\otimes_\B \B\otimes_{\B'}\B'$ spanned by chains of length less than $n+1$ (this is a submodule since the structure maps of the bimodule are not increasing the length). This is clearly a finite complex of infinite sums of (shifted) twisted Yoneda bimodules. Moreover, $\fM\otimes_\B \B\otimes_{\B'}\B'$ is a (homotopy) colimit of $(\fM\otimes_\B \B\otimes_{\B'}\B')^{\leq n}$, since homotopy colimits of injective inclusions can be taken as the ordinary colimit (one can describe the homotopy colimit as a cone of two direct sums, with an induced map into (\ref{eq:barmodfg}), then the induced map is a chain equivalence, since the statement that homotopy colimit is the same as the limit is true at the chain level). Therefore, $\fM\otimes_\B \B\otimes_{\B'} \B'$ is a colimit of twisted Yoneda bimodules. 
\end{proof}
As expected, we also have the following:
\begin{lem}\label{lem:twyoneda}[Yoneda Lemma]
The chain complexes $hom(h_L\otimes_{tw}h_{L'},\fM)$ and $\fM(L,L')$ are quasi-isomorphic with a quasi-isomorphism given by 
\begin{align}
\gamma_{L,L'}:hom(h_L\otimes_{tw}h_{L'},\fM)\to \fM(L,L')\\ f \longmapsto f^{1|0;0}(e_L\otimes e_{L'})
\end{align}
\end{lem}
\begin{proof}
The proof of this is similar to \cite[Lemma 2.12]{seidelbook}. Namely, one writes a quasi-inverse
\begin{equation}
\lambda:\fM(L,L')\to hom(h_L\otimes_{tw}h_{L'},\fM)
\end{equation}
similar to \cite[(1.25)]{seidelbook}. For simplicity, assume $\B$ and $\B'$ are dg categories and $\fM$ is a twisted dg bimodule (i.e. has vanishing higher structure maps). Let $d\in \fM(L,L')$. $\lambda(d)$ given by 
\begin{align}
\lambda(d)^{1|0;0}(b\otimes b')=\mu_\fM(\mu_\fM(d|b;)|;b' )\\ \lambda(d)^{1|i;j}=0\text{ if }i\neq 0\text{ or }j\neq 0
\end{align}
defines a right quasi-inverse to $\gamma_{L,L'}$. To see $\lambda$ is a quasi-isomorphism, one can apply the same length filtration spectral sequence argument in \cite[Lemma 2.12]{seidelbook} (more precisely, one has to show exactness of another bar resolution for twisted bimodules: for this one can simply take the dual of the resolution in Lemma \ref{lem:standardbar} or follow its proof to construct the other bar resolution). One can 
generalize the map $\lambda$ to the general $A_\infty$ case (note one has to take twisting into account, the rule is as always $b\otimes x'\to \phi^{-|b|}(x')\otimes b$ etc.) and apply the same proof; however, we take the following route:

Alternative to using more general $\lambda$, one can choose quasi-equivalences from $\B$ and $\B'$ to dg categories $\underline{\B}$ and $\underline{\B'}$ carrying an extra grading and a strict auto-equivalence respectively such that the quasi-equivalences are strictly compatible with extra grading, resp. strictly commute with given auto-equivalences (one can also assume the chosen cohomological units map to strict units, but this is not necessary). Then, the induced map \begin{equation}\label{eq:restbimod}
Bimod_{tw}(\underline{\B},\underline{\B'})\to Bimod_{tw}(\B,\B')
\end{equation} 
is an equivalence by Lemma \ref{lem:eqbimod}. Hence, Yoneda lemma holds in the essential image of dg bimodules. As twisted Yoneda bimodules over $\underline{\B}$-$\underline{\B'}$ are dg and their image under (\ref{eq:restbimod}) are quasi-isomorphic to twisted Yoneda bimodules, the essential image of dg bimodules is all $Bimod_{tw}(\B,\B')$ by Corollary \ref{cor:genbimodtw}. This finishes the proof.
\end{proof}
\subsection{Mapping torus category as a twisted tensor product}\label{subsec:mttwi}
Let $\Tt_0$ denote the universal cover of the nodal elliptic curve $\T_0$, which is an infinite chain of projective lines. See Figure \ref{figure:catgluingcover}. $\Tt_0$ carries a translation automorphism denoted by $\tr$. It moves every projective line to the next (to the right in Figure \ref{figure:catgluingcover}) and generates the group of Deck transformations of $\Tt_0\to\T_0$, where $\T_0$ is the nodal elliptic curve over $\C$, which can as well be defined by $\T_0:=\Tt_0/(y\sim\tr(y))$. In \cite{ownpaperalg}, we constructed a dg category $\Om(\Tt_0)_{dg}$ such that \begin{equation}
H^0( tw^\pi(\Om(\Tt_0)_{dg} )  )\simeq D^b(Coh_p(\Tt_0) )
\end{equation}where $Coh_p(\Tt_0)$ is the abelian category of coherent sheaves with proper support on $\Tt_0$. The objects of $\Om(\Tt_0)_{dg}$ correspond to sheaves $\Om_{C_i}(-1)$ and $\Om_{C_i}$. We use $\Om_{C_i}(-1)$ and $\Om_{C_i}$ to denote the corresponding objects of $\Om(\Tt_0)_{dg}$ as well. Push-forward along $\tr$ induces a strict dg auto-equivalence of $\Om(\Tt_0)_{dg}$, which we still denote by $\tr$. 

Recall the following construction from \cite{ownpaperalg}: let $\A$ be a dg category, and let $\phi$ be a strict dg auto-equivalence of $\A$. Define \textbf{the mapping torus category} $M_\phi$ as the dg category with objects 
\begin{equation}\label{eq:obsofcat}
ob(M_\phi):=ob(\Om(\Tt_0)_{dg})\times ob(\A)
\end{equation} and with morphisms \begin{equation}\label{eq:homsofcat}
M_\phi(\scrF\times a,\scrF'\times a')=\bigoplus_{n\in \Z} \Om(\Tt_0)_{dg}(\scrF,\tr^{-n}(\scrF'))\otimes \A(a,\phi^{-n}(a'))
\end{equation} for $\scrF,\scrF'\in ob(\Om(\Tt_0)_{dg})$ and $a,a'\in ob(\A)$. (\ref{eq:obsofcat}) and (\ref{eq:homsofcat}) can be written concisely as \begin{equation}M_\phi:=(\Om(\Tt_0)_{dg}\otimes \A)\#\Z
\end{equation}
To define the mapping torus category for a more general $A_\infty$-category $\A$ with a strict quasi-equivalence $\phi$ (possibly with higher components), one has to find a dg category $\A^{str}$, a strict dg auto-equivalence $\phi^{str}$ on $\A^{str}$ and a quasi-equivalence $\A\to \A^{str}$ that commutes with $\phi$ and $\phi^{str}$. For instance, one can let $\A^{str}$ to be the Yoneda image of $\A$ under the right Yoneda embedding and $\phi^{str}$ to be $(\phi^{-1})^*$, the pull-back under the strict inverse $\phi^{-1}$ of $\phi$. 

Alternatively, one can describe $\Om(\Tt_0)_{dg}\otimes \A$ using a model for tensor products of $A_\infty$-algebras (such as \cite{loday}) and $(\Om(\Tt_0)_{dg}\otimes \A)\#\Z$ as a Grothendieck construction (see Note \ref{note:finalnote}).
\begin{rk}In \cite{ownpaperalg}, hom-sets were defined as \begin{equation}\label{eq:mphihom2}
\bigoplus_{n\in \Z} \Om(\Tt_0)_{dg}(\tr^n(\scrF),\scrF')\otimes \A(\phi^n(a),a')\end{equation} instead of (\ref{eq:homsofcat}). It is easy to identify (\ref{eq:homsofcat}) and (\ref{eq:mphihom2}) as chain complexes and under this identification, one can describe the product structure on $M_\phi$ by (\ref{eq:mphiprd2}).
Hence, the definitions are equivalent, but (\ref{eq:homsofcat}) is better suited for description of $M_\phi$ as a twisted tensor product.
\end{rk}
Let $\Om(\T_0)_{dg}$ denote the category with objects $\Om_{C_0}(-1),\Om_{C_0}$ and morphisms \begin{equation}\label{eq:ommorphism}
\Om(\T_0)_{dg}(\scrF,\scrF')=\bigoplus_{n\in \Z} \Om(\Tt_0)_{dg}(\scrF,\tr^{-n}(\scrF'))
\end{equation}
Endow $\Om(\T_0)_{dg}$ with an extra $\Z$-grading by setting $\Om(\Tt_0)_{dg}(\scrF,\tr^{-n}(\scrF'))$ to be the degree $n$ morphisms of $\Om(\T_0)_{dg}(\scrF,\scrF')$. 
\begin{prop}\label{prop:mttwistedtensor}
Let $\A$ be a dg category and $\phi$ be a strict dg auto-equivalence. Then, $M_\phi$ is the twisted tensor product of $\Om(\T_0)_{dg}$ and $\A$.
\end{prop}
\begin{proof}
This becomes a tautology once one recalls the product structure on $M_\phi$ and $\Om(\T_0)_{dg}$. For instance, given \begin{equation}
\alpha_1\otimes f_1\in \Om(\Tt_0)_{dg}(\scrF,\tr^{-m}(\scrF'))\otimes \A(a,\phi^{-m}(a') )\subset M_\phi(\scrF\times a,\scrF'\times a')\atop
\alpha_2\otimes f_2\in \Om(\Tt_0)_{dg}(\scrF',\tr^{-n}(\scrF''))\otimes \A(a',\phi^{-n}(a'') )\subset M_\phi(\scrF'\times a',\scrF''\times a'')
\end{equation}
the product in $M_\phi$ is defined as \begin{equation}\label{eq:mphiprd2}
(\alpha_2\otimes f_2)(\alpha_1\otimes f_1)=\pm\tr^{-m}(\alpha_2)\alpha_1\otimes \phi^{-m}(f_2)f_1
\end{equation}
where $\pm$ is the Koszul sign coming from switching $f_2$ and $\alpha_1$. On the other hand, considered as morphisms in $\Om(\T_0)_{dg}$, i.e. elements of (\ref{eq:ommorphism}), the product of $\alpha_1$ and $\alpha_2$ is given by $\tr^{-m}(\alpha_2)\alpha_1$. By definition of the extra grading on $\Om(\T_0)_{dg}$, $\alpha_1$ is homogenous of degree $m$. Therefore, the description (\ref{eq:mphiprd2}) of the product in $M_\phi$ coincide with the product in $\Om(\T_0)_{dg}\otimes_{tw}\A$ under the natural identification of hom-complexes of $M_\phi$ and $\Om(\T_0)_{dg}\otimes_{tw}\A$ (compare Definition \ref{defn:twdga} with (\ref{eq:homsofcat})). 
\end{proof}
A natural question one can ask is the dependence of quasi-equivalence type of $M_\phi$ on the dg model $\Om(\Tt_0)_{dg}$. One knows any other dg model for $D^b(Coh_p(\Tt_0))$ is quasi-equivalent to $\Om(\Tt_0)$ by the main result of \cite{orlovlunts}. Moreover,
one can improve (zigzags) of quasi-equivalence(s) to make it strictly $\tr$-equivariant. More precisely:
\begin{lem}\label{lem:zigzagt0}
Consider pairs $(\B,\psi)$, where $\B$ is an $A_\infty$-category, $\psi$ is an auto-equivalence acting bijectively on objects and hom-sets, and acting freely on objects. Let $(\Om',\tr')$ be another model for $D^b(Coh_p(\Tt_0))$ with the same set of objects as $\Om(\Tt_0)_{dg}$ and with a strict lift $\tr'$ of $\tr_*$. Then, there exists a zigzag of $A_\infty$-quasi-equivalences between $(\Om(\Tt_0)_{dg},\tr)$ and $(\Om',\tr')$ through pairs $(\B,\psi)$.
\end{lem}
A very simple proof of the lemma can be given by using a push-out description of $\Om(\Tt_0)_{dg}$. We will explain this later in this section. See Note \ref{note:zigzagbylocal}.

Lemma \ref{lem:zigzagt0} implies that the extra grading on $\Om(\T_0)_{dg}=\Om(\Tt_0)_{dg}\#\Z$ is independent of the chosen dg model for which $\tr$ lifts as a strict dg auto-equivalence. Hence:
\begin{lem}
$M_\phi$ does not depend on the chosen dg model $\Om(\Tt_0)_{dg}$ or on the chosen strictification $(\A^{str},\phi^{str})$	
\end{lem}
\begin{proof}
The twisted tensor product of dg categories is equivalent to span of twisted Yoneda bimodules. Changing the model for $\Om(\Tt_0)_{dg}$ or the strictification for $(\A,\phi)$ does not change this span by Lemma \ref{lem:eqbimod}.
\end{proof}
Now, we want to express $\Om(\Tt_0)_{dg}$ as a homotopy push-out, as this will be used in Section \ref{subsec:extragr}. This will also give a proof of Lemma \ref{lem:zigzagt0}.

Consider the normalization map 
\begin{equation}\label{eq:normal}
\pi_N:\mathbb{P}^1\times \Z\to \Tt_0
\end{equation}
In the notation of \cite{ownpaperalg}, $\mathbb{P}^1\times\{i\}$ is the component that maps to $C_i\subset \Tt_0$, $0\in\mathbb{P}^1$ maps to the nodal point $x_{i-1/2}\in\Tt_0$ and $\infty\in\mathbb{P}^1$ maps to the nodal point $x_{i+1/2}\in\Tt_0$. We also assume $\tr$ lifts to the normalization as $(y,i)\mapsto (y,i+1)$ (the lift is also denoted by $\tr$). Choose a dg enhancement for $D^b(Coh(\mathbb{P}^1) )$ and take the subcategory spanned by $\Om_{\mathbb P^1},\Om_{\mathbb P^1}(-1),\Om_0$ and $\Om_\infty$. Denote it by $\Om(\mathbb P^1)_{dg}$. Without loss of generality enlarge the category $\Om(\Tt_0)_{dg}$ by adding objects corresponding to nodes $\Om_{x_{i+1/2}}$ in a $\tr$-equivariant way. This does not change the twisted envelope obviously and it causes $\Om(\T_0)_{dg}=\Om(\Tt_0)_{dg}\#\Z$ to enlarge in its twisted envelope as well (together with the natural extra grading). One can choose the enhancement $\Om(\mathbb{P}^1)_{dg}$ so that
\begin{enumerate}
\item There is a dg functor $\Xi_0$ from $\Om(\mathbb P^1)_{dg}$ to $\Om(\Tt_0)_{dg}$ lifting the push-forward of $\mathbb P^1\to C_0\subset\Tt_0$ 
\item There are dg functors $i_0,i_\infty:\C\to \Om(\mathbb P^1)_{dg}$ lifting the push-forward of $\{0\}\to\mathbb{P}^1$ and $\{\infty\}\to\mathbb{P}^1$ 
\item Compositions $\Xi_0\circ i_0,\Xi_0\circ i_\infty:\C\to \Om(\mathbb P^1)_{dg}\to\Om(\Tt_0)_{dg}$ are strictly related by $\tr$, i.e.  $\tr\circ\Xi_0\circ i_0=\Xi_0\circ i_\infty$
\end{enumerate}
Define $\Xi_i:=\tr^i\circ\Xi_0$. These are dg functors lifting the push-forward of $\mathbb P^1\to C_i\subset\Tt_0$. Taking $\Z$-many pairwise orthogonal copies of this dg category, we obtain a dg model for the properly supported coherent sheaves on $\mathbb{P}^1\times \Z$, denoted by $\Om(\mathbb{P}^1\times \Z)_{dg}$. 

There is a dg lift of the push-forward of normalization map (\ref{eq:normal}) which we also denote by $\pi_N$. Let $Pt_\infty$ denote the dg category consisting of infinitely many copies of $\C$ indexed by $i+1/2,i\in\Z$. Denote its objects by $*_{i+1/2}$. $Pt_\infty$ has an auto-equivalence mapping $*_{i-1/2}$ to $*_{i+1/2}$, which we still denote by $\tr$. The collection of functors \begin{equation}
i_0:(\C)_{i+1/2}=:\C\to \Om(\mathbb{P}^1 )_{dg}\cong\Om(\mathbb{P}^1\times \{i+1\} )_{dg}
\end{equation}
(i.e. $i_0$ used to map $*_{i+1/2}$ to $(i+1)^{th}$ $\mathbb{P}^1$) gives a functor \begin{equation}
Pt_\infty\to \Om(\mathbb{P}^1\times \Z)_{dg}
\end{equation} which is essentially the push-forward of $0\in \mathbb{P}^1$ at each component. Denote this functor by $i_0$ as well. Similarly, the collection of functors $i_\infty:(\C)_{i+1/2}\to \Om(\mathbb{P}^1\times \{i\} )_{dg}$ gives a functor \begin{equation}
Pt_\infty\to \Om(\mathbb{P}^1\times \Z)_{dg}
\end{equation} which is essentially push-forward of $\infty\in \mathbb{P}^1$ at each component (to a different $\mathbb P^1$ though). Denote it by $i_\infty$. 
Clearly, $\pi_N\circ i_0=\pi_N\circ i_\infty$. In other words, we have a strictly commutative diagram
\begin{equation}\label{eq:acommdiag}
\xymatrix{ &\Om(\Tt_0)_{dg}&\\ 
\Om(\mathbb{P}^1\times \Z)_{dg} \ar[ru]^{\pi_N} &  &  \Om(\mathbb{P}^1\times \Z)_{dg}\ar[lu]_{\pi_N} \\ 
& Pt_\infty\ar[lu]^{i_\infty} \ar[ru]_{i_0} & }
\end{equation}
Thus, we have an induced map 
\begin{equation}\label{eq:hocolimtot0}
hocolim\big(Pt_\infty\rightrightarrows \Om(\mathbb{P}^1\times \Z)_{dg}\big) \to \Om(\Tt_0)_{dg}
\end{equation}
One can define the homotopy coequalizer above as $\Om(\mathbb{P}^1\times \Z)_{dg} \coprod_{(Pt_\infty\coprod Pt_\infty)}Pt_\infty $, i.e. by gluing $\Om(\mathbb{P}^1\times \Z)_{dg}$ and $Pt_\infty$ along $i_0\coprod i_\infty:(Pt_\infty\coprod Pt_\infty)\to \Om(\mathbb{P}^1\times \Z)_{dg}$ and $id\coprod id$ (note it is not the same as colimit of (\ref{eq:acommdiag}), see Figure \ref{figure:catgluingcover} for a schematic picture). See \cite{oldGPS2better} for the definition of homotopy push-outs.
\begin{figure}\centering
	\includegraphics[height= 4 cm]{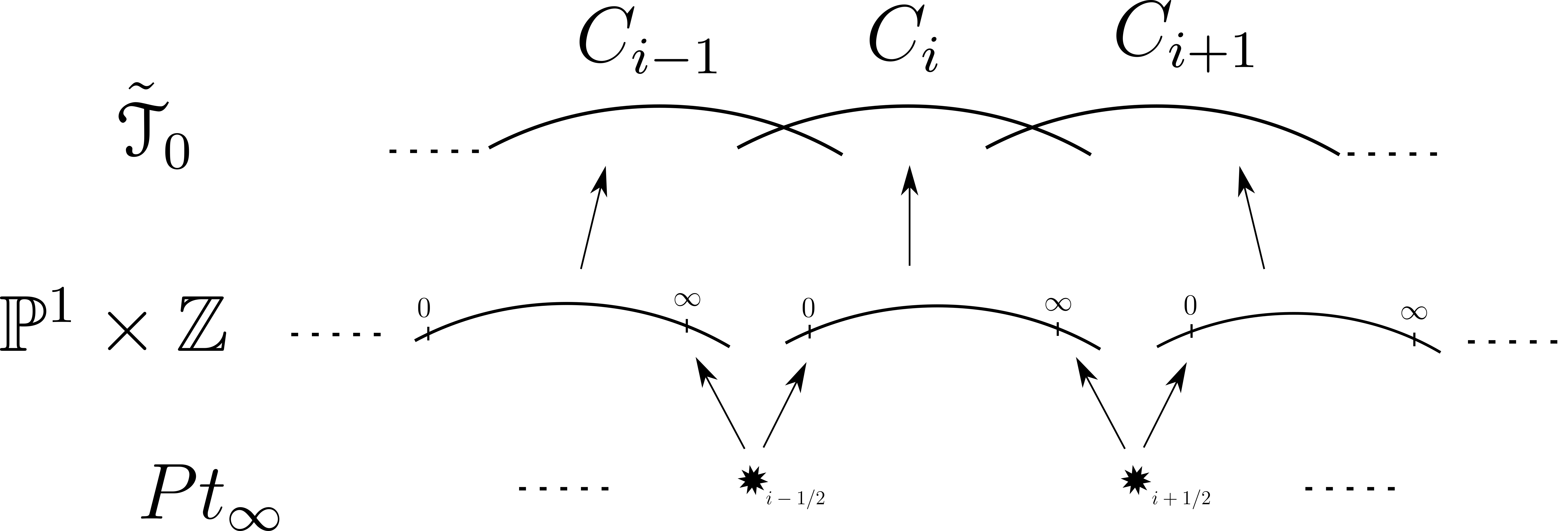}
	\caption{A schematic picture of diagram (\ref{eq:acommdiag})}
	\label{figure:catgluingcover}
\end{figure}

For convenience, let us spell out a description of this coequalizer via Grothendieck construction, following \cite{oldGPS2better} (more precisely, we give an equivalent, slightly modified version that works for coequalizer diagrams, this version is equivalent to one given in \cite{thomasongroth}). Consider the category $\mathcal{G}r$ with objects \begin{equation}\label{eq:obgr1}
ob(\mathcal{G}r)=ob(\Om(\mathbb{P}^1\times \Z)_{dg})\coprod ob(Pt_\infty)
\end{equation} 
We define the morphisms to be 
\begin{equation}\label{eq:homgr1}
hom(X,X'):=\begin{cases}
\Om(\mathbb{P}^1\times \Z)_{dg}(X,X'),&\text{if }X,X'\in \Om(\mathbb{P}^1\times \Z)_{dg}\\ Pt_\infty (X,X'),&\text{if }X,X'\in Pt_\infty\\ \Om(\mathbb{P}^1\times \Z)_{dg}(i_0(X),X') \oplus\\ \Om(\mathbb{P}^1\times \Z)_{dg}(i_\infty(X),X'),& \text{if }X\in Pt_\infty ,X'\in \Om(\mathbb{P}^1\times \Z)_{dg}  \\ 0,& \text{if }X\in \Om(\mathbb{P}^1\times \Z)_{dg},X'\in Pt_\infty 
\end{cases}
\end{equation}
In other words, $\mathcal{G}r$ is a category that contains $Pt_\infty$ and $\Om(\mathbb P^1\times\Z)_{dg}$ as full subcategories, and contains additional morphisms corresponding to maps $i_0(X)\to X'$ and $i_\infty(X)\to X'$. In particular, if we let $X'$ be $i_0(X)$, resp. $i_\infty(X)$, then $\mathcal{G}r(X,X')$ contains morphisms corresponding to identity. Denote the family of these morphisms by $C$. The homotopy coequalizer can be defined as \begin{equation}\label{eq:hocolimbilbil}
hocolim\big(Pt_\infty\rightrightarrows \Om(\mathbb{P}^1\times \Z)_{dg}\big):=C^{-1}\cG r
\end{equation}
For the definition of localization, see \cite{GPS1} or proof of Lemma \ref{lem:descendingextra}. 

Now we prove:
\begin{lem}
(\ref{eq:hocolimtot0}) is a quasi-equivalence.
\end{lem}
\begin{proof}
One way to see this is by direct computation: namely, write the Grothendieck construction for the diagram, then check that the induced functor from the localization is a quasi-equivalence. This is a cohomology level check, in the sense that one can localize after taking the cohomology. The localization of the cohomological category of $\cG r$ has an explicit description in terms of sequences of morphisms and their formal inverses, and it is not hard to check in this case that the induced functor into $H^0(\Om(\Tt_0)_{dg})$ is an equivalence.

Another option is to use the fact ``$Coh$ sends colimits to colimits''. As it was explained to us by Vivek Shende, this follows from \cite[Theorem A.1.2]{gaitsgoryrozenblyumbook2}.
See also \cite[Corollary 2.5]{nadlerwrapped} for the explanation on how the statement follows from \cite[Theorem A.1.2]{gaitsgoryrozenblyumbook2}.

More precisely, \cite[Theorem A.1.2]{gaitsgoryrozenblyumbook2} states that the contravariant functor $X\mapsto IndCoh(X)$, $f\mapsto f^!$ restricted to the category of affine, Noetherian schemes with closed embeddings sends push-outs to pull-back squares (in a category of dg categories). Unfortunately, this does not immediately apply to our situation; however, we can use it easily. 

First, assume the statement that $X\mapsto IndCoh(X)$, $f\mapsto f_*$ sends push-outs to push-outs for Noetherian, projective schemes with closed embeddings holds. Then the same holds with $IndCoh$ replaced by $Coh$. Let $Ev(n)$ denote the full subcategory of $\Om(\mathbb P^1\times \Z)_{dg}$ spanned by sheaves on even indexed curves $C_{-2n},C_{-2n+2},\dots,C_{2n}$. In other words,
\begin{equation}
Ev(n)=\bigsqcup_{i= -2n,-2n+2,\dots,2n} \Om(\mathbb P^1)_{dg}
\end{equation}
Let $Ev$ denote the union of all $Ev(n)$. Similarly let $Od(n)$ denote the subcategory corresponding to curves with odd index $-2n+1,-2n+3,\dots ,2n-1$ and $Od$ denote their union. Let $Pt_n$ denote the subcategory of $Pt_\infty$ spanned by points with index $-2n+1/2,-2n+3/2,\dots ,2n-1/2$ (i.e. the points of intersection of curves in $Ev(n)$ and $Od(n)$). Finally let $\Tt_0(n)$ denote the reduced subvariety of $\Tt_0$ given by the union of $C_{-2n},C_{-2n+1},\dots, C_{2n}$. Clearly, $\Tt_0(n)$ is a push-out of curves involved in $Ev(n)$ and $Od(n)$ along $Pt_n$. As we assume that ``push-outs to push-outs'' holds for Noetherian projective schemes with closed embeddings, we have 
\begin{equation}\label{eq:finiteequivalence}
Od(n)\sqcup_{Pt_n}Ev(n)\simeq \scrC oh(\Tt_0(n))
\end{equation}
where $\scrC oh(\Tt_0(n))$ is a dg-model for $Coh(\Tt_0(n))$. Moreover, the colimit of $\scrC oh(\Tt_0(n))$ gives a dg model for properly supported coherent sheaves on $\Tt_0$, which is derived equivalent to $\Om(\Tt_0)_{dg}$ (the hom's between given two objects stabilize for large $n$, since all objects are properly supported on $\Tt_0$). Hence, 
\begin{equation}
\Om(\Tt_0)_{dg}\simeq hocolim_n \scrC oh(\Tt_0(n))\simeq hocolim_n (Od(n)\sqcup_{Pt_n}Ev(n))\simeq \atop(hocolim_n Od(n))\sqcup_{(hocolim_n Pt_n)} (hocolim_n Ev(n))=Od\sqcup_{Pt_\infty} Ev
\end{equation}
The second equivalence follows from the fact that (\ref{eq:finiteequivalence}) is induced by push-forward maps. Hence, these maps commute up to homotopy with maps $Od(n)\sqcup_{Pt_n}Ev(n)\to Od(n+1)\sqcup_{Pt_{n+1}}Ev(n+1)$ and $\scrC oh(\Tt_0(n))\to \scrC oh(\Tt_0(n+1))$, and they induce an equivalence of colimits. 
The third equivalence is an abstract category theory statement. It is easy to see that $Od\sqcup_{Pt_\infty} Ev$ is equivalent to $hocolim\big(Pt_\infty\rightrightarrows\Om(\mathbb P^1\times\Z)_{dg} \big)$. Hence, the claim follows.

Now, we need to show why ``push-outs to push-outs'' hold for projective Noetherian schemes, at least in our specific case. We only need to show the functor $X\mapsto IndCoh(X)$, $f\mapsto f^{!}$ sends push-out diagrams of projective Noetherian schemes along closed embeddings to pull-back diagrams, as ``push-outs to push-outs'' follows in the same way as \cite[Corollary 2.5]{nadlerwrapped}. The basic idea is to combine the statement for affine push-outs with Zariski descent.

Let $X$ be the union of curves $C_{-2n},C_{-2n+2},\dots,C_{2n}$ and $Y$ be the union of curves $C_{-2n+1},C_{-2n+3},\dots,C_{2n-1}$. Let $Z$ be the union of nodal points in the intersections of $C_{-2n},C_{-2n+1},\dots,C_{2n}$. Then $\Tt_0(n)=X\sqcup_Z Y$. Let $\{U_i\}$ be an affine open cover over $\Tt_0$. Assume $\{U_i\}$ is closed under intersections and the index set is ordered so that $i\leq j$ if and only if $U_i\subset U_j$. Let $U_i^X=U_i\cap X$, $U_i^Y=U_i\cap Y$ and $U_i^Z=U_i\cap Z$. These give affine open covers of $X$, $Y$ and $Z$. By Zariski descent for $IndCoh$ (see \cite{indcoh})
\begin{equation}
IndCoh(\Tt_0(n) )\simeq holim IndCoh(U_i) 
\end{equation}
Indeed, this can be stated by saying that the functor $V\mapsto IndCoh(V), f\mapsto f^{!}$ from the category of schemes with open embeddings sends colimits to limits (note that $f^*=f^!$ for open embeddings). 
Similar descent statement hold for $X,Y$ and $Z$. Moreover, $U_i=U_i^X\sqcup_{U_i^Z} U_i^Y$ and they are affine so push-outs to pull-back hold for them. Thus,
\begin{align*}
IndCoh(\Tt_0(n) )\simeq holim IndCoh(U_i) \simeq holim IndCoh(U_i^X\sqcup_{U_i^Z} U_i^Y)\simeq \\ holim (IndCoh(U_i^X)\times_{IndCoh(U_i^Z)} IndCoh(U_i^Y))\simeq\\ holim (IndCoh(U_i^X))\times_{holim(IndCoh(U_i^Z))} holim(IndCoh(U_i^Y))\simeq \\
IndCoh(X)\times_{IndCoh(Z)} IndCoh(Y)
\end{align*}
This completes the proof.
%
\end{proof}
Hence, (\ref{eq:hocolimbilbil}) generates another enhancement for $D^b(Coh_p(\Tt_0) )$ and (\ref{eq:hocolimtot0}) is $\Z$-equivariant. Taking smash products with respect to $\Z$-action (see \cite[Section 4]{ownpaperalg}), we obtain 
\begin{cor}
There is a quasi-equivalence $(C^{-1}\cG r)\#\Z\to\Om(\T_0)_{dg}$ that is compatible with extra gradings. 
\end{cor}
Notice that $(C^{-1}\cG r)\#\Z$ does not depend on the choice of enhancement made for $D^b(Coh (\mathbb P^1))$ by \cite{orlovlunts}.
\begin{note}\label{note:zigzagbylocal}
Lemma \ref{lem:zigzagt0} also follows from these considerations. Namely, given any other model $(\Om',\tr')$ for $\Om(\Tt_0)$ with a strict auto-equivalence $\tr'$ lifting $\tr_*$, we can choose a dg functor similar to $\Xi_0$ and define $\Xi_i$ by composing with $\tr^i$. Assume the chosen model $\Om(\mathbb P^1)_{dg}$ is minimal. We then obtain a diagram similar to (\ref{eq:acommdiag}), corresponding Grothendieck construction and a functor to $\Om'$. This is strictly compatible with translation. Hence, there is a quasi-equivalence from the explicit localization of the Grothendieck construction to $\Om'$ that is strictly compatible with translation. The localization of Grothendieck construction does not depend on $\Om'$; hence, it gives us a zigzag as promised in Lemma \ref{lem:zigzagt0}.
\end{note}

By $\Z$-equivariance, one can also realize $(C^{-1}\cG r)\#\Z$ as a localization of $\cG r\#\Z$. This localization carries an extra grading by Lemma \ref{lem:descendingextra} and it is quasi-equivalent to $(C^{-1}\cG r)\#\Z$ (hence to $\Om(\Tt_0)_{dg}\#\Z=\Om(\T_0)_{dg}$) since the localization map $\cG r\#\Z\to (C^{-1}\cG r)\#\Z$ is compatible with extra grading.

Notice $\cG r\#\Z$ is equivalent to Grothendieck construction for 
\begin{equation}\label{eq:prediag1}
Pt_\infty\#\Z\rightrightarrows \Om(\mathbb P^1\times\Z)\#\Z
\end{equation}
which is defined similar to (\ref{eq:homgr1}) by replacing hom-sets with the hom-sets of smash product. Recall, $Pt_\infty$ consists of infinitely many pairwise orthogonal objects denoted by $*_{i+1/2}$ and the $\Z$-action on $Pt_\infty$ sends $*_{i-1/2}$ to $*_{i+1/2}$. Since the action is transitive, all objects become isomorphic in $Pt_\infty\#\Z$. The morphisms from $*_{i+1/2}$ to $*_{j+1/2}$ can be identified with $\bigoplus_{n\in \Z} hom(*_{i+n+1/2},*_{j+1/2})=\C$ (it is non-zero only at $n=j-i$). In other words, $Pt_\infty\#\Z\simeq \C$. Similarly, $\Om(\mathbb P^1\times \Z)\#\Z\simeq \Om(\mathbb P^1)_{dg}$. Using these identifications, one can see that the Grothendieck construction for (\ref{eq:prediag1}) is equivalent to a dg category with objects $ob(\Om(\mathbb{P}^1)_{dg})\coprod \{*\}$ and with morphisms
\begin{equation}\label{eq:homsforsmallgrothalg}
hom(X,X'):=\begin{cases}
\Om(\mathbb{P}^1)_{dg}(X,X'),&\text{if }X,X'\in \Om(\mathbb{P}^1)_{dg}\\ \C,&\text{if }X=X'\in\{*\}\\ \Om(\mathbb{P}^1)_{dg}(i_0(X),X') \oplus\\ \Om(\mathbb{P}^1)_{dg}(i_\infty(X),X'),& \text{if }X\in \{*\} ,X'\in \Om(\mathbb{P}^1)_{dg}  \\ 0,& \text{if }X\in \Om(\mathbb{P}^1)_{dg},X'\in \{*\} 
\end{cases}
\end{equation} 
In other words, this is a dg category consisting of (some) coherent sheaves on $\mathbb P^1$, an extra object $*$ and morphisms $*\to \scrF$ corresponding to morphisms $\Om_0\to\scrF$ and $\Om_\infty\to\scrF$. This category is the Grothendieck construction for the diagram $\C\rightrightarrows \Om(\mathbb P^1)_{dg}$, which is defined similar to $\cG r$. Let us denote this category by $\cG r\#\Z$ as well.

$\Om(\T_0)_{dg}$ is obtained by localizing $\cG r\#\Z$ at two morphisms from $*$ corresponding to identity maps of $\Om_0$ and $\Om_\infty$. Denote these morphisms by $c_0$ and $c_\infty$. This process geometrically corresponds to identifying $0$ and $\infty$ on $\mathbb{P}^1$. 

The corresponding extra grading is given by setting the summand \begin{equation}\Om(\mathbb{P}^1)_{dg}(i_\infty(X),X')\subset hom(X,X')
\end{equation} to be the degree $1$-part and the remaining expressions in (\ref{eq:homsforsmallgrothalg}) to be of degree $0$. The extra grading descends to localization $\Om(\T_0)_{dg}$ (see Lemma \ref{lem:descendingextra}) and it clearly matches the extra grading coming from smash product with respect to $\Z$-action. 
\begin{rk}
This extra grading on (\ref{eq:homsforsmallgrothalg}) comes from the identification with $\cG r\#\Z$ via $*\mapsto x_{-1/2}$ and $\mathbb{P}^1\mapsto \mathbb{P}^1\times\{0\}\subset\mathbb{P}^1\times \Z$. In a different identification, one can set elements of $\Om(\mathbb{P}^1)_{dg}(i_0(X),X')$ to be of degree $-1$ and the rest to be of degree $0$. The descriptions become equivalent after localization, and we will go with the former.
\end{rk}
\begin{rk}
In Appendix \ref{sec:appendixgluing}, we give a description of $M_\phi$ as a homotopy coequalizer of $\A\rightrightarrows\Om(\mathbb{P}^1)_{dg}\otimes \A$, where the arrows are given by $i_0\otimes1_\A$ and $i_\infty\otimes\phi$. Assume $\A$ is a dg enhancement for $D^b (Coh(M^\vee))$, where $M^\vee$ is a projective variety and $\phi$ is push-forward along an automorphism $\phi_{M^\vee}$. One can construct an algebraic space 
\begin{equation}\label{eq:agmappingtorus}
\mathbb{P}^1\times M^\vee/(0,x)\sim (\infty,\phi_{M^\vee}(x))
\end{equation}
that is isomorphic to the one given in \cite[Example 1.1]{ownpaperalg}. Then push-out preserving property of $Coh$ combined with a K\"unneth theorem for $Coh$ (similar to\cite[Prop 4.6.2]{indcoh}) proves that $M_\phi$ is derived equivalent to coherent sheaves on this algebraic space. Note that we have not checked the details as this is not our main interest. If $\cW(M)\simeq D^b(Coh(M^\vee))$ (i.e. $M$ and $M^\vee$ are homological mirrors), combining this statement with Theorem \ref{mainthmsymp} implies a homological mirror symmetry statement for $T_\phi$ and the algebraic space (\ref{eq:agmappingtorus}). 
\end{rk}
\subsection{Extra grading on $\cW(T_0)$}\label{subsec:extragr}
By \cite{Weinsteinsector} and \cite[Theorem 1.10]{GPS2}, $\cW(T_0)$ is generated by Lagrangians that lift under the covering map $\tilde T_0\to T_0$. We will consider only these Lagrangians as objects of $\cW(T_0)$, and we fix a lift for each object of $\cW(T_0)$. For a Lagrangian $L\subset T_0$, denote the lift by $\tilde L$. Given $r\in \Z$, let $\tilde L\langle r\rangle$ denote another lift of $L$ obtained by shifting $\tilde L$ by $r$ in the positive direction. 

According to the definition of \cite{generation}, the chain complexes $CW(L_1,L_0)$ are generated by Hamiltonian chords from $L_1$ to $L_0$ for a fixed Floer datum. This chord lifts to a path from $\tilde L_1$ to $\tilde L_0\langle -r\rangle$, for a unique $r$. We define the extra grading by letting this chord be of degree $r$. This should not be confused with the original grading of $\cW(T_0)$. See Figure \ref{figure:liftingchordsincovering} for a lift $x$ of a chord of degree $1$ and a lift $y$ of a chord of degree $-1$ from $L_0$ to $L_0$.
\begin{figure}\centering
	\includegraphics[height=3 cm]{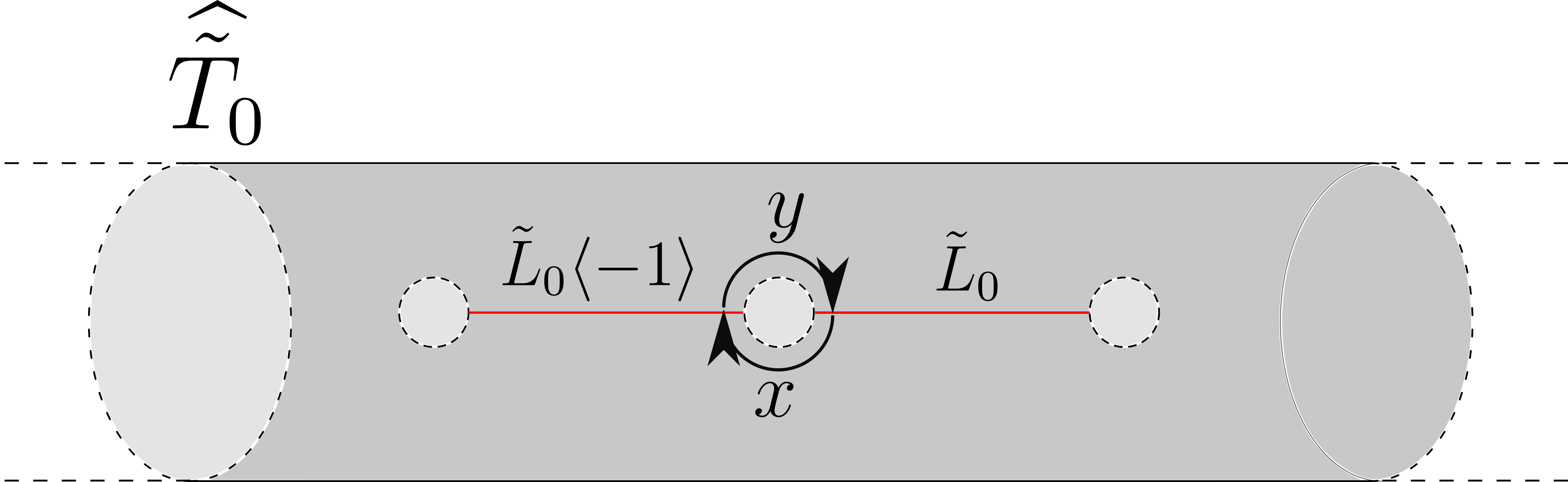}
	\caption{A lift $x$ of a degree $1$ chord and a lift $y$ of a degree $-1$ chord from $L_0$ to $L_0$}
	\label{figure:liftingchordsincovering}
\end{figure}

We now wish to compare the categories $\Om(\T_0)_{dg}$ and $\cW(T_0)$, while taking their extra gradings into account. Recall:
\begin{lem}\cite[Lemma 9.9]{ownpaperalg} $\Om(\T_0)_{dg}$ generates a dg model for $D^b(Coh(\T_0))$.
\end{lem}
\begin{thm}\label{thm:hmst0}\cite[Theorem B.(ii)]{lekpol} $\cW(T_0)$ is derived equivalent to $Coh(\T_0)$. Indeed, $tw(\cW(T_0))$ is an $A_\infty$-enhancement of $D^b(Coh(\T_0))$, and one can choose the equivalence such that $L_{gr}$ and $L_{pur}$ (green and purple curves in Figure \ref{figure:handletorus}) correspond to $\widetilde{\Om}_{\T_0}$ and $\Om_x$ respectively. Here, $\widetilde{\Om}_{\T_0}$ is the push-forward of the structure sheaf under normalization map and $\Om_x$ is the structure sheaf of the singular point.
\end{thm}
Choose suitable generators for $\cW(T_0)$ ($L_{gr}$, i.e. the green curve in Figure \ref{figure:handletorus}, and the diagonal curve not shown in the picture). One can rephrase \cite[Theorem B.(ii)]{lekpol} as:
\begin{cor}\label{cor:hmst0}
$\Om(\T_0)_{dg}$ and $\cW(T_0)$ are quasi-equivalent $A_\infty$-categories. Moreover, under this quasi-equivalence $\widetilde \Om_{\T_0}$ corresponds to $L_{gr}$, $\widetilde \Om_{\T_0}(-1)$ corresponds to the curve that wraps around the torus once (which would be the diagonal curve in Figure \ref{figure:handletorus}), and $\Om_x$--- the structure sheaf of the node, as an object of $tw(\Om(\T_0)_{dg})$--- corresponds to $L_{pur}$. 
\end{cor}
However, we need the comparison of categories $\Om(\T_0)_{dg}$ and $\cW(T_0)$ as $A_\infty$-categories with extra grading. In other words, we need to relate these categories by zigzags of quasi-equivalences which respect extra gradings.
\begin{rk}
It is well known that the notions of an $A_\infty$-category and a dg category are equivalent over fields of characteristic $0$. More precisely, any dg category is an $A_\infty$-category with no higher operations and for any $A_\infty$-category, one can construct a dg category and $A_\infty$-quasi-equivalences between the two. Moreover, the $A_\infty$-morphisms from a dg category to another are equivalent to zigzags of dg quasi-equivalences followed by a dg morphism to the target. We specify whether we are considering a dg category in the category of dg categories or $A_\infty$-categories. 
\end{rk}
\begin{rk}
By choosing homotopy transfer data to be of degree $0$, one can construct minimal models with extra grading. Moreover, the constructed quasi-equivalences and homotopies all respect the extra grading as well. See \cite{markltransfer} for more about transferring $A_\infty$-structures. It is easy to see that such a minimal model is unique up to a gauge equivalence that preserves the extra grading. Hence, one can equivalently define the notion of equivalence for extra graded $A_\infty$-categories as the graded quasi-equivalence of their minimal models.
\end{rk}
One can check by hand that the extra gradings on $\Om(\T_0)_{dg}$ and $\cW(T_0)$ match at a cohomological level (up to a minor modification of the quasi-equivalence between them, for instance using symmetries of $\cW(T_0)$). However, this does not directly imply that they are equivalent, for the same reason that two gauge equivalent minimal $A_\infty$-structures on a extra graded vector space (i.e. a doubly graded vector space, note that $A_\infty$-maps are of degree $0$ in the second grading) are not necessarily equivalent via a gauge equivalence that respects the grading. 

Nevertheless, one can prove:
\begin{lem}\label{lem:extragrdlemma}
$\Om(\T_0)_{dg}$ and $\cW(T_0)$ are quasi-equivalent as  $A_\infty$-categories with an extra grading. 
\end{lem}
To prove Lemma \ref{lem:extragrdlemma}, we will reprove  Theorem \ref{thm:hmst0} and Corollary \ref{cor:hmst0} using the gluing formula of \cite{GPS2}. In other words, we will give a description of $\cW(T_0)$ as a homotopy coequalizer similar to description of $\Om(\T_0)_{dg}$ in Section \ref{subsec:mttwi}. $\cW(T_0)$ will be described as a localization of an intermediate category $\cG r_s$ with extra grading. To obtain extra grading on $\cW(T_0)$, we need:
\begin{lem}\label{lem:descendingextra}
Let $\B$ be a category with extra grading and $C$ be a set of homogeneous morphisms. Then the category $C^{-1}\B$ can be endowed with an extra grading such that the localization map respects gradings of $\B$ and $C^{-1}\B$. Moreover, if $\B_1$ and $\B_2$ are quasi-equivalent as extra graded categories, and $C_1$ and $C_2$ correspond to each other in cohomology under the (zigzag of) quasi-equivalence(s), then $C_1^{-1}\B_1$ is quasi-equivalent to $C_2^{-1}\B_2$ as an extra graded category.
\end{lem}
\begin{proof}
First, let us recall the definition of localization following \cite{GPS1}: consider the set of cones $\B_C$ of elements of $C$. Take the Lyubashenko-Ovsienko/Drinfeld quotient of $\B$ by $\B_C$ (see \cite{lyubaquot}, \cite{partiallywrapped}, \cite{GPS1}). In general, this specific model allows one to endow the quotient $\B/\B'$ with an extra grading when $\B$ is an extra graded category and $\B'\subset \B$ is a full subcategory. In our case, one has to enlarge $\B$ within $tw(\B)$ by adding cones of $C_1$, resp. $C_2$. However, the extra grading also extends to this larger subcategory as the morphisms of $C$ are all homogeneous. 
Hence, the first claim follows. 
	
For the second claim assume without loss of generality that there is a quasi-equivalence $\B_1\to \B_2$ that respects the extra gradings and that carries $C_1$ to $C_2$ strictly (for instance assume $\B_2$ is minimal). Then, enlarge both categories by adding cones. The quasi-equivalence extends to a (graded) quasi-equivalence of enlarged categories as well (that sends cones of $C_1$ to cones of $C_2$). The natural functor from the localization with respect to cones of $C_1$ to the localization with respect to cones of $C_2$ preserves the extra gradings.
Hence, there is an induced quasi-equivalence $C_1^{-1}\B_1\to C_2^{-1}\B_2$ that preserves the gradings. 
\end{proof}
To apply gluing formula, we need to decompose $T_0$ into Liouville sectors. Recall:
\begin{defn}(\cite[Definition 1.1]{GPS1}, \cite[Definition 1.1]{GPS2}) A Liouville sector $X$ is a Liouville manifold with boundary such that there exists a function $I:\partial X\to\mathbb{R}$ that is linear at infinity and whose Hamiltonian vector field is pointing outward along $\partial X$. 	
\end{defn}
The examples of Liouville sectors include Liouville manifolds, $T^*N$, where $N$ is a manifold with boundary, and open Riemann surfaces with boundary such that the boundary has no closed component. For this section, we are only concerned about the last type of examples. We refer to $\partial X$ as the finite boundary of $X$ (as opposed to boundary at infinity defined analogous to contact boundary at infinity of Liouville manifolds). 
\begin{figure}\centering
\includegraphics[height=4 cm]{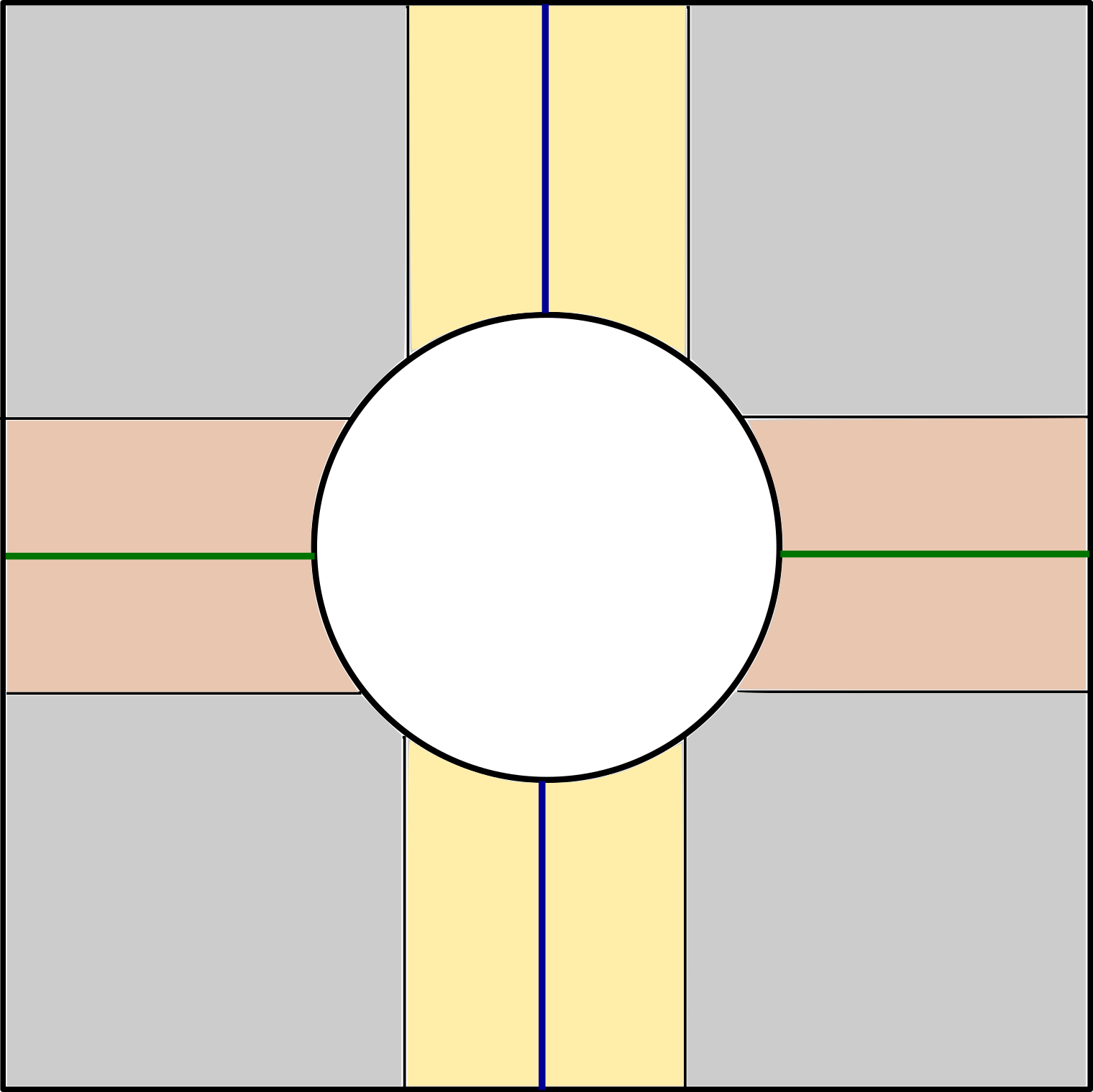}
\caption{Handlebody decomposition of $T_0$}
\label{figure:handletorus2}
\end{figure}

Decompose $T_0$ into sectors as follows: let $T$ denote the vertical $1$-handle in Figure \ref{figure:handletorus2} that is shown in yellow. Cut $T_0$ into sectors $T$ and $N$ along the boundary of $T$, where $N=\overline{T_0\setminus T}$. The finite boundary of these sectors correspond to side edges of $T$. See Figure \ref{figure:handletorus2} or Figure \ref{figure:sectorgluing} for a clearer picture (we are being sloppy about the notation as $T_0$ previously referred to the Liouville domain rather than its completion, similarly with $T$). 

As a Liouville sector, $T$ is equivalent to $T^*[0,1]$. Note that the notion of a Liouville sector is equivalent to the notion of a Liouville manifold with stops (see \cite{partiallywrapped}). Under this equivalence, $T$ corresponds to a disc with $2$-stops at its boundary, whereas $N=\overline{T\setminus T_0}$ corresponds to a cylinder with one stop at each of its boundary components. 

In \cite{GPS1}, the authors associate an $A_\infty$-category $\cW(X)$ to a Liouville sector $X$, called the wrapped Fukaya category. We will later recall the definition (see Section \ref{subsec:reminderwfuk}); however,
for our current purposes, we take this notion as a black box. The following are easy to prove:
\begin{exmp}\label{exmp:wt=c}(\cite[Example 1.22]{GPS2})
$\cW(T)=\cW(T^*[0,1])$ is derived equivalent to $\mathbb{C}$. It is generated by $L_{pur}$, the purple curve in Figure \ref{figure:handletorus2}.
\end{exmp}
\begin{figure}\centering
	\includegraphics[height= 4 cm]{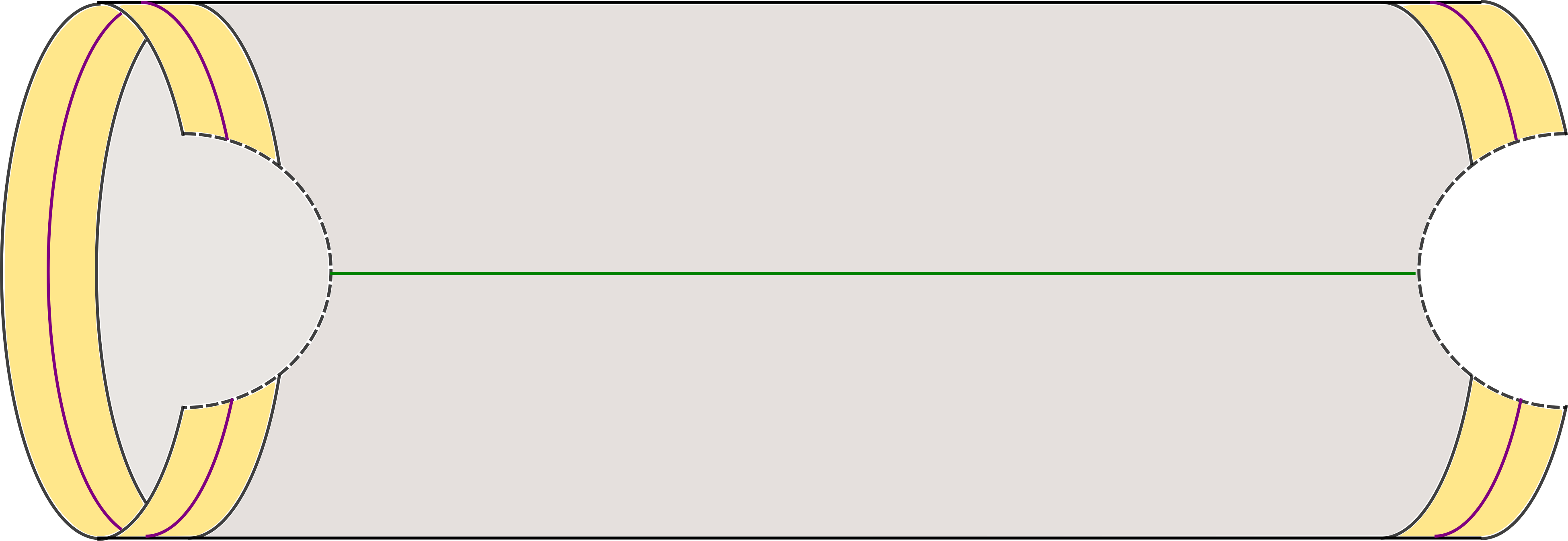}
	\caption{The sector $N=\overline{T_0\setminus T}$ and two inclusions of $T$ into $N$}
	\label{figure:cuttori}
\end{figure}
\begin{exmp}\label{exmp:wn=p1}(\cite[Example 1.18]{GPS2})
$\cW(N)$ is derived equivalent to $D^b(Coh(\mathbb{ P}^1) )$ (the partially wrapped category computed in \cite[Example 1.18]{GPS2} is equivalent to $\cW(N)$). As generators one can take the green curve in Figure \ref{figure:cuttori} and another curve that wraps around once without intersecting the green curve (the diagonal of Figure \ref{figure:handletorus2}). The subcategory spanned by these curves is equivalent to Kronecker quiver. For an alternative set of generators one can take the horizontal green curve and the vertical purple curves in Figure \ref{figure:cuttori}. Indeed, under the equivalence with $D^b(Coh(\mathbb{ P}^1) )$, the green curve corresponds to a line bundle and the purple curves correspond to skyscraper sheaves, and together they generate the category.
As a side note, this category is equivalent to Fukaya-Seidel category of Landau-Ginzburg model $(\C,z+z^{-1})$, which is well known to be a mirror to $\mathbb P^1$. 
\end{exmp}
\begin{figure}\centering
	\includegraphics[height= 2 cm]{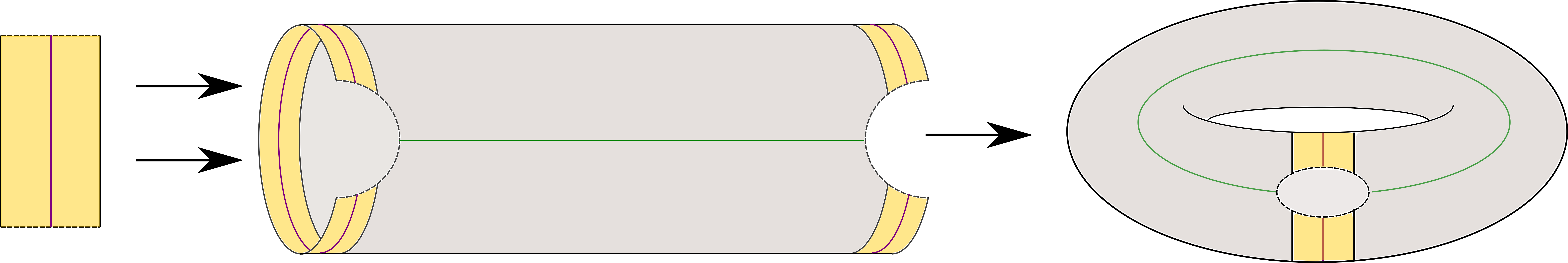}
	\caption{The inclusions of sector $T$ into $N=\overline{T_0\setminus T}$ and the inclusion of $N$ into $T_0$}
	\label{figure:sectorgluing}
\end{figure}
$T$ and $N=\overline{T_0\setminus T}$ are glued along the sector given by a neighborhood of their shared edges (represented by two yellow sectors in Figure \ref{figure:cuttori}). This sector is isomorphic to $T^*[0,1]\coprod T^*[0,1]$. Then, by the gluing formula \cite[Theorem 1.20]{GPS2}, we have a pushout diagram:
\begin{equation}\label{eq:gluediagsymp}
\xymatrix{ \cW( \overline{T_0\setminus T})=\cW(N)\ar[r]& \cW(T_0)\\ \cW(T^*[0,1])\coprod \cW(T^*[0,1])\ar[u]\ar[r]& \cW(T^*[0,1])\simeq \cW(T)\ar[u]  }
\end{equation}
Even though the inclusion of two yellow sectors into the 1-handle $T$ is not an isomorphism, it induces an equivalence between their wrapped Fukaya categories. In other words, the lower horizontal arrow in (\ref{eq:gluediagsymp}) can be seen as the identity on each component. Hence, we can write this gluing diagram as 
\begin{equation}\label{eq:coeqmap}
\C\simeq \cW(T^*[0,1])=\cW(T)\rightrightarrows \cW( N )\to \cW(T_0)
\end{equation}
where $\cW(T_0)$ is equivalent to homotopy coequalizer of
\begin{equation}\label{eq:coeq}
\cW(T)\rightrightarrows \cW( N )
\end{equation}
A pictural representation of (\ref{eq:coeqmap}) is given by Figure \ref{figure:sectorgluing} (Figure \ref{figure:sectorgluing} can also be seen as a coequalizer diagram; however, from the perspective of \cite{GPS1}, \cite{GPS2} this picture is slightly informal, as the maps of sectors in Figure \ref{figure:sectorgluing} are not global inclusions, but rather like ``\'etale maps'' for sectors). 

Abusing notation, let $j_0$, resp. $j_1$, be the left inclusion, resp. right inclusion, of both $T\rightrightarrows N=\overline{T_0\setminus T}$ and $\cW(T)\rightrightarrows\cW(N)$. To make statement about homotopy coequalizer precise, consider the category $\cG r_s$  with objects 
\begin{equation}\label{eq:obgr2}
ob(\cW(T))\coprod ob(\cW(N))
\end{equation}
and with morphisms 
\begin{equation}\label{eq:homgr2}
hom(X,X'):=\begin{cases}
\cW(N)(X,X'), &\text{if }X,X'\in \cW(N)\\
\cW(T)(X,X')=\C, &\text{if }X=X'\in \cW(T)\\
\cW(N)(j_0 (X),X' )\oplus \cW(N)(j_1 (X),X' ), &\text{if } X\in \cW(T), X'\in\cW(N)\\
0,&\text{if } X\in\cW(N), X'\in\cW(T)
\end{cases}
\end{equation}
As before, $\cG r_s$ can be seen as the Grothendieck construction for (\ref{eq:coeq}). The rightmost arrow in (\ref{eq:coeqmap}) induces a functor 
\begin{equation}\label{eq:locloc}
\cG r_s\to \cW(T_0)
\end{equation}
That (\ref{eq:coeqmap}) is a homotopy coequalizer diagram means that (\ref{eq:locloc}) is a localization at two morphisms $L_{pur}\to j_0(L_{pur})$ and $L_{pur}\to j_1(L_{pur})$ corresponding to identities of $j_0(L_{pur})$ and $j_1(L_{pur})$ (we are abusing the notation and denoting the purple curves in copies of $T$ in Figure \ref{figure:sectorgluing} by $L_{pur}$ as well). Denote the set of these two morphisms by $C_s$. 

Grade $\cG r_s$ as we graded (\ref{eq:homsforsmallgrothalg}): let morphisms of type $\cW(N)(j_0(X),X' )$ be of degree $1$ and the remaining components be of degree $0$. 
This extra grading descends to localization $C_s^{-1}\cG r_s$ by Lemma \ref{lem:descendingextra}. To see the extra grading obtained by localizing $\cG r_s$ matches the previously given one on $\cW(T_0)$ at the beginning of this section, one only has to show the map (\ref{eq:locloc}) respects the extra grading. For instance, let $L$, resp. $L'$ denote the image of $j_1(X)$, resp. $X'\subset N$ in $T_0$. One can arrange the lifts $\tilde L$, resp. $\tilde{ L'}$, of $L$, resp. $L'$, such that the image of a chord from $j_1(X)$ to $X'$ lifts to a chord from $\tilde L$ to $\tilde{L'}$ (hence, degree $0$ in the previously given extra grading). This implies that the image of any chord from $j_0(X)$ to $X'$ lifts to a chord from $\tilde L$ to $\tilde{L'}\langle -1\rangle$ (hence, degree $1$). The other cases are easier. 
%
%

To reprove Theorem \ref{thm:hmst0}, observe that $\cG r\#\Z$ (i.e. the Grothendieck construction for (\ref{eq:prediag1}) with hom-complexes given by (\ref{eq:homsforsmallgrothalg})) is equivalent to $\cG r_s$ (defined by (\ref{eq:obgr2}) and (\ref{eq:homgr2})) and the sets of morphisms $C$ and $C_s$ correspond to each other under the equivalence (in cohomology). Hence, the homotopy colimits $C^{-1}(\cG r\#\Z)$ and $C_s^{-1}\cG r_s$ are equivalent. Moreover, the equivalence respects the extra gradings on $\cG r\#\Z$ and $\cG r_s$; hence, the induced equivalence also respects the extra grading by Lemma \ref{lem:descendingextra}. This proves Lemma \ref{lem:extragrdlemma}. 
\begin{note}
Homological mirror symmetry for $T_0$ and $\T_0$ with extra $\Z/n$-grading can be seen as homological mirror symmetry for their $n$-fold covers. Similarly, equivalence with extra $\Z$-grading can informally be thought as mirror symmetry between $\Tt_0$ and $\tilde{T}_0$.	
\end{note}
\section{K\"unneth and twisted K\"unneth theorems}\label{sec:kunneth}
\subsection{Introduction}
In this section, we define an $A_\infty$-functor \begin{equation}\label{eq:twfunct}
\cW(T_\phi)\rightarrow Bimod_{tw}(\cW(T_0),\cW(M))
\end{equation}
and show that it is full and faithful. We also show that the essential image is spanned by twisted Yoneda bimodules. This finishes the proof of:
\begingroup
\def\thethm{\ref{mainthmsymp}}
\begin{thm}
$\cW(T_\phi)$ is quasi-equivalent to twisted tensor product of $\cW(T_0)$ and $\cW(M)$.
\end{thm}
\addtocounter{thm}{-1}
\endgroup
For simplicity, we will first expose the reader to the main idea on how to use quilted strips to define (\ref{eq:twfunct}) and to the basic TQFT argument

Then, we will explain how to do the same for wrapped Fukaya categories. We start by defining a category $\cW^2(T_\phi)$ that is analogous to category $\cW^2$ in \cite{sheelthesis} and $\cW^{prod}$ in \cite{GPS2}. It is equivalent to $\cW(T_\phi)$ by an argument similar to \cite{GPS2}.

By Corollary \ref{torusgenerators}, $\cW(T_\phi)$ is split generated by Lagrangians of type $L\times_\phi L'$. Hence, we will restrict attention to only these objects, we will prove their images are quasi-isomorphic to Yoneda bimodules, and that (\ref{eq:twfunct}) is fully faithful on these objects. 
\subsection{Quilted strips}
Moduli of $n$-quilted strips are defined in \cite{quiltedstrip}, and they control $A_\infty$ n-modules. Their main strata can be identified with $n$-parallel lines with markings in $\C$ with fixed distance from each other (up to conformal equivalence). For $n=3$, this is used in \cite{sheelthesis} to define functors from a version of wrapped Fukaya category on $M\times M^-$ to bimodules over $\cW(M)$. In general, defining a functor 
\begin{equation}
\cW(X\times Y)\to Bimod(\cW(X),\cW(Y))
\end{equation}
is equivalent to defining a left-right-right $\cW(X\times Y)$-$\cW(X)$-$\cW(Y)$-trimodule. We would like to exploit similar ideas to define (\ref{eq:twfunct}) (see Remark \ref{rk:definetrimodule}). Let us start by describing moduli of quilted strips first:
\begin{defn}
	Let $\textbf d=(d_1,d_2,d_3)\in\Z_{\geq 0}^3$. A $3$-quilted strip with $\textbf d$-markings is 
	\begin{itemize}
		\item a pair of strips $r_1,r_2$ biholomorphic to $\mathbb R\times[0,1]$
		\item $d_1$-markings on the upper boundary of $r_1$, and $d_2$-markings on the upper boundary of $r_2$
		\item $d_3$-markings on the lower boundary of $r_1$ and $r_2$
		\item an identification of $r_1$ and $r_2$ preserving the incoming/outgoing ends of the strip and mapping lower markings of $r_1$ to lower markings of $r_2$
	\end{itemize}
\end{defn}
The isomorphisms of such quilted strips are given by isomorphisms of both strips commuting with the identification (i.e. by simultaneous isomorphisms of $r_1$ and $r_2$).
\begin{defn}
	Let $\cQ(\textbf d)=\cQ(d_1,d_2,d_3)$ denote the moduli space of $3$-quilted strips up to isomorphism.
\end{defn}
\begin{figure}\centering
	\includegraphics[height= 3 cm]{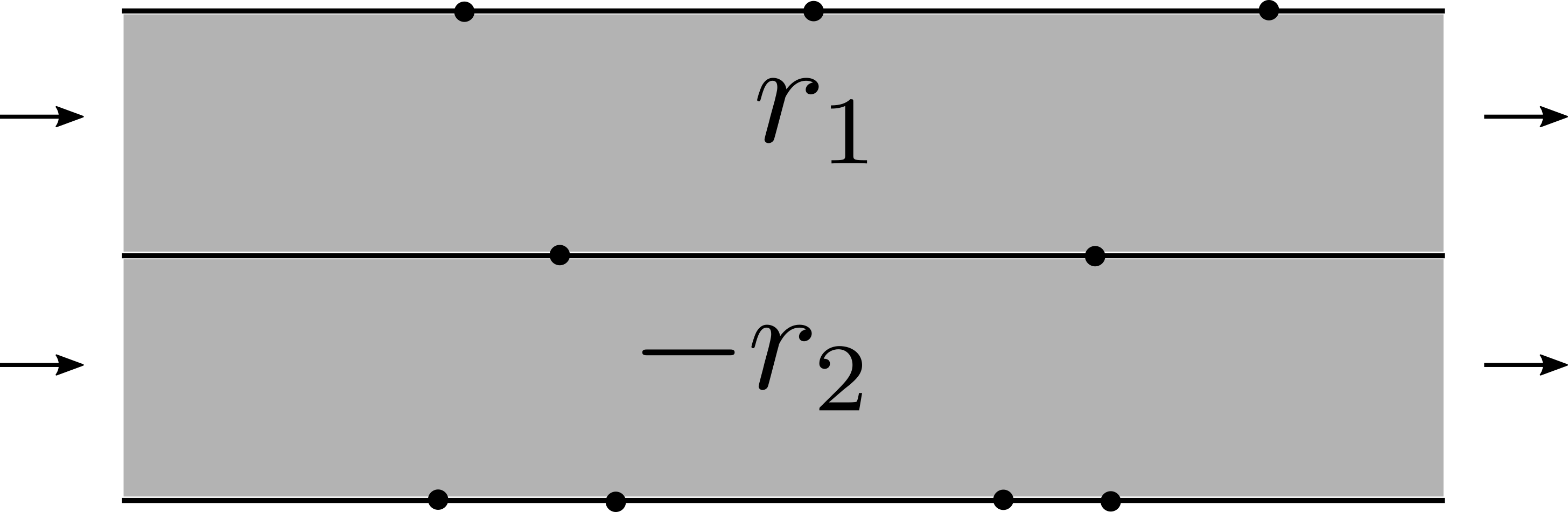}
	\caption{An element of $\cQ(3,4,2)$}
	\label{figure:3quiltedreg}
\end{figure}
The identification of $r_1$ and $r_2$ is uniquely determined up to translation. When $d_3>0$, it is uniquely determined. When $d_3=0$, different identifications give different elements. It is not hard to identify $\cQ(d_1,d_2,d_3)$ with the space of $3$-quilted lines in \cite{quiltedstrip}. We will indeed picture these objects as in Figure \ref{figure:3quiltedreg} which is similar to \cite{sheelthesis} and \cite{GPS2}. In this figure, $-r_2$ is the strip $r_2$ with conjugate holomorphic structure. This quilted surface can be folded to obtain a single strip, thanks to global identification of $r_1$ and $r_2$. The complement of the markings in the quilted strip $r=(r_1,r_2)\in \cQ(\textbf d)$ will be denoted by $\cS^q_r$. The complement of markings in the folding of the strip  will be denoted by $\cS^f_r$ (the superscript $q$ stands for quilted and $f$ stands for folded). The family of these surfaces form universal bundles over $\cQ(\textbf d)$, denoted by $\cS^q$ and $\cS^f$ respectively. The complement of the markings in $r_1$ and $r_2$ will be denoted by $\cS^{(1)}_r$ and $\cS^{(2)}_r$ respectively (hence, $\cS^f_r=\cS^{(1)}_r\cap\cS^{(2)}_r$).

$\cQ(\textbf d)$ admits a natural compactification described in detail in \cite{quiltedstrip}. We denote this compactification by $\overline{\cQ(\textbf{d})}$.  In Figure \ref{figure:comp3quilt}, we give example of a typical boundary element of  $\overline{\cQ(\textbf{d})}$. 
\begin{figure}\centering
	\includegraphics[height=4 cm]{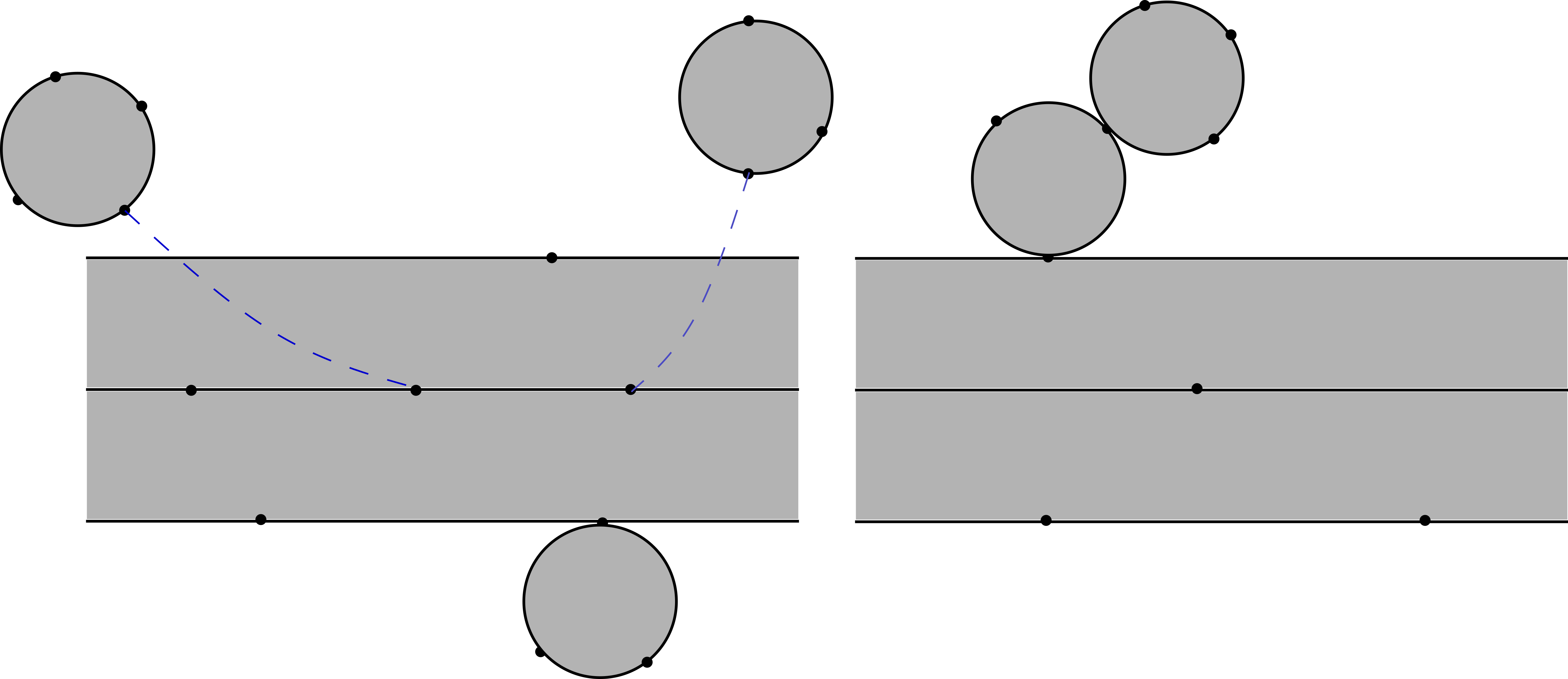}
	\caption{An element of $\partial \overline{\cQ(5,5,7)}$}
	\label{figure:comp3quilt}
\end{figure}
In summary, if we have a sequence of elements of $\cQ(\textbf{d})$ such that a set of points on the upper/lower boundary or on the seam collide, the limit has a disc bubbled off that component, or seam. If the horizontal distance between two sets of boundary marked points goes to infinity, this causes strip breaking. Figure \ref{figure:comp3quilt} illustrates both situations. It represents an element of the codimension $6$ strata of $\partial \overline{\cQ(5,5,7)}$, one for each bubbling and one for disc breaking. 
\begin{note}
To describe the boundary structure, one has to choose strip-like ends for each boundary marking as well as the incoming and outgoing ends of the quilted strip. These are conformal maps from the half infinite strip $Z_\pm=\mathbb{R}_\pm \times[0,1]$ to marked discs/components of the strip, and they allow one to  describe how to glue boundary elements for given gluing parameters. In our case, the global identification of $r_1$ and $r_2$ (i.e. folding) has to be kept as well. If we allow arbitrary strip-like ends, the identification would require Riemann mapping theorem, and the bottom markings of $r_1$ and $r_2$ (i.e. the marked points on the seam) would no longer match. Therefore, we have to restrict strip-like ends as follows: we use rational strip-like ends on the disc bubbles, i.e. they extend to biholomorphic maps from the whole strip $Z=\mathbb{R}\times[0,1]$ to the disc. If this is a strip-like end for an incoming marked point on a disc bubble, we also ask the biholomorphic extension to $Z$ to converge at the other infinite end to outgoing marked point of the disc. For the ends on the quilted component, we choose standard ends on the left and right. To choose ends for the markings on the upper boundary of $r_1$ and $r_2$ (i.e. lower boundary of $-r_2$), we consider the natural embedding of the strip into upper half-plane so that upper boundary maps to real axis, and that does not change the strip width. This embedding is determined up to translation, and we require the incoming strip-like ends corresponding to upper markings of $r_1$ and $r_2$ to be as before, i.e. rational, and such that the extended biholomorphic map from $Z=\mathbb R\times [0,1]$ converges to $\infty$ (of the ambient half-plane). The ends for the middle markings are similar. 

For a pictural explanation of gluing, consider elements of boundary strata described as in Figure \ref{figure:3quiltedwithhalf}.
Gluing a hyperplane to the strip is essentially taking a large half-disc in that hyperplane, taking out a small half-disc from the edge of the strip, and gluing the large half-disc after rescaling. 
\end{note}
\begin{figure}\centering
\includegraphics[height=4 cm]{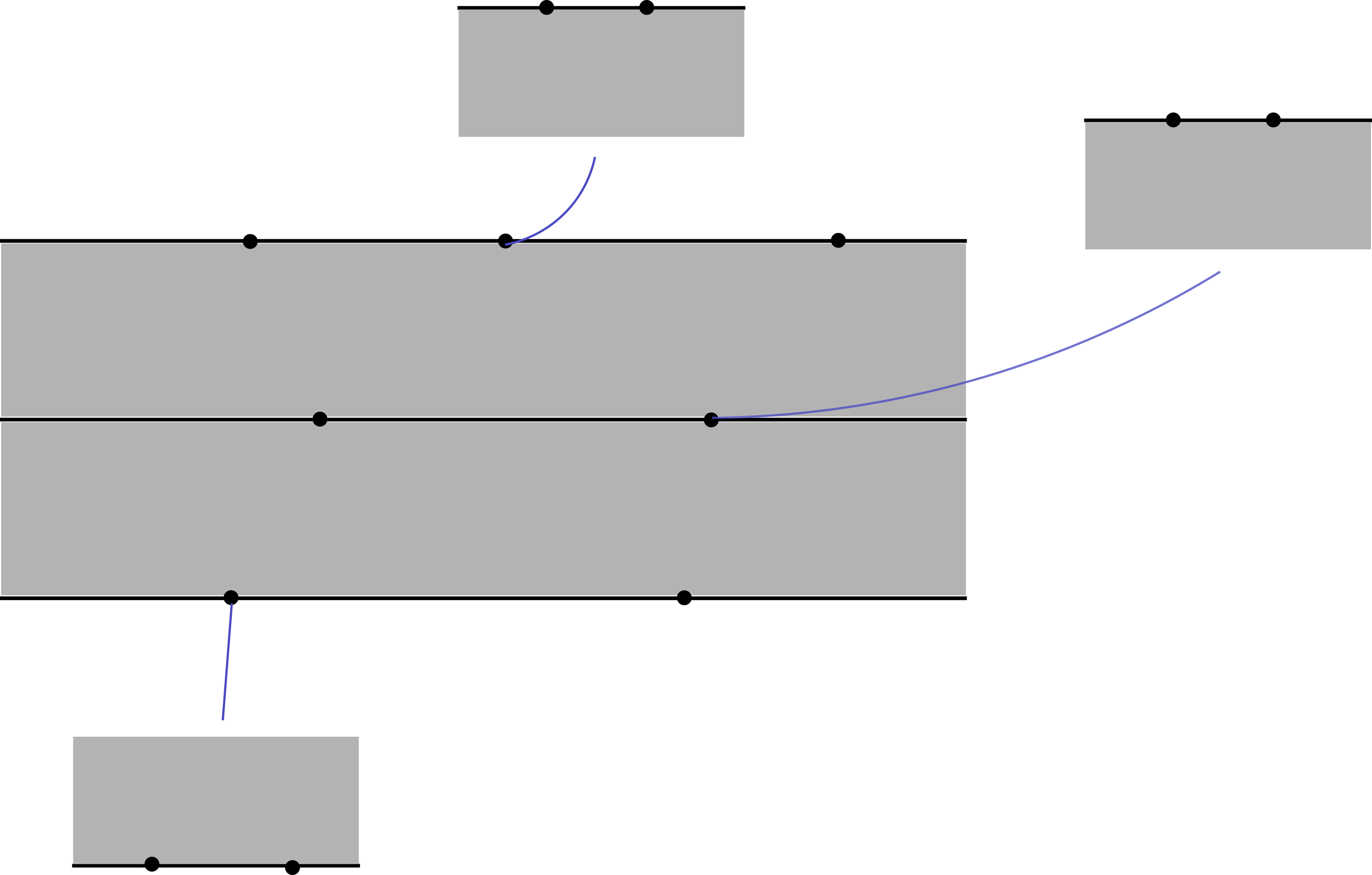}
\caption{A more convenient description of boundary elements of $\cQ(\textbf d)$ for the purposes of gluing}
\label{figure:3quiltedwithhalf}
\end{figure}
Fix a consistent choice of strip-like ends for all moduli spaces $\cQ(\textbf{d})$. In other words, we fix strip-like ends on each $\cS_r^{(1)}$ and $\cS_r^{(2)}$, and the ends chosen for the left/right end and lower boundary (i.e. the seam) coincide for $\cS_r^{(1)}$ and $\cS_r^{(2)}$.
\subsection{Definition of K\"unneth functor and proof of its fully faithfulness via count of quilts}\label{subsec:defnkunnethearly}
In this section, we show how to use counts of quilted strips to define a functor \begin{equation}\label{eq:twfunct2}
\cW(T_\phi)\to Bimod_{tw}(\cW(T_0),\cW(M))
\end{equation}
Indeed, for simplicity, we first define it for $\phi=1_M$ (i.e. for $T_\phi=T_0\times M$), and later explain the twisted case. Then, we will show how to prove fully faithfulness by a TQFT argument. As this section is intended to be expository, we use the definition of wrapped Fukaya category via quadratic Hamiltonians \`a la \cite{generation}, and we ignore all technical problems like rescaling and compactness issues. In the following sections, we will switch to definition given in \cite{GPS1}, show necessary compactness results, and necessary modifications in definition and proof of fully faithfulness of (\ref{eq:twfunct}). See \cite{kartalthesis} for the same argument worked with the definition of \cite{generation}.

Such a functor in the case $\phi=1_M$ and $T_\phi=T_0\times M$ is defined in \cite{sheelthesis} for a version of wrapped Fukaya category of $T_0\times M$. Namely, define $\cW^2(T_0\times M)$ to be a category whose objects are given by $L\times L'$, where $L$ and $L'$ are exact, cylindrical Lagrangian branes, i.e. exact, cylindrical Lagrangians equipped with grading and spin structures, in $\wh{T_0}$ and $\wh{M}$, respectively. Cylindrical means the corresponding Lagrangian is invariant under the Liouville flow outside a compact subset. Define an $A_\infty$-structure using Floer data that decompose into Floer data on $\wh{\tilde T_0}$ and $\wh{M}$ (i.e. split type Floer data), whose components are cylindrical and quadratic at infinity. Defining an $A_\infty$-functor
\begin{equation}
\cW^2(T_0\times M)\to Bimod(\cW(T_0),\cW(M) )
\end{equation}
is equivalent to defining a left-right-right trimodule over $\cW^2(T_0\times M)$-$\cW(T_0)$-$\cW(M)$. Given exact Lagrangian branes $L\subset \wh{T_0}$, $L'\subset \wh{M}$ and $L''\subset \wh{T_0}\times \wh{M}$, define $\fM(L'',L,L')$ as the linear span of Hamiltonian chords $L\times L'\to L''$. To make it a trimodule, define structure maps 
\begin{align*}
CW(L_{p-1}'',L_p'')\otimes CW(L_{p-2}'',L_{p-1}'')\otimes \dots \otimes CW(L_0'',L_1'')\otimes  \fM(L_0'',L_0,L_0')\otimes \\ CW(L_1,L_0)\otimes \dots CW(L_m,L_{m-1})\otimes CW(L_1',L_0')\dots \otimes CW(L_n',L_{n-1}')\\ \to \fM(L_p'',L_m,L_n')[1-m-n-p]
\end{align*} by counting pseudo-holomorphic quilted strips as in Figure \ref{figure:labelledquilt} with target $\wh{T_0}\times \wh{M}$ (i.e. upper component maps to $\wh{T_0}$ and bottom component maps to $\wh{M}$).
\begin{figure}
\centering
\includegraphics[height=3cm]{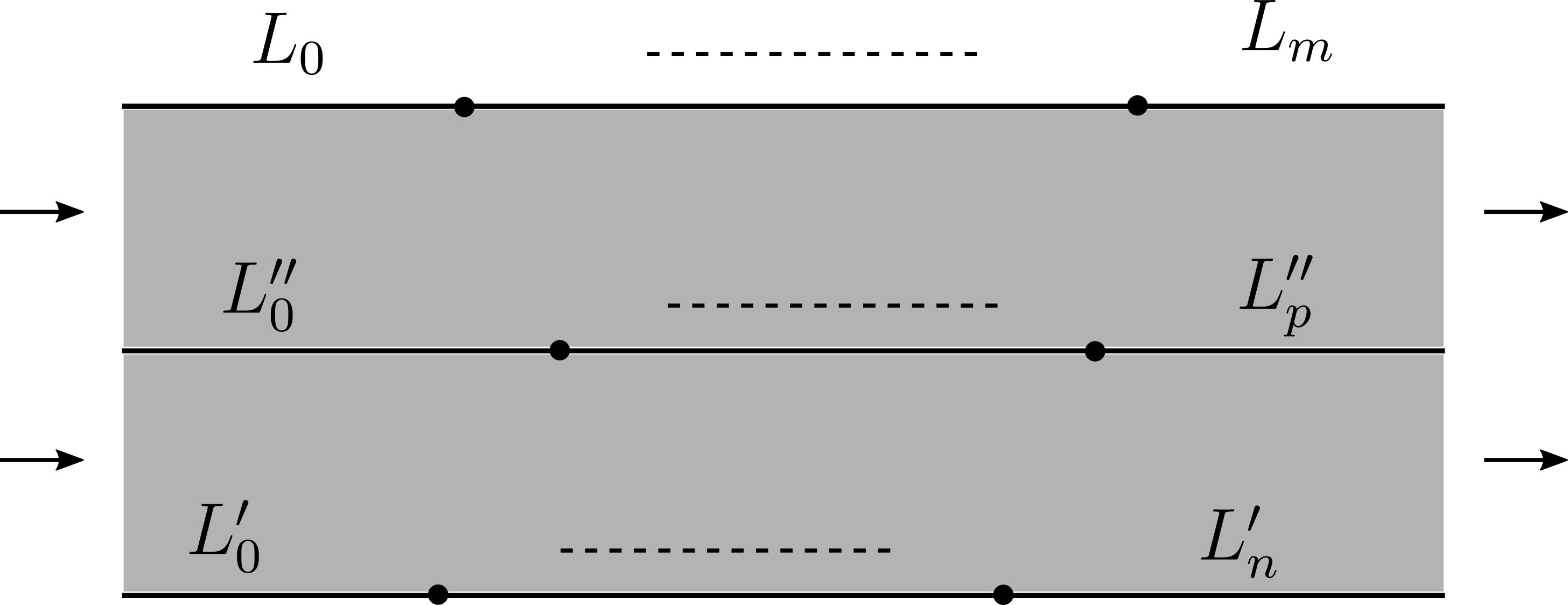}
\caption{The quilts defining bimodule/functor structures on $\fM$}
\label{figure:labelledquilt}
\end{figure}
For this count, one chooses Floer data for upper and lower strips, $\cS_r^{(1)}$ and $\cS_r^{(2)}$ separately. 

Similarly, defining a functor (\ref{eq:twfunct2}) is equivalent to defining a twisted left-right-right trimodule over $\cW(T_\phi)$-$\cW(T_0)$-$\cW(M)$, where twisting is between $\cW(T_0)$ and $\cW(M)$. One can first define a version of $\cW(T_\phi)$ analogous to $\cW^2(T_0\times M)$. Namely, as we will remark later, even though $T_\phi$ is not a product itself, its conical end $\wh{(T_0\times M)}\setminus (T_0\times M)$ can be identified with the conical end of $\wh{T_0}\times \wh{M}$ (since $\phi$ is compactly supported). Hence, define $\cW^2(T_\phi)$ to be a category with objects $L\times_\phi L'$, where $L\subset T_0,L'\subset M$ are exact Lagrangian branes (see Definition \ref{defn:twlagr}). Define the $A_\infty$-structure using Floer data that is split at the conical end. 

Note that, we assume $\cW(T_0)$ consists of Lagrangians with fixed lifts under $\wh{\tilde T_0}\to\wh{T_0}$ (so that $\cW(T_0)$ has an extra grading as in Section \ref{subsec:extragr}). On the other hand, the symplectomorphism $\phi$ does not a priori induce a strict auto-equivalence of the Fukaya category. To make it strict, one needs to add quasi-isomorphic objects $[L',n]$ for each Lagrangian brane $L'\subset M$ and $n\in\Z$ (and let $\phi$ act by $[L',n]\mapsto [\phi(L'),n+1]$). This allows one to choose Floer data invariant under $\phi$. We will abuse the notation and keep denoting the objects of $\cW(M)$ by letters such as $L'$ and their images under the induced strict autoequivalence by $\phi(L')$.

\begin{figure}
	\centering
	\includegraphics[height=3cm]{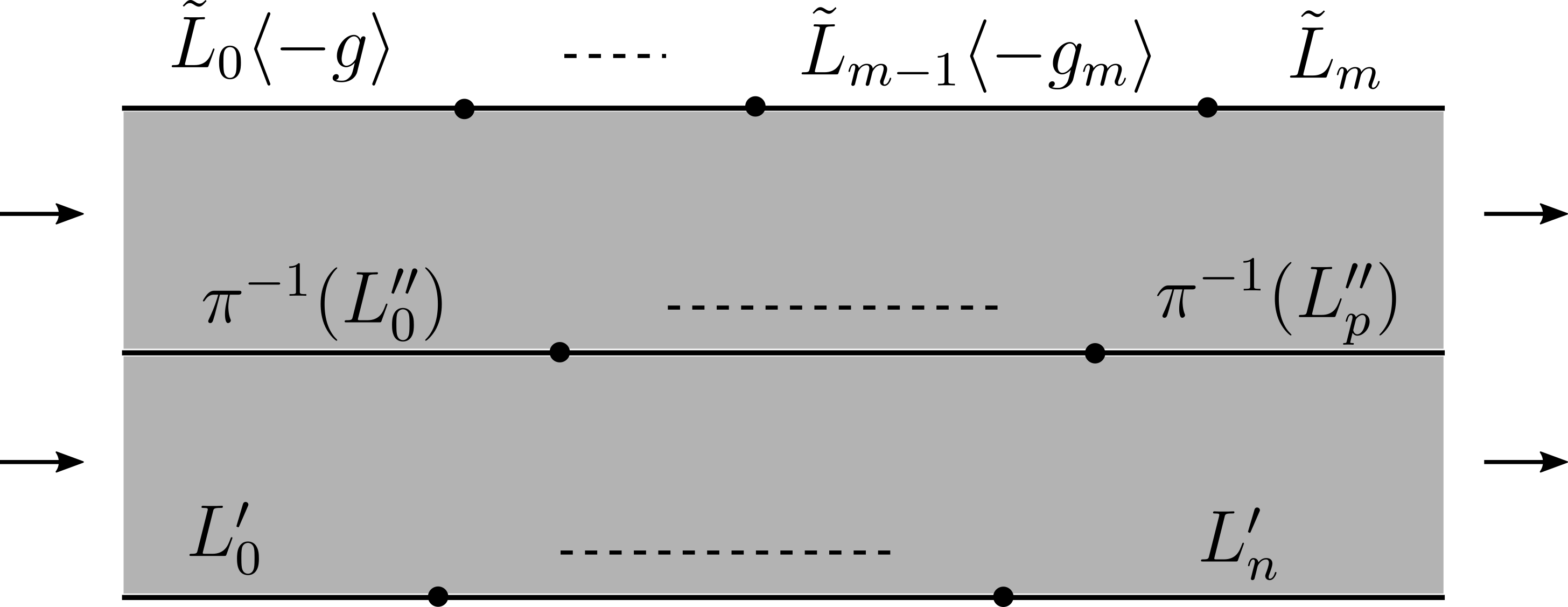}
	\caption{Labeling for quilted strips defining $\fM$ in twisted case}
	\label{figure:labelledtwisted}
\end{figure}
\begin{figure}
	\centering
	\includegraphics[height=3cm]{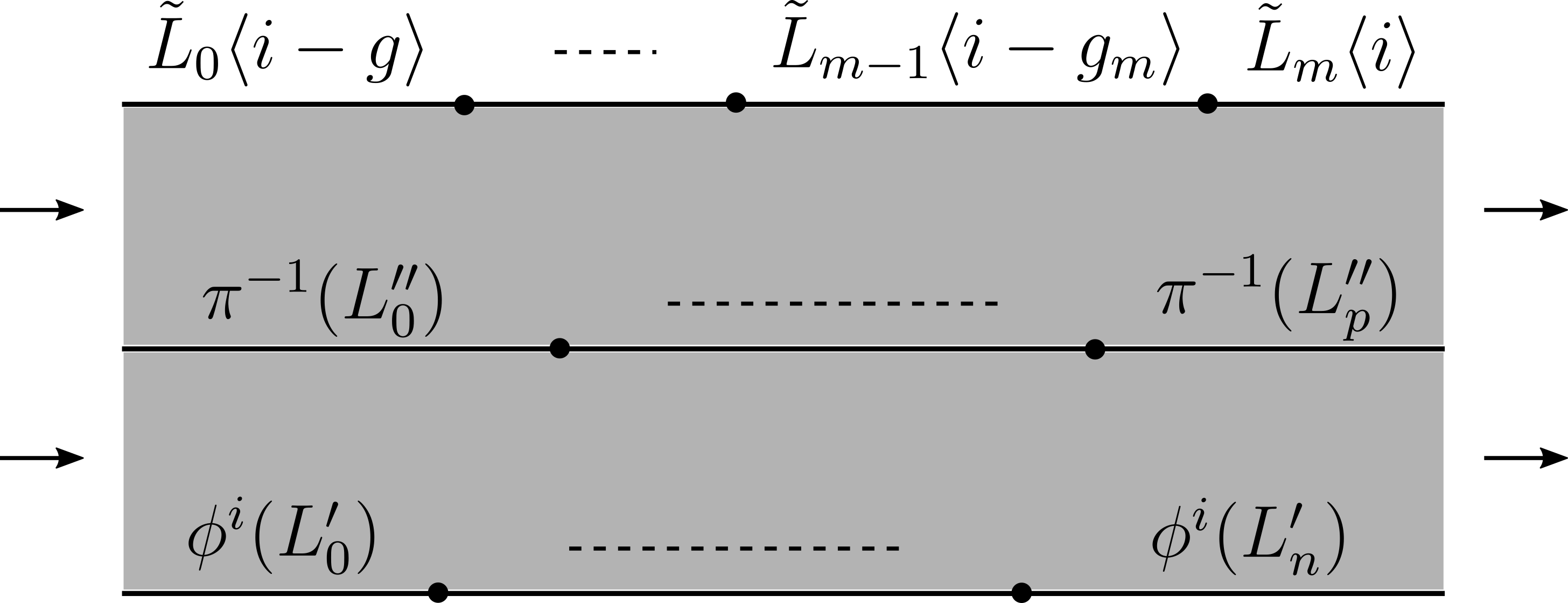}
	\caption{Equivalent labeling for quilted strips defining $\fM$}
	\label{figure:labelledtwisted2}
\end{figure}
To define the $\cW^2(T_\phi)$-$\cW(T_0)$-$\cW(M)$-trimodule $\fM$, we would like to count quilted strips as in Figure \ref{figure:labelledtwisted} mapping to $\wh{\tilde T_0}\times \wh{M}$. However, to ensure $A_\infty$-equations are satisfied, we would need the Floer data to come from $\wh{T_\phi}$ via pull-back along $\pi:\wh{\tilde T_0}\times \wh{M}\to \wh{T_\phi}$ near the marked points on the seam, and to be split type near the markings on upper and lower boundary components. Hence, instead of choosing Floer data on $\cS_r^{(1)}$ and $\cS_r^{(2)}$ separately, we choose data parametrized by $\cS_r^{f}$ on the symplectic manifold $\wh{\tilde T_0}\times \wh{M}$. Following \cite{GPS1}, let $N_\epsilon$ denote the set of points that have distance less than $\epsilon$ to at least one of the marked points on the lower boundary of $\cS_r^{f}$ (i.e. the markings on the seam). We ask the perturbation data on the folded strip $\cS_r^f$ to satisfy:
\begin{enumerate}
\item\label{fdexp:trphi} it is $\tr\times\phi$-invariant on $N_{1/3}$ (i.e. it is pull-back of a perturbation datum on $\wh{T_\phi}$)
\item\label{fdexp:prd} it is split type on $\cS_r^f\setminus N_{2/3}$ (i.e. it decomposes into perturbation data on $\wh{\tilde T_0}$ and $\wh{M}$)
\item\label{fdexp:prdfirstcomp} $\wh{\tilde T_0}$-component extends to a perturbation datum on $\cS_r^{(1)}\setminus N_{2/3}$ and it is $\tr$-invariant (i.e. it is pull-back of a perturbation datum on $\wh{T_0}$)
\item\label{fdexp:prdMcomponent} $\wh{M}$-component extends to a perturbation datum on $\cS_r^{(2)}\setminus N_{2/3}$
\end{enumerate}
Choose such data consistently such that 
\begin{condition}\label{condition:equivdata}
The data for labelings as in Figure \ref{figure:labelledtwisted} and Figure \ref{figure:labelledtwisted2} are related by $(\tr\times\phi)^i$.	
\end{condition}
Since $\phi$ is acting freely on objects of $\cW(M)$, this condition is unproblematic. The purpose of Condition \ref{condition:equivdata} is to ensure the counts with labeling as in Figure \ref{figure:labelledtwisted} and Figure \ref{figure:labelledtwisted2} are equivalent. In particular, chords from $\tilde L\langle i-g\rangle\times \phi^i(L')$ to $\pi^{-1}(L'')$ can be identified with chords from $\tilde L\langle -g\rangle\times L'$ to $\pi^{-1}(L'')$. Condition \ref{condition:equivdata} is needed not only for bimodule equation but also to obtain structure maps (\ref{eq:trimodule}). For example, the $\fM$-input in (\ref{eq:trimodule}) is a chord from $\tilde L_0\times\phi^g(L_0')$ to $\pi^{-1}(L_0'')$, but it identifies with  a chord from $\tilde L_0\langle-g\rangle\times L_0'$ to $\pi^{-1}(L_0'')$.

As usual, the data for strips with no markings (i.e. the data defining Floer differential) are assumed to be translation invariant. We also assume the choices are made so that 
\begin{condition}\label{condition:youneedish}
In the absence of markings on the middle seam, the data (which are necessarily split-type) depend only on $\cS_r^{(1)}$ and $\cS_r^{(2)}$.
\end{condition}
Recall that in this case, $\cS_r$ still depends on the boundary identification of upper and lower strips (hence, such strips with fixed upper and lower strips are parametrized by $\mathbb{R}$). This assumption is needed to show that images of objects under (\ref{eq:twfunct2}) are twisted Yoneda bimodules.

Define the $\cW^2(T_\phi)$-$\cW(T_0)$-$\cW(M)$-trimodule $\fM$ (or equivalently the functor (\ref{eq:twfunct2})) as follows: given objects $L\subset \wh{T_0}$, $L'\subset \wh M$ and $L''\subset \wh{T_\phi}$, let $\fM(L'',L,L')$ be the $\C$-linear span of Hamiltonian chords from $\tilde L\times L'$ to $\pi^{-1}(L'')$. Define the structure maps 
\begin{align}\label{eq:trimodule}
\begin{split}
CW(L_{p-1}'',L_p'')\otimes CW(L_{p-2}'',L_{p-1}'')\otimes \dots \otimes CW(L_0'',L_1'')\otimes  \fM(L_0'',L_0,\phi^g(L_0'))\otimes \\ CW(L_1,L_0)^{g_1}\otimes \dots CW(L_m,L_{m-1})^{g_m}\otimes CW(L_1',L_0')\dots \otimes CW(L_n',L_{n-1}')\\ \to \fM(L_p'',L_m,L_n')[1-m-n-p]
\end{split}
\end{align}
by counting pseudo-holomorphic quilted strips as in Figure \ref{figure:labelledtwisted} (as before $g=\sum g_i$ and $CW(L_1,L_0)^{g_1}$ denotes the degree $g_1$ part in extra grading). More precisely, count rigid pseudo-holomorphic maps $\cS_r^f\to \wh{\tilde T_0}\times \wh{M}$ with respect to chosen Floer data with given asymptotic conditions. Notice, the asymptotic conditions put on the upper boundary markings are the natural lifts of chords in $\wh{T_0}$, and the conditions on the lower boundary markings are Hamiltonian chords in $\wh{M}$. On the other hand, the asymptotic conditions we put on the markings on the seam are given by chords from $L''_i$ to $L''_{i+1}$ in $\wh{T_\phi}$. These chords have infinitely many lifts to chords from $\pi^{-1}(L''_i)$ to $\pi^{-1}(L''_{i+1})$; however, once the labeling and other asymptotic conditions are fixed, the lift is uniquely determined (this also uses the fact that we only consider $L''$ that lifts under $\pi$). For instance, assume that the labeling is as in Figure \ref{figure:labelledtwisted}. If we know the chord from $\tilde L_0\langle -g\rangle\times L_0'$ to $\pi^{-1}(L_0'')$, this determines the component of $\pi^{-1}(L_0'')$ to which that part of the seam maps. Hence, this determines the lift of the chord from $L_0''$ to $L_1''$, determining the component of $\pi^{-1}(L_1'')$ that the seam maps to, and so on. 

To illustrate the correspondence between Figure \ref{figure:labelledtwisted} and (\ref{eq:trimodule}) further, consider the case $m=n=p=1$. Let $x\in CW(L_1,L_0)^g$, $x'\in CW(L_1',L_0')$, $x''\in CW (L_0'',L_1'')$ and $y\in \fM(L_0'',L_0, \phi^g(L_0') )$ be some generators, where $g=\sum g_i=g_1$. By definition, $y$ is a chord from $\tilde L_0\times \phi^g(L_0')$ to $\pi^{-1}(L_0'')$ and by Condition \ref{condition:equivdata}, it can be identified with a chord from $\tilde L_0\langle -g\rangle \times L_0'$ to $\pi^{-1}(L_0'')$. We use this as the trimodule input of Figure \ref{figure:labelledtwisted}. As above, the endpoint of $y$ determines a component of $\pi^{-1}(L_0'')$, and there is a unique lift of $x''$ to a chord from this component of $\pi^{-1}(L_0'')$ to $\pi^{-1}(L_1'')$, which is to be used as the input on the seam. By definition, $x$ lifts to a chord from $\tilde L_1\to \tilde L_0\langle -g\rangle$, we use this lift and $x'$ as the other inputs in Figure \ref{figure:labelledtwisted} and the output is a chord from $\tilde L_1\times L_1'$ to $\pi^{-1}(L_1'')$, which is a generator of $\fM(L_1'',L_1,L_1')$. Therefore, the count of pseudo-holomorphic quilted strips gives us a map as in (\ref{eq:trimodule}).

By condition (\ref{fdexp:trphi}), the stable discs that bubble off the seam are lifts of pseudo-holomorphic discs on $\wh{T_\phi}$. Similarly, due to conditions (\ref{fdexp:prd}) and (\ref{fdexp:prdfirstcomp}) the discs on the upper boundary come from pseudo-holomorphic discs on $\wh{T_0}$, and due to conditions (\ref{fdexp:prd}) and (\ref{fdexp:prdMcomponent}) the discs on the lower boundary component come from pseudo-holomorphic discs on $\wh{M}$. 
Moreover, by Condition \ref{condition:equivdata}, this count is equivalent to count in Figure \ref{figure:labelledtwisted2}. Using these, one can easily show that twisted trimodule equations hold by standard gluing arguments if we assume compactness. Moreover, the generators of $\fM(L'',L,L')$ are graded as usual, and the moduli spaces defining $\fM$ are oriented (see \cite{sheelthesis},\cite{GPS2}). Hence, the trimodule $\fM$ is defined over $\C$, is $\Z$-graded, and defines a functor (\ref{eq:twfunct2}) of $\Z$-graded categories over $\C$. We denote the image of $L''\subset \wh{T_\phi}$ under this functor by $\fM(L'')$ and $\fM_{L''}$ as well. It is easy to show the following:
\begin{lem}\label{lem:mequalsyoneda}
Let $L''=L\times_\phi L'$, where $L\subset \wh{T_0}, L'\subset\wh{M}$ are graded Lagrangian branes (where $L$ is endowed with a fixed lift $\tilde L$ as usual). Then, $\fM_{L''}$ is isomorphic to twisted Yoneda bimodule $h_L\otimes_{tw} h_{L'}$. 
\end{lem}
\begin{proof}
The proof of the same statement for untwisted case is given in \cite[Prop 9.4]{sheelthesis} and we adapt their argument to twisted case.

We start by matching the generators: the generators of $\fM_{L''}(L_0,L_0')$ are given by chords from $\tilde L_0\times L_0'$ to $\pi^{-1}(L\times_\phi L')=\bigsqcup_{r\in\mathbb{Z}} \tilde L\langle -r\rangle \times \phi^{-r}(L')$. Conditions (\ref{fdexp:prd}) and (\ref{fdexp:prdfirstcomp}) imply that the strips with no markings on the middle seam are endowed with product type Floer data coming from $\wh{T_0}\times \wh{M}$; therefore, the graded vector space generated by these chords can be identified with $\bigoplus_{r\in\Z}CW(L_0,L)^r\otimes CW(L_0', \phi^{-r}(L') )$, where $CW(L_0,L')^r$ denotes the degree $r$ part in the extra grading as before. 
Hence, as a graded vector space, we have an identification of $\fM_{L''}(L_0,L_0')$ with $(h_L\otimes_{tw} h_{L'})(L_0,L'_0)$ (see (\ref{eq:twyon})).

That $\fM_{L''}$ and $h_L\otimes_{tw}h_{L'} $ have the same structure maps follows from the same proof as \cite[Prop 7.3, Prop 9.4]{sheelthesis}: first, the differential $\mu^{1|0;0}$ of $\fM_{L''}$ is defined via the count of pseudo-holomorphic quilted strips up to translation. As remarked the absence of marked points on the seam, and Conditions (\ref{fdexp:prd}) and (\ref{fdexp:prdfirstcomp}), imply the Floer data on these strips are of product type coming from $\wh{T_0}\times \wh{M}$. Therefore, such quilted strips can be identified with pairs of pseudo-holomorphic strips mapping into $\wh{ T_0}$ and $\wh{M}$, where one strip is constant and the other is a Floer strip. Therefore, under the identification as vector spaces above, $\mu^{1|0;0}$ turns into the differential $\pm\mu^1\otimes 1\pm 1\otimes \mu^1$ of $h_L\otimes_{tw}h_{L'} $.

The proof that $\mu^{1|0;s}, s\neq 0$ match is similar. We prove matching of $\mu^{1|r;0},r\neq 0$. For simplicity assume $r=1$. Then, the structure map is as follows:
\begin{equation}
\mu^{1|1;0}:\fM_{L''}(L_0,\phi^g(L_0') )\otimes CW(L_1,L_0)^g\to \fM_{L''}(L_1,L_0')
\end{equation}
Let $y''$ be a generator of $\fM_{L''}(L_0,\phi^g(L_0') )$, i.e. a chord from $\tilde L_0\times\phi^g(L_0')$ to $\pi^{-1}(L'')$, with components $y$ and $y'$ (hence identifying with a generator $y\otimes y'$ of $(h_L\otimes_{tw} h_{L'})(L_0,\phi^g(L_0'))$). Let $x$ be a chord from $L_1$ to $L_0$ of degree $g$. To define $\mu^{1|1;0}$ we count strips as in Figure \ref{figure:labelledtwisted} with $\fM$-input $y''$, with one input, $x$, on the upper boundary, and with no marked points on the seam or lower boundary. As before, for this, one identifies $y''$ with a chord from $\tilde L_0\langle -g\rangle\times L_0'$ to $\pi^{-1}(L'')$ by applying $(\tr\times\phi)^{-g}$. The new chord has components $\tr^{-g}(y)$ and $\phi^{-g}(y')$. One identifies the quilted strips in consideration
with pairs of marked strips in $\wh{T_0}$ and $\wh{M}$. By rigidity, $\wh{M}$-component is a constant strip with output $\phi^{-g}(y')$, and the count on $\wh{T_0}$-component gives $\pm \mu^2(y,x)$. In other words, the count of quilted strips as in Figure \ref{figure:labelledtwisted} gives $\pm \mu^2(y,x)\otimes \phi^{-g}(y')$. This is exactly the structure map of $h_L\otimes_{tw}h_{L'}$, applied to $(y\otimes y'|x)$. 

To prove the vanishing of $\mu^{1|r;s}$ when $r,s>0$, consider a rigid pseudo-holomorphic quilted strip with no markings on the seam, but with at least one marking on the upper and lower boundary components. One can change the identification of $r_1$ and $r_2$ to obtain a one parameter family of quilted strips, and by Condition \ref{condition:youneedish}, they all contribute to the count. This contradicts the rigidity; therefore, there is no such strip.
\end{proof}
\begin{figure}
	\centering
	\includegraphics[height=3cm]{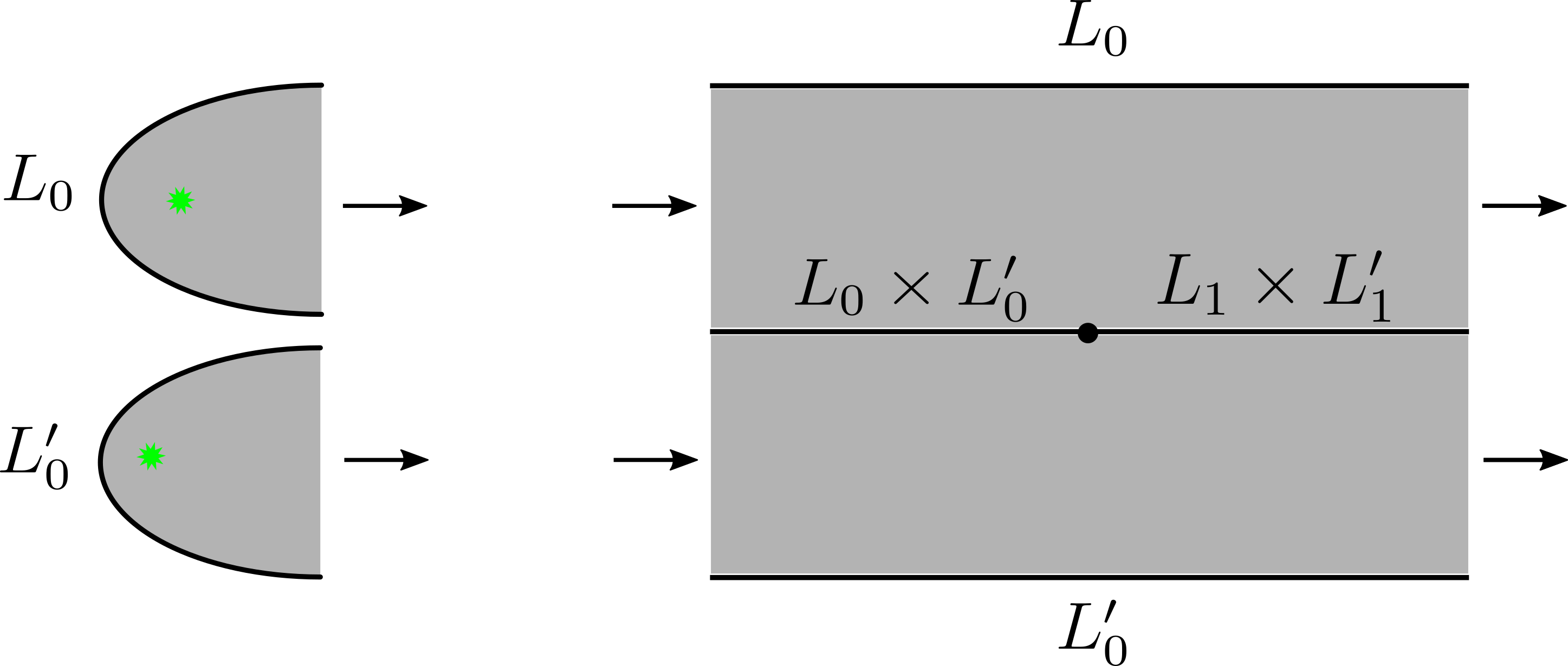}
	\caption{Rigid discs defining $e_{L_0}\otimes e_{L_0'}$ (on the left) and the initial component of the map denoted by $\fM$ in (\ref{eq:gammacompose}) (on the right)}
	\label{figure:quiltprecompose}
\end{figure}
Now, we turn to the proof of fully faithfulness of (\ref{eq:twfunct2}). Instead of explicit identification of generators, we give a geometric argument. By Lemma \ref{lem:mequalsyoneda}, the bimodule $\fM_{L''}$ is (strictly) isomorphic to $h_L\otimes_{tw}h_{L'}$. In particular, there exists an element of $\fM_{L''}(L,L')$ corresponding to $e_L\otimes e_{L'}\in (h_L\otimes_{tw}h_{L'})(L,L')$. We denote this element by $e_L\otimes e_{L'}$ as well. Moreover, by Yoneda Lemma (Lemma \ref{lem:twyoneda}), the map 
\begin{equation}
hom(\fM_{L''},\mathfrak{ N})\to \mathfrak{N}(L,L')
\end{equation}
that sends $f$ to $f^{1|0;0}(e_L\otimes e_{L'})$ is a quasi-isomorphism. Recall we denote this map by $\gamma_{L,L'}$. Therefore, given Lagrangian branes $L_0''=L_0\times_\phi L_0'\subset\wh{T_\phi}$ and $L_1''=L_1\times_\phi L_1'\subset\wh{T_\phi}$, we have a diagram
\begin{equation}\label{eq:gammacompose}
\xymatrixcolsep{5pc}\xymatrix{CW(L_0'',L_1'')\ar[r]^{\fM}\ar@/_1pc/[rd]^{\Gamma} & hom(\fM_{L_0''},\fM_{L_1''})\ar[d]^{\gamma_{L_0,L_0'}} \\ & \fM_{L_1''}(L_0,L_0')   }
\end{equation}
where $\Gamma$ is the composition. Showing $\fM$ is a quasi-isomorphism is equivalent to showing $\Gamma$ is a quasi-isomorphism, and the latter follows from a geometric description of $\Gamma$. For simplicity in labeling, consider the untwisted case (i.e. $\phi=1_M$). Cohomological units for $CW(L_0,L_0)$ and $CW(L_0',L_0')$ can be obtained by counting rigid pseudo-holomorphic discs with one output marking point and with boundary condition given by $L_0$, resp. $L_0'$ (see \cite[Section (8c)]{seidelbook}). Hence, a schematic picture for the maps in (\ref{eq:gammacompose}) is as in Figure \ref{figure:quiltprecompose} and a picture for the composition $\Gamma$ is as in Figure \ref{figure:shortquiltcompose}. More precisely, one would need to fold the quilt before composing (since, we choose Floer data on the folded strip). The green asterisks are auxiliary unconstrained points that rigidify the discs, and the marking on the middle seam is the input from $CW(L_0'',L_1'')$. After folding and gluing, $\Gamma$ looks like Figure \ref{figure:shortfoldedcompose}. As before, the purpose of the green asterisk is to stabilize the surface. Figure \ref{figure:shortfoldedcompose} is equivalent to count of rigid strips with Floer data used in definition of $\cW^2(T_\phi)$ on the input, and split type product data on the output. Hence, this is a continuation map, and it is a quasi-isomorphism. An explicit quasi-inverse can be defined using the same count of pseudo-holomorphic strips with types of Floer data reversed on the incoming and outgoing ends. 
\begin{figure}
\centering
\includegraphics[height=3cm]{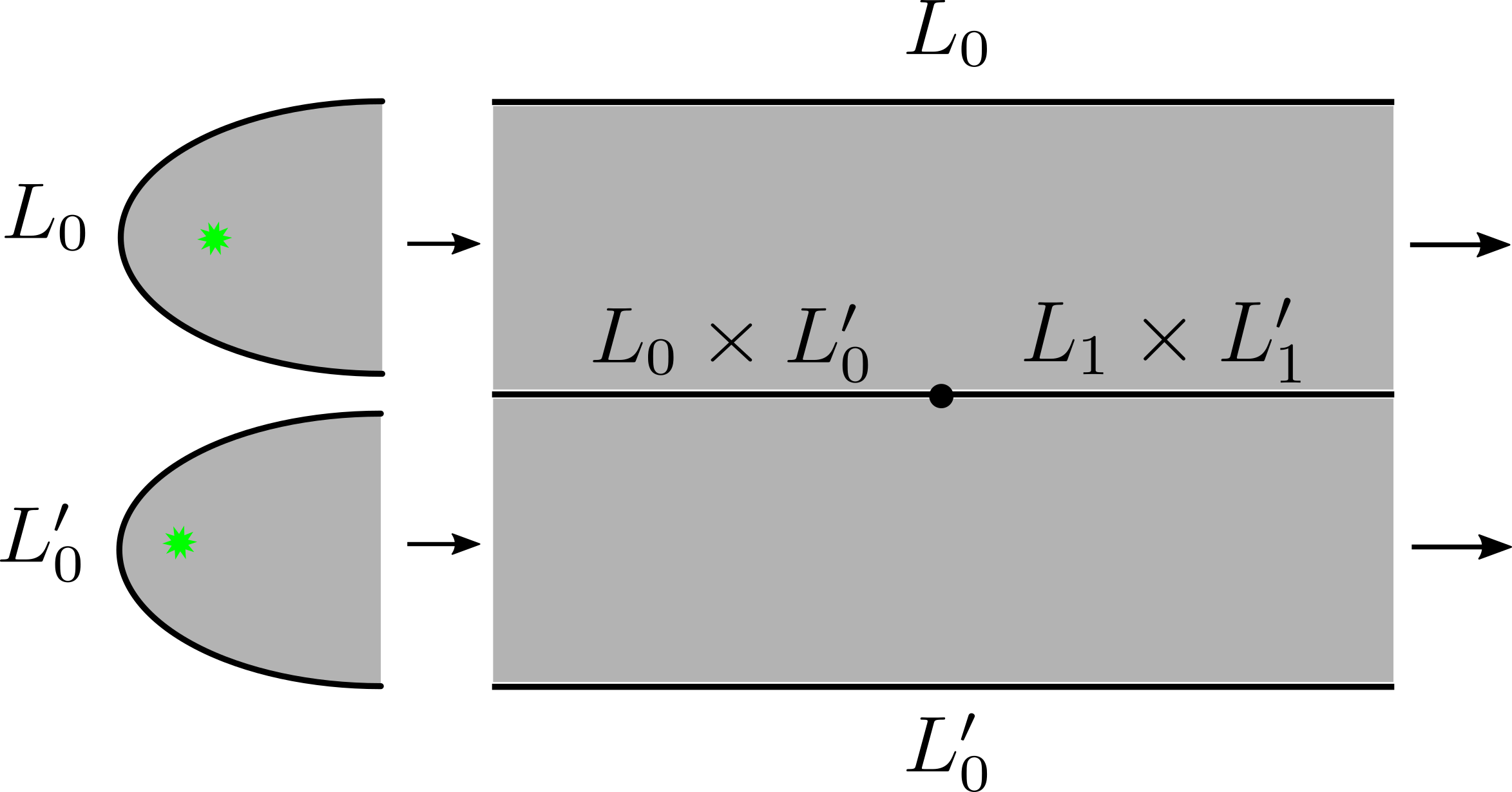}
\caption{The composition $\Gamma$}
\label{figure:shortquiltcompose}
\end{figure}
\begin{figure}
	\centering
	\includegraphics[height=3cm]{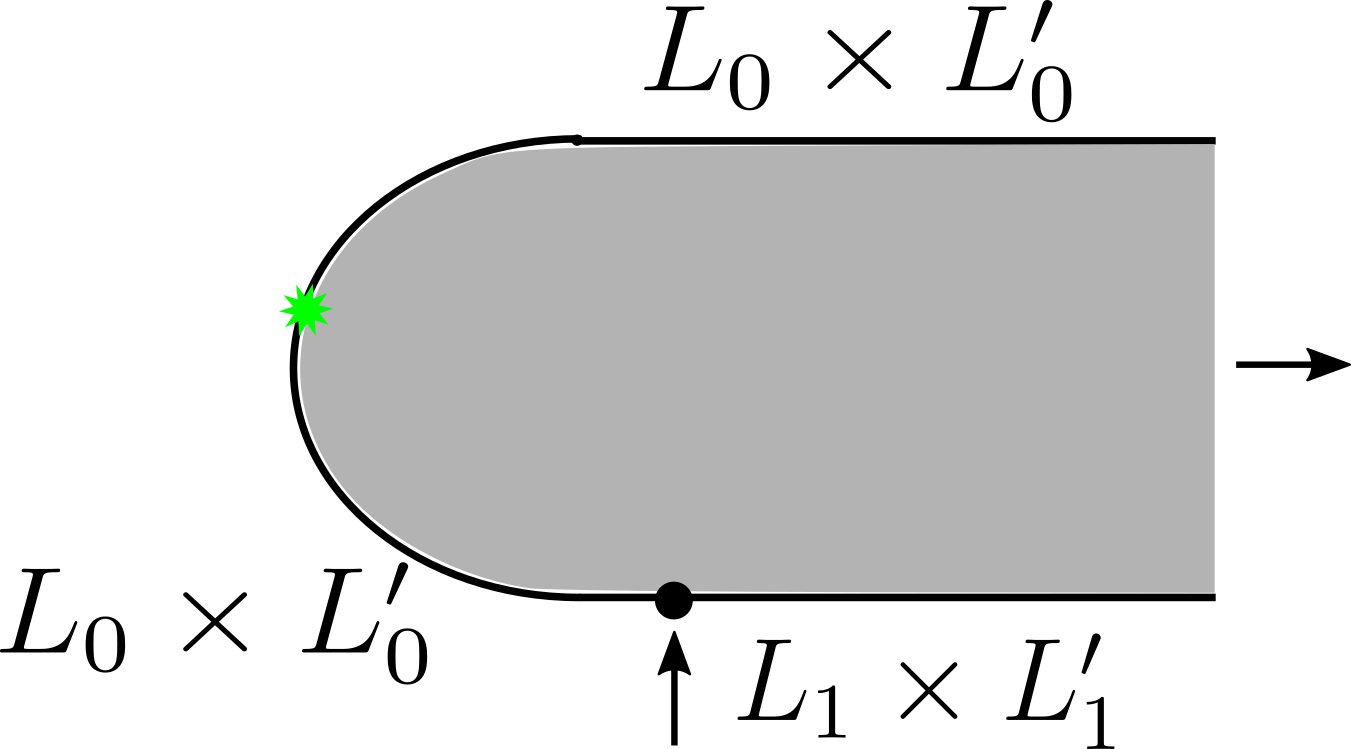}
	\caption{The composition $\Gamma$ after folding and gluing}
	\label{figure:shortfoldedcompose}
\end{figure}

The description of $\Gamma$ is similar in the twisted case: one modifies the labeling in Figure \ref{figure:shortquiltcompose} by replacing $L_0$ by the fixed lift $\tilde L_0$, $L_0''$ by $\pi^{-1}(L_0'')$ and $L_1''$ by $\pi^{-1}(L_1'')$. However, the part with label $\pi^{-1}(L_0'')$ would necessarily map to $\tilde L_0\times L_0'\subset \pi^{-1}(L_0'')$, and the rest of the seam maps to the component of $\pi^{-1}(L_1'')$ determined by the input chord in $\wh{T_\phi}$ (recall the clarification we have given about the labels on seam). The labeling in Figure \ref{figure:shortfoldedcompose} is modified similarly, and its count gives $\Gamma$ (up to homotopy). 
The rest of the argument is identical, and this proves fully faithfulness of (\ref{eq:twfunct2}) up to compactness issue that we have slid under the rug so far. One can define the appropriate Floer data (including a non-zero perturbation term) so that Gromov compactness holds. This is the approach that we took in our thesis \cite{kartalthesis}; however, it makes analysis substantially harder. Instead, we will switch to the definition of wrapped Fukaya category given in \cite{GPS1}. This will make the algebra slightly more elaborate, but simplify the analysis substantially.
\subsection{Reminder of wrapped Fukaya categories}\label{subsec:reminderwfuk}
In this section, we briefly recall the definition of wrapped Fukaya categories following \cite{GPS1}.

Let $X$ be a Liouville domain and let $X\subset \wh{X}$ denote its completion. Let $L_0, L_1\subset \wh{X}$ be two exact, cylindrical Lagrangian branes, i.e. exact, cylindrical Lagrangians equipped with grading and spin structures, and assume they are transverse to each other (recall that a Lagrangian is called cylindrical if it is invariant under the Liouville flow outside a compact subset of $\wh{X}$). Denote the complex vector space generated by the elements of $L_0\cap L_1$ by $CF(L_0,L_1)$. As we equip $L_0$ and $L_1$ with gradings, $CF(L_0,L_1)$ carries a natural $\Z$-grading. For a generic cylindrical almost complex structure (i.e. an almost complex structure that is invariant under Liouville flow outside a compact subset of $\wh{X}$), the count of pseudo-holomorphic strips with boundary components on $L_0$ and $L_1$ defines a differential $\mu^1$ of degree $1$ on $CF(L_0,L_1)$. More precisely, given generators $x,y\in L_0\cap L_1$, the moduli space of such holomorphic strips asymptotic to $x$ and $y$ is an oriented manifold of dimension $deg(y)-deg(x)-1$. Therefore, when $deg(y)=deg(x)+1$, one can define the coefficient of $y$ in $\mu^1(x)$ to be the signed count of the elements of the corresponding $0$-dimensional moduli space. Standard gluing and compactness arguments show that $\mu^1\circ \mu^1=0$. Similarly, given Lagrangian branes $L_0,\dots ,L_n$ that are pairwise transverse to each other, one can define a map \begin{equation}
\mu^n:CF(L_{n-1},L_n)\otimes \dots \otimes CF(L_0,L_1) \to CF(L_{0},L_n)[2-n]
\end{equation}
by counting pseudo-holomorphic discs with $n+1$ points removed from its boundary. These can be shown to satisfy $A_\infty$-equations, as long as the almost complex structures are chosen consistently. See \cite{seidelbook} for more details. 

Now consider a set of cylindrical Lagrangian branes, i.e. cylindrical Lagrangians equipped with grading and spin structures, in $\wh{X}$. One can ``wrap'' a Lagrangian in this set by applying the Hamiltonian flow of a function that is linear and positive on the conical end (outside a compact set). Choose a set of wrapped Lagrangians for each such brane, denoted by $\{L^{(i)}:i\in\mathbb{N} \}$. More precisely, each $L^{(i+1)}$ is obtained from $L^{(i)}$ by a positive isotopy as above. Also assume this set is cofinal in the sense that for any such positive linear isotopy $L\to L^+$, one can find an $i$ such that there is a positive linear isotopy $L^+\to L^{(i)}$. Further assume $L_0^{(i)}$ and $L_1^{(j)}$ are transverse if $j>i$ (or $i>j$). 

Let $\Om(X)$ denote the $A_\infty$-category whose objects are $L^{(i)}$ for all $i$ and $L$. The set of objects is partially ordered by $i$ (i.e. $L_1^{(j)}>L_0^{(i)}$ if and only if $j>i$). Hence, we will drop the superscript and denote the objects by $L_0$, $L_1$ etc. One defines morphism sets to be 
\begin{equation}\label{eq:homsetsom}
\Om(X)(L_0,L_1)=\begin{cases}
CF(L_0,L_1),& \text{if }L_1>L_0\\
\C,&\text{if } L_1=L_0\\
0&\text{otherwise}
\end{cases}
\end{equation}
We modified the order in \cite{GPS1}. As mentioned above, to define the $A_\infty$-structure on $\Om(X)$, we make asymptotically consistent choices of cylindrical almost complex structures as in \cite{seidelbook} (it suffices to make the choice for labelings satisfying $L_n>L_{n-1}>\dots>L_0$). Gromov compactness holds (either by maximum principle or by \cite[Lemma 7.2]{abousei}), and for a generic choice of almost complex structures, we obtain an ordered $A_\infty$-category as usual (we let $1\in \Om(X)(L_0,L_0)= \C$ act as a strict unit). See \cite{GPS1} for more details.

For each positive isotopy $L\to L^+$, one has a continuation element $c_{L,L^+}\in CF(L,L^+)$. This element can be defined by the count of rigid pseudo-holomorphic discs with (positively) varying boundary conditions (see Figure \ref{figure:contelt}).
\begin{figure}
	\centering
	\includegraphics[height=3cm]{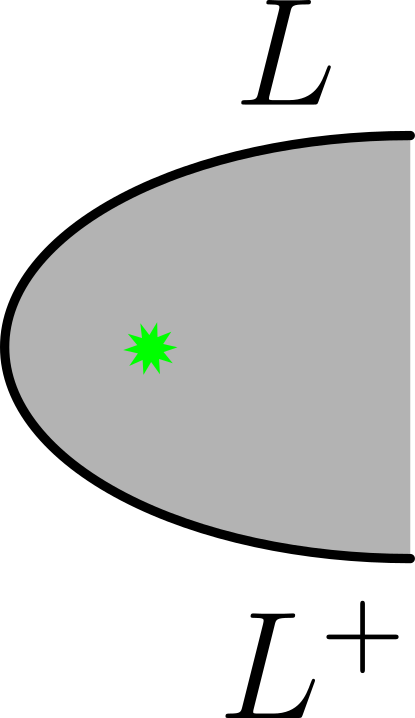}
	\caption{Rigid discs defining the continuation element}
	\label{figure:contelt}
\end{figure}
To define this element, one has to choose a family of almost complex structures parametrized by the rigid disc that restricts to chosen almost complex structure for the pair $(L,L^+)$ on the strip-like end. It is easy to see that different choices give rise to homologous continuation elements. Moreover, the product $\mu^2(c_{L^+,L^{++} },c_{L,L^+})$ is homologous to continuation element for $L\to L^{++}$. Let $C$ denote the set of all continuation elements for all $L^{(i)}\to L^{(i+1)}$. Define the wrapped Fukaya category $\cW(X)$ to be $C^{-1}\Om(X)$--- the localization of $\Om(X)$ on all continuation elements. See \cite{GPS1} for more details. 

The advantage of this definition is that the lack of perturbation term in Fukaya category makes compactness arguments substantially easier. 
\begin{notation}
Let $\Om_X(L_0,L_1)$, resp. $\cW_X(L_0,L_1)$ also denote the hom-complex $\Om(X)(L_0,L_1)$, resp. $\cW(X)(L_0,L_1)$.
\end{notation}
Let us now define an analogue of wrapped Fukaya category with split Hamiltonians for $T_\phi$, which is analogous to $\cW^2$ of \cite{sheelthesis}, $\cW^s$ of \cite{gaofunctor} and $\cW^{prod}$ of \cite{GPS2}. The basic idea is the following: even though $T_\phi$ is not a product, its conical end can be identified with the conical end of the product. Hence, one can talk about the product type data on the conical end.

More precisely:
\begin{equation}
\wh T_\phi\setminus T_\phi=\big((\wh{ \tilde T_0}\times\wh M)\setminus (\tilde T_0\times M)\big)/(\tr\times \phi)=\atop \bigg[\big((\wh{ \tilde T_0}\setminus \tilde T_0)\times \wh M\big)/(\tr\times\phi)\bigg]\cup \bigg[\big(\wh{ \tilde T_0}\times (\wh M\setminus M)\big)/(\tr\times\phi)\bigg]
\end{equation}
$(\wh{ \tilde T_0}\setminus \tilde T_0)$ is isomorphic to infinitely many copies of $\wh T_0\setminus T_0$ and $\tr$ moves one to the next. Hence
\begin{equation}\label{eq:prd1}
\big((\wh{ \tilde T_0}\setminus \tilde T_0)\times \wh M\big)/(\tr\times\phi)\cong (\wh{  T_0}\setminus T_0)\times \wh M
\end{equation}
Moreover, $\phi$ acts trivially on $(\wh M\setminus M)$; hence
\begin{equation}\label{eq:prd2}
\big(\wh{ \tilde T_0}\times (\wh M\setminus M)\big)/(\tr\times\phi)\cong \wh T_0\times (\wh M\setminus M)
\end{equation} 
The intersection of these subsets is isomorphic to $(\wh T_0\setminus T_0)\times (\wh M\setminus M)$ with the obvious embeddings. Hence, the conical end of $T_\phi$ can be written as the union of products (\ref{eq:prd1}) and (\ref{eq:prd2}). Define
\begin{defn}\label{defn:om2phitype}
An almost complex structure on $\wh{T_\phi}$ is called $\Om^2(T_\phi)$-type if its restriction to $\overline{\wh{  T_0}\setminus T_0}\times \wh M$, resp. $\wh T_0\times \overline{\wh M\setminus M}$ is product type and its first, resp. second component is cylindrical outside a compact subset of $\wh{T_\phi}$. 
A family of such data is called $\Om^2(T_\phi)$-type if further the compact subset of $\wh{T_\phi}$ 
can be chosen locally uniformly over the family (c.f. Condition \ref{condition:loccon}).
\end{defn}
\begin{lem}\label{lem:existom2acs}
$\Om^2(T_\phi)$-type almost complex structures exists and the space of such is weakly contractible (i.e. any $\Om^2(T_\phi)$-type family extend to the cone of its parameter space).
\end{lem}
\begin{proof}
For the existence, put almost complex structures on $\overline{\wh{  T_0}\setminus T_0}\times \wh M$, resp. $\wh T_0\times \overline{\wh M\setminus M}$ that agree on the intersection and that satisfy the assumptions (product type with components that are cylindrical on conical ends), then extend smoothly to the rest of $\wh{T_\phi}$.

For the connectedness of such almost complex structures, one needs to preserve the product type assumption. Let $J$ and $J'$ be two such almost complex structures, and for simplicity assume that cylindrical components assumption holds over all $\overline{\wh{  T_0}\setminus T_0}\times \wh M$ and $\wh T_0\times \overline{\wh M\setminus M}$. First, construct an almost complex structure $J_{int}$ that agree with $J$ in the first components (on the subsets $\overline{\wh{  T_0}\setminus T_0}\times \wh M$ and $\wh T_0\times \overline{\wh M\setminus M}$) and with $J'$ on the second components. Using the connectedness of spaces of cylindrical almost complex structures on the second components, one can interpolate $J$ and $J_{int}$ along almost complex structures that are product type on the same subsets and that have the same first components. Then, one can interpolate $J_{int}$ and $J'$ similarly. The homotopies extend to the interior of $\wh{T_\phi}$ (A simple illustration of the idea is the following: to connect two metrics $g_1\times g_2$ and $g_1'\times g_2'$ on the product along product type metrics, one first connects $g_1\times g_2$ to $g_1\times g'_2$ then $g_1\times g'_2$ to $g'_1\times g'_2$. Indeed, this together with retraction of the space of metrics onto almost complex structures let us connect product type almost complex structures as well).  

The higher connectivity and contractability are similar. Indeed, one can do the same for a family of almost complex structures parametrized by a topological space $X$. Namely, first extend the family from $X\cong X\times\{0\}$ to a family of almost complex structures parametrized by $X\times[0,1]$, where the first components are the same over $X\times \{1\}$. Then, use the contractability of almost complex structures on the second component to extend this family to cone of $X$ (i.e. to $X\times[0,2]/X\times\{2\}$).
\end{proof}
To define the category $\Om^2(T_\phi)$, choose a generating set of Lagrangians for $\cW(T_0)$ and $\cW(M)$ (indeed, we take this set to be $\{L_{pur},L_{gr}\}$ for the former and the set of cocores of a Weinstein structure for the latter). Endow them with brane structure, and choose a cofinal set of positive isotopies $L\to L^{(i)}\subset\wh{T_0}$, $L'\to L'^{(j)}\subset\wh{M}$ for these generators, as before (satisfying transversality of ordered pairs). Then, one obtains a set of Lagrangians $L^{(i)}\times_\phi L'^{(j)}$, indexed by $\mathbb{N}\times\mathbb{N}$ (let the lifts of positive wrappings be determined by the fixed lifts of $L_{pur}$ and $L_{gr}$). Define a partial ordering on this set by $L_1^{(i_1)}\times_\phi L_1'^{(j_1)}>L_0^{(i_0)}\times_\phi L_0'^{(j_0)}$ if and only if $i_1>i_0$ and $j_1>j_0$. We will drop the superscripts as before, and denote the Lagrangians in $\wh{T_\phi}$ by $L''$, $L''_k$, etc. Notice, if $L_1''>L_0''$, then $L_0''$ and $L_1''$ are transverse, as they are transverse in each component. Define $CF(L_0'',L_1'')$ as the linear span of intersections. Make asymptotically consistent choices of $\Om^2(T_\phi)$-type families of almost complex structures over moduli of stable discs for all labelings, which is possible due to Lemma \ref{lem:existom2acs}. 

Let $\Om^2(T_\phi)$ denote the category with objects $L^{(i)}\times_\phi L'^{(j)}$, and with hom-sets defined similar to (\ref{eq:homsetsom}). For a generic choice of almost complex structures, one can use the count of pseudo-holomorphic discs to define an $A_\infty$-structure. Similar to before, one has continuation elements in $CF(L'',L''^{+})$ for any componentwise positive linear isotopy $L''\to L''^{+}$ (defined using the count of stable discs as in Figure \ref{figure:contelt} with $\Om^2(T_\phi)$-type almost complex structures). Define $\cW^2(T_\phi)$ to be the localization of $\Om^2(T_\phi)$ at all continuation elements.  
\begin{notation}
Let $\Om^2(L_0'',L_1'')$, resp. $\cW^2(L_0'',L_1'')$ also denote the hom-complexes $\Om^2(T_\phi)(L_0'',L_1'')$, resp. $\cW^2(T_\phi)(L_0'',L_1'')$. 
\end{notation}
\begin{rk}
One can restrict the set of objects to those with diagonal indices, i.e. $L^{(i)}\times_\phi L'^{(i)}$. After localization, one obtains the same category since $\{(i,i):i\in\mathbb{N} \}$ is cofinal in the poset $\mathbb{N}\times\mathbb{N}$. See \cite[Section 6.5]{GPS2}.
\end{rk}
\begin{rk}
The compactness hold for $\Om^2(T_\phi)$ by applying (integrated) maximum principle to each component. More precisely, since we use $\Om^2(T_\phi)$-type almost complex structures, outside a compact subset they are product type over $\overline{\wh{  T_0}\setminus T_0}\times \wh M$, resp. $\wh T_0\times \overline{\wh M\setminus M}$. Therefore, the curves cannot escape to infinity over $(\wh{  T_0}\setminus T_0)$-component, resp. $(\wh{M}\setminus M)$-component.
\end{rk}
The first step to define (\ref{eq:twfunct}) is to define a functor 
\begin{equation}
\Om^2(T_\phi)\to Bimod_{tw}(\Om(T_0),\Om(M) )
\end{equation}
To define the right hand side, we assume $\Om(T_0)$ is endowed with an extra grading as before (i.e. via the covering map $\wh{\tilde T_0}\to \wh{T_0}$), and $\Om(M)$ is made $\phi$-equivariant as before (assume that $\phi^k(L')^{(i)}=\phi^k(L'^{(i)})$). After defining the functor, we will show it descends to a fully faithful functor of wrapped Fukaya categories. 
\begin{note}\label{note:w2equalsw}
One still needs to show that $\cW^2(T_\phi)$ and $\cW(T_\phi)$ are equivalent. A proof of equivalence of these categories is given in \cite{GPS2} in the untwisted case (i.e. when $\phi=1_M$ and $T_\phi=T_0\times M$, they consider more general products). Their proof applies verbatim in the twisted case, since the cylindrical end of $\wh{T_\phi}$ can be identified with that of the product. Thanks to this identification, one can use their method (``the cylindrization'') to make Lagrangians of type $L\times_\phi L'$ cylindrical. One can also use same type of almost complex structures (on the cylindrical end) to write  a $\Om^2(T_\phi)$-$\Om(T_\phi)$-bimodule that induces functors $\Om^2(T_\phi)\to \Om(T_\phi)^{mod}$, and $\cW^2(T_\phi)\to \cW(T_\phi)^{mod}$. The essential image is the span of Yoneda modules over $\cW(T_\phi)$ of Lagrangians of type $L\times_\phi L'$. By Corollary \ref{torusgenerators}, this span generates $\cW(T_\phi)$.
Same proof in \cite{GPS2} for fully faithfulness applies. We will not include this proof here.
\end{note}
\begin{rk}
Presumably, an alternative proof of fully faithfulness can be given similar to geometric way we presented in Section \ref{subsec:defnkunnethearly}: namely, one can compose the map $\cW^2(T_\phi)(L_0'',L_1'')\to hom_{\cW(T_\phi)}(h_{L_0''},h_{L_1''})$ induced by the functor with unit insertion/Yoneda quasi-isomorphism, and one gives a geometric description of the composition as a count of curves similar to Figure \ref{figure:shortfoldedcompose}. Then it is easy to describe a quasi-inverse as before. We were unable to solve compactness related problems that can occur in the definition of \cite{generation}; however, these can be overcome as in \cite{GPS2} by switching the definitions. Then, one has to provide further algebraic arguments similar to Section \ref{subsec:algmodify}. 
\end{rk}
\subsection{Compactness issues}
 In this section, we find conditions on the almost complex structures defining the trimodule such that Gromov compactness holds. The trimodule structure is defined by counting pseudo-holomorphic curves into $\wh{\tilde T_0}\times\wh{M}$. Therefore, one needs to control these curves from escaping to infinity on the conical end and escaping to right/left ends of infinite type Liouville manifold $\wh{\tilde T_0}$. The former is standard, and is achieved via integrated maximum principle. For the latter, we use almost complex structures such that the energy of pseudo-holomorphic curves increase at least by a fixed amount each time they cross any of the pre-determined shells. This can presumably seen as a version of intermittently bounded almost complex structures of \cite{groman}. 
\subsubsection{Escaping the conical end}\label{subsubsec:noescapeconical}
When there is no perturbation term, curves can be prevented from escaping to conical end using standard maximum principle. Alternatively, one can use the baby version of \cite[Lemma 7.2]{abousei}: namely, let $(X,\lambda_X)$ be a Liouville domain with completion $\wh{X}$, $L_X\subset\wh{X}$ be a Lagrangian that is cylindrical on $\wh{X}\setminus X$ such that $\lambda_X|_{L_X}=0$ and choose an almost complex structure that is cylindrical on the conical end. Let $\Sigma$ be a closed Riemann surface with corners such that $\partial\Sigma$ can be written as union of smooth curves $\partial_l\Sigma$ and $\partial_n\Sigma$. Then, any pseudo-holomorphic map $u:\Sigma\to \wh{X}$ such that $u(\partial_l\Sigma)\subset L_X$ and $u(\partial_n\Sigma)\subset \partial X$ is contained entirely in $X$. 

This can be used for split type almost complex structures that are cylindrical in each component, which we already remarked for $\Om^2(T_\phi)$-type almost complex structures on $\wh{T_\phi}$. Now, consider the cylindrical end $\wh{\tilde T_0}\setminus \tilde T_0$ of $\tilde T_0$. It has infinitely many components, which are isomorphic to cylindrical end $\wh{T_0}\setminus T_0\cong \mathbb{R}_{>0}\times S^1$. Let $\tilde B_i$ denote the closures of these components (enumerated by $i\in\Z$ such that $\tr$ moves $\tilde B_i$ to $\tilde B_{i+1}$). We define 
\begin{defn}\label{defn:o2type}
An almost complex structure $J$ on $\wh{\tilde T_0}\times\wh{M}$ is called $\Om^2$-type if 
\begin{itemize}
\item for each $i\in\Z$, the restriction of $J$ to $\tilde B_i\times\wh{M}$ is of split type and its $\tilde B_i$-component is cylindrical outside a compact subset of $\tilde B_i$
\item the restriction of $J$ to $\wh{\tilde T_0}\times \overline{\wh{M}\setminus M}$ is of split type and its $\overline{\wh{M}\setminus M}$ component is cylindrical outside a compact subset of $\overline{\wh{M}\setminus M}$
\end{itemize}
\end{defn}
\begin{rk}
We will make an assumption about the families of $\Om^2$-type data similar to one in Definition \ref{defn:om2phitype} (see Condition \ref{condition:loccon}). However, we do not assume any uniformity for the compact subsets of $\tilde B_i$ on which $J$ can be non-cylindrical as $i$ varies. By the (integrated) maximum principle, the projection of pseudo-holomorphic curve (with fixed asymptotic/boundary conditions) to $\wh{\tilde{T}_0}$ is contained in a compact subset that depends on $\tilde B_i$. Next, we will put a condition on our almost complex structures preventing curves from going to left/right infinite ends of $\tilde T_0$; hence, only finitely many of $\tilde B_i$ will be of concern.
\end{rk}
\subsubsection{Escaping the left and right ends of $\tilde T_0$}\label{subsubsec:noescapelr}
$\tilde T_0\times M$ is not of finite type; therefore, we have to find a class of almost complex structures for which pseudo-holomorphic curves with fixed (finite) boundary conditions, fixed asymptotic conditions and bounded energy do not escape to left and right ends of $\tilde T_0$. 
Presumably, one can use $i$-bounded almost complex structures in \cite{groman}. However, there is a simpler solution, which can actually be seen as a baby version of $i$-boundedness of \cite{groman}. Choose an annulus $A\subset T_0$ as in Figure \ref{figure:annuliontori}. In particular it satisfies:
\begin{enumerate}
	\item $A$ lifts to $\tilde T_0$
	\item Every path connected subset of $\tilde T_0$ that is not contained in a finite subdomain crosses infinitely many lifts of $A$ on both boundary components.
\end{enumerate}
Fix a symplectic trivialization of fibration $T_\phi\to T_0$ over $A$, i.e. identify the pre-image of $A$ with $A\times M$. Let $\tilde A_i\subset \tilde T_0$ denote the lifts of $A$. We will choose the almost complex structure on $\wh{\tilde T_0}\times \wh{M}$ to be of product type over 
$\tilde A_i\times \wh{M}$ (we do not have to say with respect to which trivialization of $\tilde A_i\times \wh{M}\to \tilde A_i$, as they all differ by $1\times \phi^k$ and this does not effect the product type assumption). 
\begin{figure}
	\centering
	\includegraphics[height=4cm]{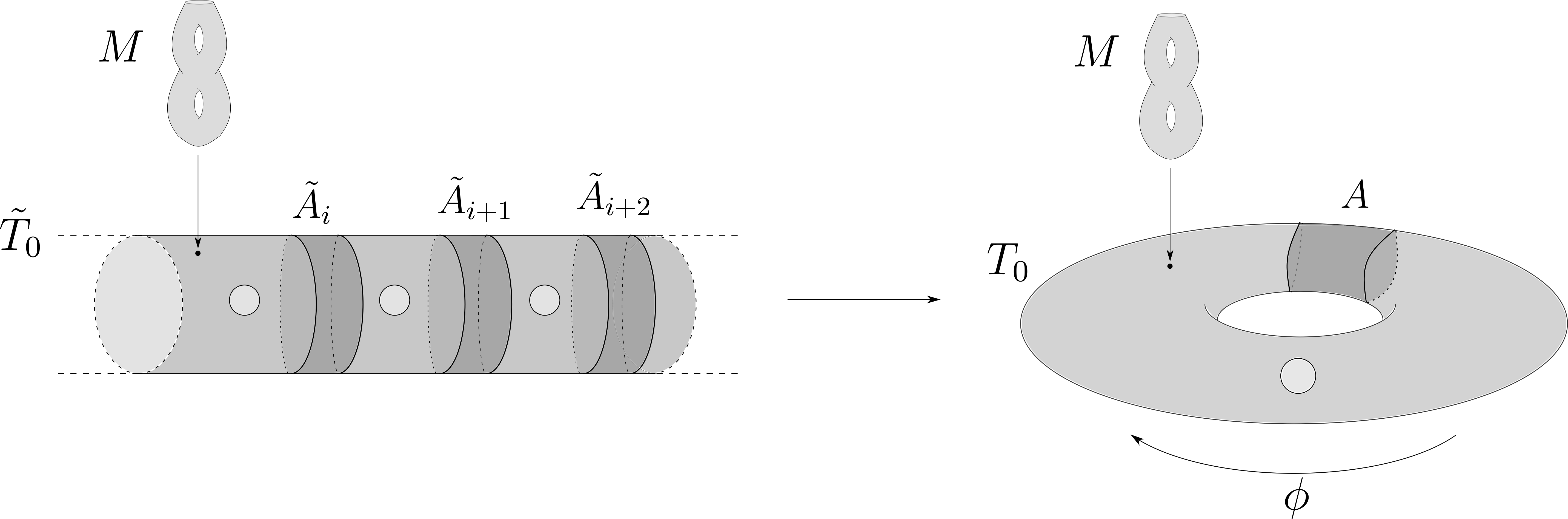}
	\caption{The annulus $A\subset T_0$ and its lifts to $\tilde T_0$}
	\label{figure:annuliontori}
\end{figure}

The projection of a given pseudo-holomorphic map into $\wh{\tilde T_0}\times \wh{M}$ is pseudo-holomorphic over the annuli $\tilde A_i$. This is due to product type assumption. More precisely, let $u:\Sigma\to \tilde A_i\times\wh{M}$ be a pseudo-holomorphic map from a closed Riemann surface $\Sigma$ with boundary such that $u(\partial\Sigma)\subset\partial \tilde A_i\times\wh{M}$ and such that $u(\partial\Sigma)$ intersects both components of $\partial \tilde A_i\times\wh{M}$. Then the energy of $u$ is bounded below by $Area(\tilde A_i)=Area(A)$. This holds since the restriction of $u$ to $u^{-1}(\tilde A_i\times\wh{M})$ splits into pseudo-holomorphic maps into $\tilde A_i$ and $\wh{M}$, and the first component has the given energy bound. 

Therefore, pseudo-holomorphic curves with fixed boundary/asymptotic conditions (hence fixed energy) cannot cross infinitely many $\tilde A_i$. 

Now, let us define a left-right-right $\Om^2(T_\phi)$-$\Om(T_0)$-$\Om(M)$-trimodule with twisting between $\Om(T_0)$-$\Om(M))$ components, which we still denote by $\fM$. In other words, we define a functor 
\begin{equation}\label{eq:omfunct}
\Om^2(T_\phi)\to Bimod_{tw}(\Om(T_0),\Om(M))
\end{equation}
This will be defined using the same count of quilted strips as before. 

Recall that the objects $L^{(i)}\times_\phi L'^{(j)}$ of $\Om^2(T_\phi)$ are partially ordered by their superscript $(i,j)$. It is easy to see that this defines a partial order on $ob(\Om^2(T_\phi))\sqcup ob(\Om(T_0))\times ob(\Om(M))$ as well. For instance, $L_1^{(3)}\times_\phi L_1'^{(5)}$ seen as an object of $\Om^2(T_\phi)$ is greater than $(L_0^{(2)},L_0'^{(4)})$, smaller than $(L_0^{(4)},L_0'^{(6)})$, but cannot be compared to $(L_0^{(5)},L_0'^{(3)})$. For simplicity, we assume $ob(\Om(T_0))\times ob(\Om(M))$ and $ob(\Om^2(T_\phi))$  can be identified via $(L,L')\mapsto L\times_\phi L'  $, and we denote the pair $(L,L')$ by $L\times_\phi L'$ as well. 

By assumptions, $L_0\times_\phi L_0'$ and $L_1\times_\phi L_1'$ are transverse when $L_1\times_\phi L_1'>L_0\times_\phi L_0'$. Given $L\in \Om(T_0)$,  $L'\in\Om(M)$ and $L''\in \Om^2(T_\phi)$, define 
\begin{equation}\label{eq:grdvsoftrimodule}
\fM(L'',L,L'):=\begin{cases}
CF(\tilde L\times L',\pi^{-1}(L'')),&\text{if }L''>L\times_\phi L'\\
0,&\text{otherwise}
\end{cases}
\end{equation}
as a graded vector space (here $\tilde L$ denotes the fixed lift of $L$ as usual). In other words, it is the $\C$-linear span of intersection points of $\pi^{-1}(L'')$ and $\tilde L\times L'$, when $L''>L\times_\phi L'$. This span admits a canonical grading, once the gradings on $L$, $L'$ and $L''$ are fixed. We will define a trimodule structure on $\fM$.

To define the structure maps, we need to choose a family $J$ of almost complex structures on $\wh{\tilde T_0}\times\wh{M}$
parametrized by the folded strip $\cS_r^{f}$ for each $r$. We impose the following conditions:
\begin{enumerate}
\newcounter{fdcond}
\item\label{fd:trphi} it is $\tr\times\phi$-invariant on $N_{1/3}$ (i.e. it is pull-back of an almost complex structure on $\wh{T_\phi}$)
\item\label{fd:prd} it is split type on $\cS_r^f\setminus N_{2/3}$ (i.e. it decomposes into almost complex structures on $\wh{\tilde T_0}$ and $\wh{M}$)
\item\label{fd:prdfirstcomp} $\wh{\tilde T_0}$-component extends to an almost complex structure on $\cS_r^{(1)}\setminus N_{2/3}$ and it is $\tr$-invariant (i.e. it is pull-back of an almost complex structure on $\wh{T_0}$)
\item\label{fd:prdMcomponent} $\wh{M}$-component extends to an almost complex structure on $\cS_r^{(2)}\setminus N_{2/3}$
\addtocounter{fdcond}{\value{enumi} }
\end{enumerate}
Recall $\cS_r^{(1)}$ and $\cS_r^{(2)}$ denote the complements of markings in components $r_1$ and $r_2$ of the quilted strip $r$, and $N_\epsilon$ denote the set of points that have distance less than $\epsilon$ to at least one of the marked points on the lower boundary of $\cS_r^{f}$.
Conditions (\ref{fd:trphi})-(\ref{fd:prdMcomponent}) are analogous to the previously given ones, and are necessary for trimodule equations. We also have the following conditions to ensure compactness:
\begin{enumerate}
\setcounter{enumi}{\value{fdcond}}
\item\label{fd:om2} for each $z\in \cS_r^{(f)}$, $J_z$ is $\Om^2$-type 
\item\label{fd:annulusproduct} for each $z\in \cS_r^{(f)}$ and each $i\in\Z$, 
the restriction of $J_z$ to $\tilde A_i\times \wh{M}$ is product type
\setcounter{fdcond}{\value{enumi} }
\end{enumerate}
As standard, we implicitly assume the almost complex structures are translation invariant on strip-like ends (for the upper markings of $\cS_r^{(1)}$, resp. $\cS_r^{(2)}$ the $\wh{\tilde T_0}$, resp. $\wh{M}$ components are translation invariant). We need the following (local) consistency condition for families of almost complex structures:
\begin{condition}\label{condition:loccon}
The compact subsets of $\tilde B_i$ (as defined in Section \ref{subsubsec:noescapeconical}), resp. $\overline{\wh{M}\setminus M}$ on which the respective component of $J_z$ is allowed to be non-cylindrical (see Definition \ref{defn:o2type}) can be chosen uniformly in a neighborhood of $z$. In particular, this condition holds for consistent choices of almost complex structures over the moduli spaces $\overline{\cQ(\textbf{d})}$.
\end{condition}
For completeness, we prove:
\begin{lem}\label{lem:exists2} Families of almost complex structures satisfying (\ref{fd:trphi})-(\ref{fd:annulusproduct}) exists and the space of such families is weakly contractible.
\end{lem}
\begin{proof}
The proof is similar to	Lemma \ref{lem:existom2acs}; hence, we will not give the full details. 

The proof of Lemma \ref{lem:existom2acs} can easily be modified to show the existence of $\Om^2(T_\phi)$-type almost complex structures that are of product type over $A\times\wh{M}\subset\wh{T_\phi}$ as well (in addition to being product type over $\overline{\wh{T_0}\setminus T_0}\times\wh{M}$ and $\wh{T_0}\times\overline{\wh{M}\setminus M}$). 
The pull-back of such almost complex structures under $\wh{\tilde T_0}\times \wh{M}\to \wh{T_\phi}$ satisfy (\ref{fd:trphi}), (\ref{fd:om2}) and (\ref{fd:annulusproduct}). Hence, fix a family of such almost complex structures over $\overline{N_{1/3}}$. One can also construct a family of $\Om^2$-type almost complex structures satisfying (\ref{fd:prd}), (\ref{fd:prdfirstcomp}) and (\ref{fd:prdMcomponent}) over $\overline{N_{2/3}^c}$ by pulling back products of almost complex structures on $\wh{T_0}$ and $\wh{M}$ ((\ref{fd:annulusproduct}) is automatic in this case). 

To extend/interpolate it over $\cS_r^{(f)}$, one again imitates the proof of Lemma \ref{lem:existom2acs}: First, extend from $N_{1/3}$ to $N_{2/5}$, while keeping conditions (\ref{fd:om2}) and (\ref{fd:annulusproduct}) so that the first components of almost complex structures on $\tilde A_i\times \wh{M}$, $\tilde B_i\times \wh{M}$ and $\tilde T_0\times \overline{\wh{M}\setminus M}$ are constant over $\partial N_{2/5}\setminus\partial\cS_r^{f}$. Then extend this to $N_{1/2}$ while fixing the first coordinates and while keeping conditions (\ref{fd:om2}) and (\ref{fd:annulusproduct}) so that the second coordinates of almost complex structures on $\tilde A_i\times \wh{M}$, $\tilde B_i\times \wh{M}$ and $\tilde T_0\times \overline{\wh{M}\setminus M}$ are constant over $\partial N_{1/2}\setminus\partial\cS_r^{f}$. Hence, over $\partial N_{1/2}\setminus\partial\cS_r^{f}$, the restrictions of almost complex structures to $\tilde A_i\times \wh{M}$, $\tilde B_i\times \wh{M}$ and $\tilde T_0\times \overline{\wh{M}\setminus M}$ does not vary. One can extend from $N_{2/3}^c$ to $N_{1/2}^c$ such that over the boundary of $N_{1/2}$, the almost complex structures on union of $\tilde A_i\times \wh{M}$, $\tilde B_i\times \wh{M}$ and $\tilde T_0\times \overline{\wh{M}\setminus M}$ are the same as the extension to $N_{1/2}$. This proves the existence (it is easy to smooth this family of almost complex structures). 

The proof that consistent choices exists and are weakly contractible requires the above construction over families, and it is the same as in Lemma \ref{lem:existom2acs}.
\end{proof}
Consider labellings as in Figure \ref{figure:labelledtwisted} satisfying $L_p''>\dots >L_0''>L_0\times_\phi L_0'$, $L_0>\dots >L_m$, and $L_0'>\dots L_n'$. Assume we make (asymptotically) consistent choices of almost complex structures satisfying (\ref{fd:trphi})-(\ref{fd:annulusproduct}), Condition \ref{condition:youneedish}, and Condition \ref{condition:loccon} (hence, implicitly the choices of $\Om^2(T_\phi)$-type data made to define this category satisfy a product type assumption over $A\subset T_0$ as well). The asymptotic consistency of $J$ means that near the boundary of moduli $\overline{\cQ(\textbf{d})}$, $J$ and the data obtained by gluing from the lower dimensional strata matches up to infinite order (see \cite{abousei}, \cite{generation}). Such a choice is possible thanks to Lemma \ref{lem:exists2}.

Consider pseudo-holomorphic maps $u:\cS_r^{f}\to \wh{\tilde T_0}\times\wh{M}$. We have:
\begin{lem}
For fixed Lagrangian boundary conditions and fixed asymptotic conditions on ends, the moduli of stable quilted pseudo-holomorphic strips is compact.
\end{lem}
\begin{proof}
It suffices to show the existence of a compact subset that contains all such strips. Fixing the asymptotic conditions and boundary conditions fixes the energy of $u$. Hence, as remarked, there is a bound to the number of annuli $\tilde A_i$ that the projection of such a curve can cross. Moreover, by $\Om^2$-assumption, the $\wh M$ components of these curves are contained in a compact subset of $\wh{M}$. Similarly, for each $\tilde B_i$, the parts of these curves that map to $\tilde B_i\times\wh{M}$ live over a compact subset of $\tilde B_i$. This finishes the proof.
\end{proof}
Hence, Gromov compactness holds, and it is standard to show that moduli spaces of such maps are cut out transversally for generic choices of almost complex structures. Therefore, by standard gluing arguments, the count of such maps define a left-right-right $\Om^2(T_\phi)$-$\Om(T_0)$-$\Om(M)$-trimodule structure on (\ref{eq:grdvsoftrimodule}) with a twisting among the last two components (or equivalently a functor (\ref{eq:omfunct})). Note that we define the structure maps to be zero when one of the conditions $L_p''>\dots >L_0''>L_0\times_\phi L_0'$, $L_0>\dots >L_m$, or $L_0'>\dots L_n'$ is not satisfied.
The trimodule equations/$A_\infty$-functor equations are satisfied thanks to Conditions (\ref{fd:trphi})-(\ref{fd:prdMcomponent}), as before. We denote both the trimodule and the functor (\ref{eq:omfunct}) by $\fM$. As before, the image of $L''\subset \wh{T_\phi}$ will be denoted by $\fM(L'')$ or $\fM_{L''}$. In the next section, we will show this induces a fully faithful functor $\cW^2(T_\phi)\to Bimod_{tw}(\cW(T_0),\cW(M))$.
\subsection{The modifications in the fully faithfulness argument}\label{subsec:algmodify}
In this section, we will show that (\ref{eq:omfunct}) induces a fully faithful functor 
\begin{equation}
\cW^2(T_\phi)\to Bimod_{tw}(\cW(T_0), \cW(M) )
\end{equation}
The proof of fully faithfulness is a modification of the one in Section \ref{subsec:defnkunnethearly}. Since the units of $\Om(T_0)$ and $\Om(M)$ do not have a geometric description, we replace them, by continuation elements. Insertion of continuation element becomes a Yoneda type quasi-isomorphism in the limit. Similarly, the composition described by Figure \ref{figure:shortfoldedcompose}, becomes a quasi-isomorphism in the limit, giving us what we desire. 

Observe that for $L_0''=L_0\times_\phi L_0'\in ob(\Om^2(T_\phi) )$, the twisted $\Om(T_0)$-$\Om(M)$-bimodule $\fM_{L''_0}$ is not isomorphic to twisted Yoneda bifunctor. Instead, it is strictly isomorphic to a truncated version of twisted Yoneda bimodule. In other words,
\begin{equation}
\fM_{L_0''}(L,L')=\begin{cases}
(h_{L_0}\otimes_{tw} h_{L_0'})(L,L'),&\text{for } L_0>L\text{ and } L_0'>L'\\
0,&\text{otherwise} 
\end{cases}
\end{equation}
The structure maps of the bimodule $\fM_{L_0''}$ are obtained by restriction from $h_{L_0}\otimes_{tw} h_{L_0'}$. This follows in the same way as Lemma \ref{lem:mequalsyoneda}. 
Notice, one cannot define a Yoneda map
\begin{equation}
hom_{\Om-\Om}(\fM_{L_0''},\fM_{L_1''})\to \fM_{L_1''}(L_0,L_0')
\end{equation}
by unit insertion (here we use $hom_{\Om-\Om}$ as an abbreviation for the homomorphisms in $Bimod_{tw}(\Om(T_0),\Om(M) )$). However, one still has canonical elements in $\fM_{L_0''}(L_0^-,L_0'^-)$ given by the tensor product of continuation elements of positive linear isotopies $L_0^-\to L_0$ and $L_0'^-\to L_0'$, where $L_0^-<L_0$ and $L_0'^-<L_0'$. 
Hence, we have maps
\begin{equation}\label{eq:preyoneda}
hom_{\Om-\Om}(\fM_{L_0''},\fM_{L_1''})\to \fM_{L_1''}(L_0^-,L_0'^-)
\end{equation}
given by the insertion of continuation elements. The left hand side of Figure \ref{figure:labelledcompforgpsy} (i.e. gluing with rigid discs) can be seen as a pictural representation of (\ref{eq:preyoneda}). 
These maps are well defined up to chain homotopy. We do not claim them to be quasi-isomorphisms; however, they induce quasi-isomorphisms after localization. To see this, first notice:
\begin{lem}
The image of $\fM_{L_0''}$ under the natural functor 
\begin{equation}
Bimod_{tw}(\Om(T_0),\Om(M)  )\to Bimod_{tw}(\cW(T_0),\cW(M))
\end{equation}
is quasi-isomorphic to twisted Yoneda bimodule $h_{L_0}\otimes_{tw}h_{L_0'}$. 
\end{lem}
\begin{proof}
One can realize $\fM_{L_0''}$ as $\tilde h_{L_0}\otimes_{tw}\tilde h_{L_0'}$, where $\tilde h_{L_0}$ and $\tilde h_{L_0'}$ denote truncated versions of Yoneda modules, i.e. they are equal to Yoneda module except at $L_0$, resp. $L_0'$, where they become $0$. For instance, $\tilde h_{L_0}$ satisfies 
\begin{equation}
\tilde h_{L_0}(L)=\begin{cases}
h_{L_0}(L)=CF(L,L_0), & \text{if } L_0>L\\ 0, & \text{if } L_0=L\\
h_{L_0}(L)=0, & \text{otherwise} 
\end{cases}
\end{equation}
with the obvious structure maps, and $\tilde h_{L_0'}$ is similar. Twisted exterior tensor product of $\tilde h_{L_0}$ and $\tilde h_{L_0'}$ is defined similar to $h_{L_0}\otimes_{tw} h_{L_0'}$. 

Observe that there is a natural submodule inclusion $\tilde h_{L_0}\to h_{L_0}$ with the cone supported only at $L_0$; therefore, the cone becomes acyclic, and $\tilde h_{L_0}$ and $h_{L_0}$ become quasi-isomorphic after localization $\Om(T_0)\to\cW(T_0)$. The same holds for $\tilde h_{L_0'}$ and $h_{L_0'}$. The formation of exterior tensor products commute with localization; hence, $\fM_{L_0''}\simeq \tilde h_{L_0}\otimes_{tw}\tilde h_{L_0'}$ and $h_{L_0}\otimes_{tw}h_{L_0'}$ become quasi-isomorphic.
\end{proof}
\begin{figure}
	\centering
	\includegraphics[height=3cm]{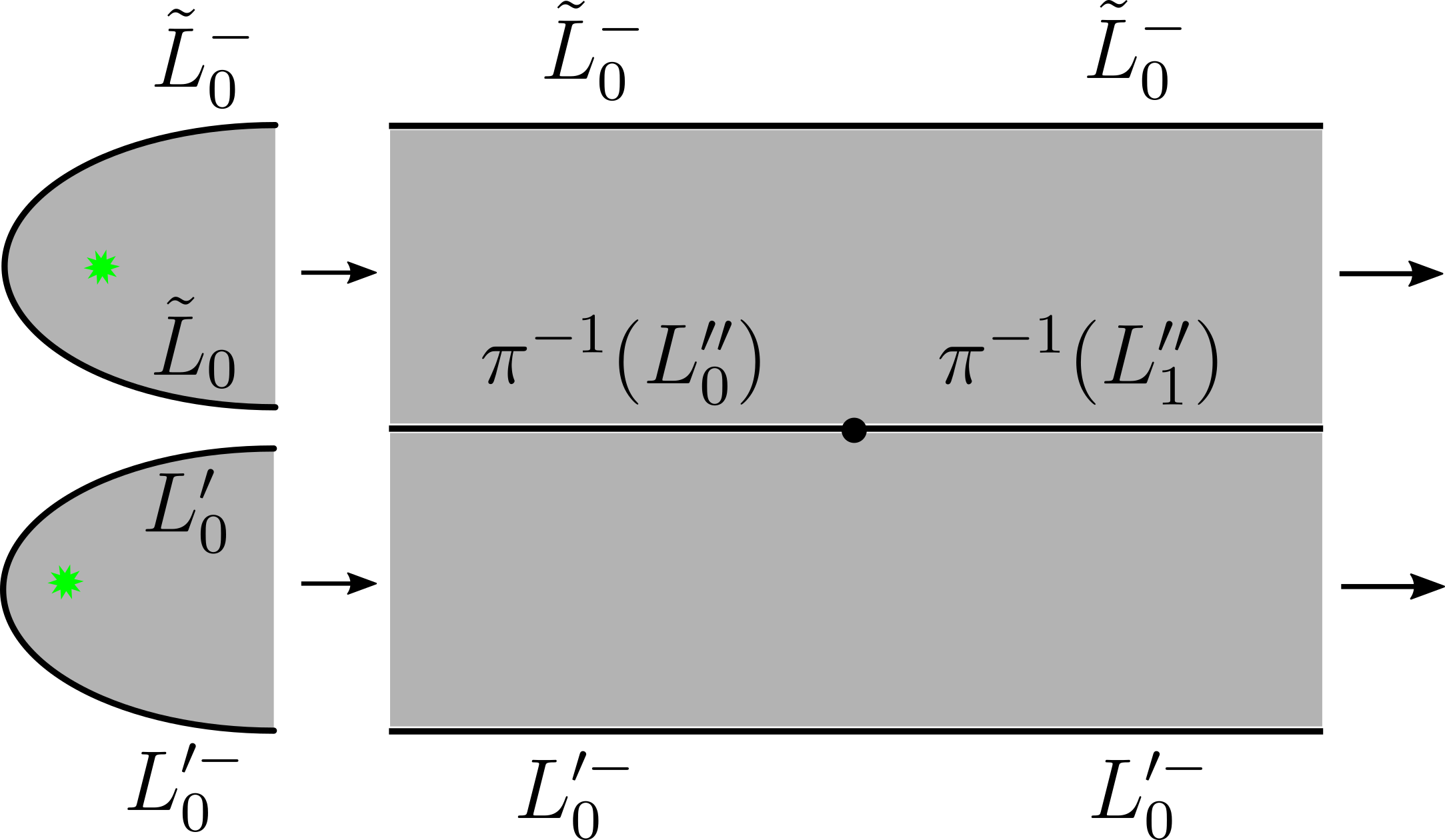}
	\caption{The composition of (\ref{eq:varyingcompose}) where gluing with rigid discs represent the map (\ref{eq:preyoneda})}
	\label{figure:labelledcompforgpsy}
\end{figure}
Denote the image of $\fM_{L_0''}$ by $\fM_{L_0''}^\cW$. In conclusion, the bottom horizontal arrow of the homotopy commutative diagram  
\begin{equation}\label{eq:diag1}
\xymatrix{hom_{\Om-\Om}(\fM_{L_0''},\fM_{L_1''})\ar[r]\ar[d]& \fM_{L_1''}(L_0^-,L_0'^-) \ar[d]\\ hom_{\cW-\cW}(\fM_{L_0''}^\cW,\fM_{L_1''}^\cW)\ar[r]^{\,\;\;\,\;\;\simeq}&\fM_{L_1''}^\cW(L_0^-,L_0'^-) }
\end{equation}
is a quasi-isomorphism by Lemma \ref{lem:twyoneda}. 

Consider the composition functor
\begin{equation}
\Om^2(T_\phi)\to Bimod_{tw}(\Om(T_0),\Om(M)  )\to Bimod_{tw}(\cW(T_0),\cW(M))
\end{equation}
To show this induces a functor $\cW^2(T_\phi)\to Bimod_{tw}(\cW(T_0),\cW(M))$, we need to show image of continuation elements are invertible. Let $L_1\to L_1^+$ and $L_1'\to L_1'^+$ be positive isotopies. There is a natural homotopy commutative diagram
\begin{equation}\label{eq:diag2}
\xymatrix{\Om^2(L_0'',L_1'')\ar[r]\ar[d]& hom_{\Om-\Om}(\fM_{L_0''},\fM_{L_1''})\ar[r]\ar[d]&\fM_{L_1''}(L_0^-,L_0'^-)\ar[d]\\ \Om^2(L_0'',L_1''^+)\ar[r]& hom_{\Om-\Om}(\fM_{L_0''},\fM_{L_1''^+})\ar[r]&\fM_{L_1''^+}(L_0^-,L_0'^-)
	}
\end{equation}
where $\Om^2(L_0'',L_1''):=\Om^2(T_\phi)(L_0'',L_1'')$, the leftmost vertical arrow is the multiplication with continuation element in $\Om^2(T_\phi)$, the middle vertical arrow is the composition with the image of this continuation element under the functor (\ref{eq:omfunct}) (denote this image by $g$), and the right vertical arrow is the lowest degree term of the twisted bimodule map $g$. If $L_1''>L_0''$, 
\begin{equation}\label{eq:fmopen1}
\fM_{L_1''}(L_0^-,L_0'^-)=\bigoplus_{r\in\Z}\Om_{T_0}(L_0^-,L_1)^r\otimes \Om_{M}(L_0'^-,\phi^{-r}(L_1'))
\end{equation}
\begin{equation}\label{eq:fmopen2}
\fM_{L_1''^+}(L_0^-,L_0'^-)=\bigoplus_{r\in\Z}\Om_{T_0}(L_0^-,L_1^+)^r\otimes \Om_{M}(L_0'^-,\phi^{-r}(L_1'^+))
\end{equation}
and the rightmost vertical arrow can be described geometrically as the count of rigid strips in $\wh{\tilde T_0}\times\wh{M}$ with upper boundary condition given by $\tilde L_0''^-:=\tilde L_0^-\times L_0'^-$ and with positively varying lower boundary condition from $\pi^{-1}(L_1'')$ to $\pi^{-1}(L_1''^+)$ (see Figure \ref{figure:cont2}). The almost complex structure is of product type on the left and right strip-like ends as well as on a neighborhood of the upper boundary, whereas it is of $\Om^2(T_\phi)$-type on a neighborhood of a point on the lower boundary (the variation of boundary condition happens in this neighborhood). This description simply follows from gluing rigid discs defining continuation element $L_1''\to L_1''^+$ in $\cW^2(T_\phi)$ with the folded strips defining structure maps of $\fM$ along the strip like end on the lower boundary (i.e. the seam before folding). 
\begin{figure}
	\centering
	\includegraphics[height=3cm]{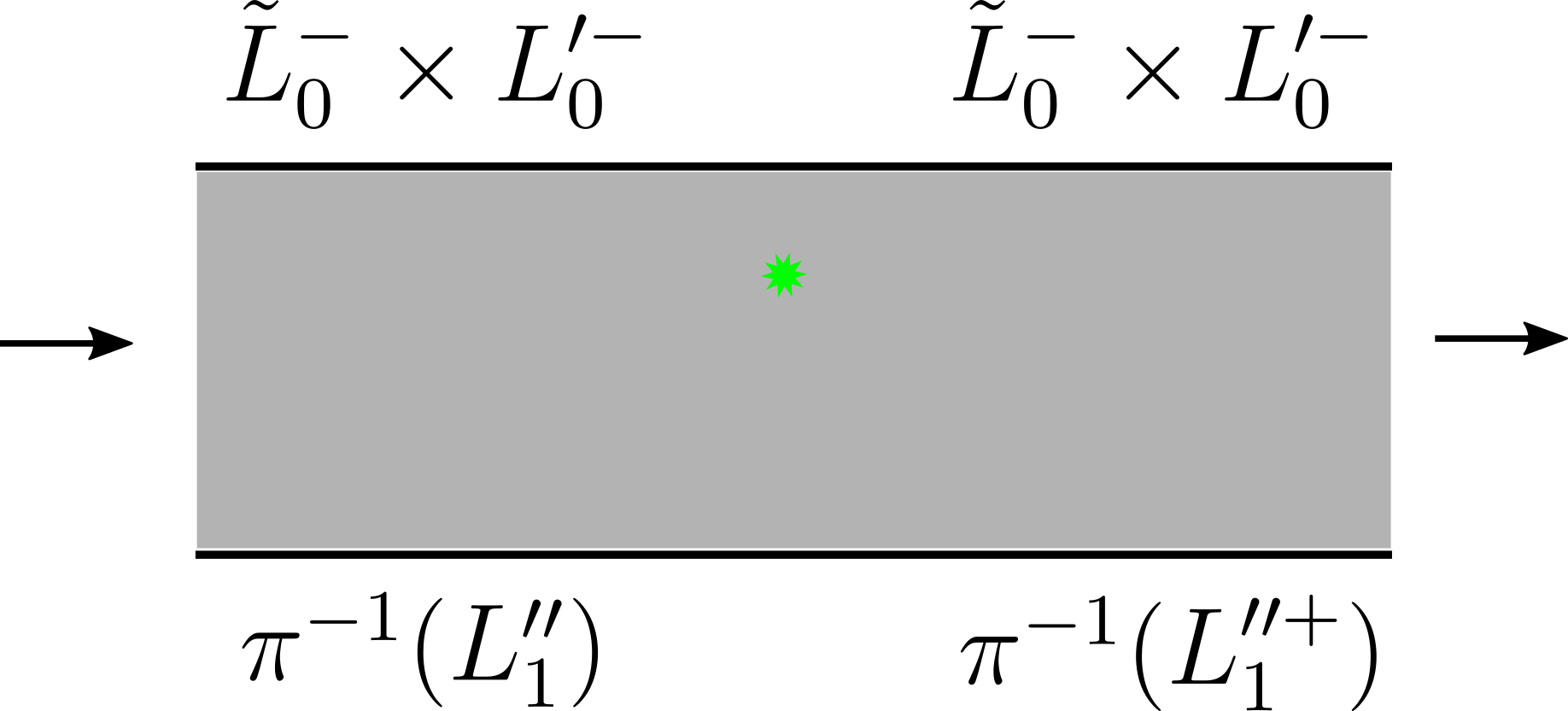}
	\caption{The rightmost vertical arrow in (\ref{eq:diag2})}
	\label{figure:cont2}
\end{figure}

Another natural map from (\ref{eq:fmopen1}) to (\ref{eq:fmopen2}) is given by multiplication by continuation elements $L_1\to L_1^+$ and $L_1'\to L_1'^+$ in each factor (more precisely, multiplication by $\phi^{-r}(c_{L_1',L_1'^+})$ in each of summand of the latter). The multiplication $\Om_{T_0}(L_0^-,L_1)\to \Om_{T_0}(L_0^-,L_1^+)$ can be described by a count of rigid strips in $\wh{\tilde T_0}$ with fixed upper boundary condition $\tilde L_0^-$ and with positively varying lower boundary condition from $\tilde L_1\langle -r\rangle\subset \pi^{-1}(L_1)$ to $\tilde L_1^+\langle -r\rangle\subset \pi^{-1}(L_1^+)$ (where $r$ is the degree in extra grading). There is a similar (indeed simpler) description of the map $\Om_{M}(L_0'^-,\phi^{-r}(L_1'))\to  \Om_{M}(L_0'^-,\phi^{-r}(L_1'^+))$ as well. Therefore, the map between the tensor products can be described by the count of rigid strips in $\wh{\tilde T_0}\times \wh{M}$ with the same varying boundary conditions, but with $\tr$-invariant, product type almost complex structure everywhere. Moreover, these two families of almost complex structures on $\wh{\tilde T_0}\times \wh{M}$ can be homotoped to each other while keeping conditions (\ref{fd:om2})-(\ref{fd:annulusproduct}) which ensure compactness. Therefore, the map induced by the lowest term of $g$ is homotopic to multiplication by continuation elements of $L_1\to L_1^+$ and $L_1'\to L_1'^+$ on each factor. As a result, the map from $\fM^\cW_{L_1''}(L_0^-,L_0'^-)$ to $\fM^\cW_{L_1''^+}(L_0^-,L_0'^-)$ is invertible due to localization at continuation elements. 
By (\ref{eq:diag1}), the induced map from $hom_{\cW-\cW}(\fM_{L_0''}^\cW,\fM_{L_1''}^\cW)$ to $hom_{\cW-\cW}(\fM_{L_0''}^\cW,\fM_{L_1''^+}^\cW)$ is also invertible. Hence, (\ref{eq:omfunct}) induces a functor 
\begin{equation}\label{eq:funcinduced}
\cW^2(T_\phi)\to Bimod_{tw}(\cW(T_0),\cW(M))
\end{equation}
\begin{figure}
	\centering
	\includegraphics[height=3cm]{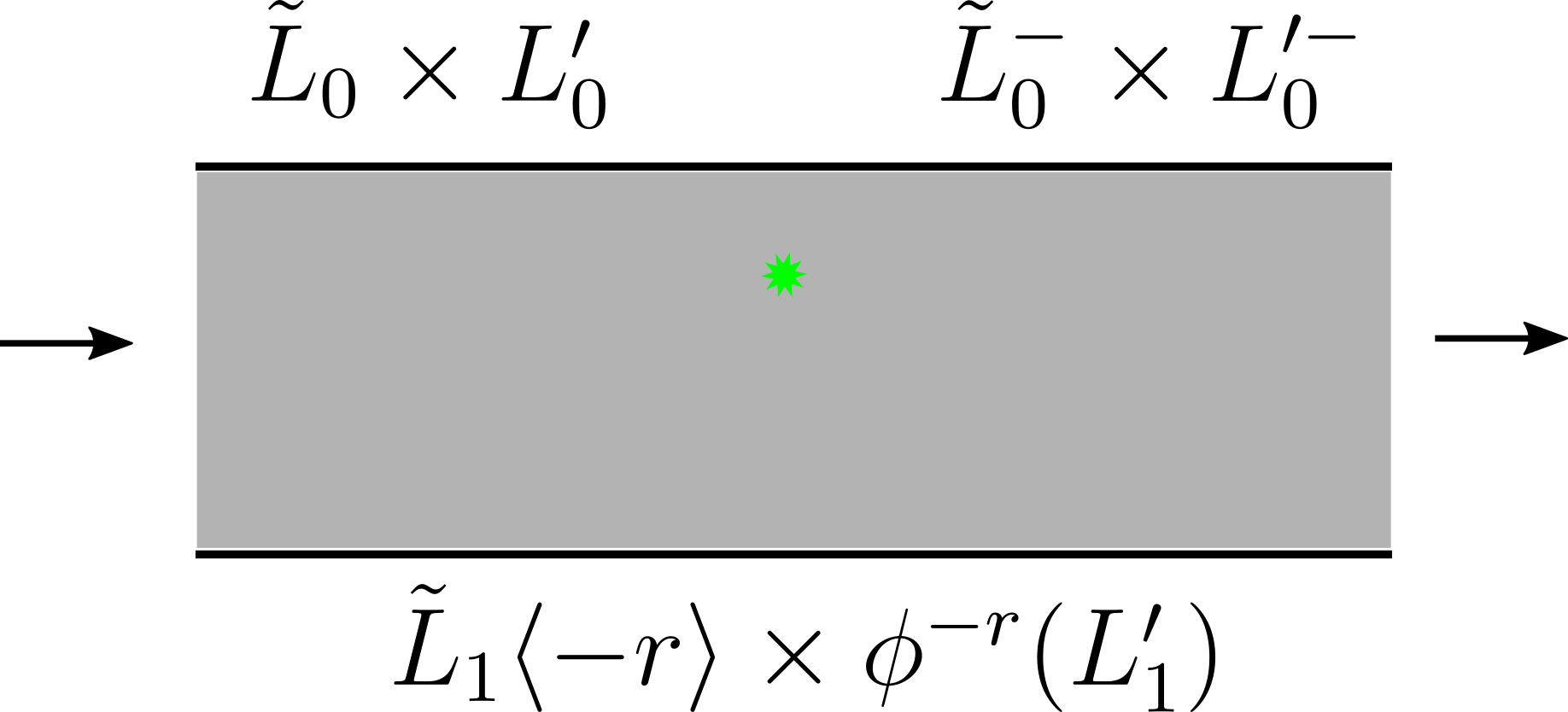}
	\caption{A geometric description of (\ref{eq:varyingcompose})}
	\label{figure:varyingcompose}
\end{figure}

To prove fully faithfulness, we start with a similar geometric description of the composition 
\begin{equation}\label{eq:varyingcompose}
\Om^2(L_0'',L_1'')\to hom_{\Om-\Om}(\fM_{L_0''},\fM_{L_1''})\to \fM_{L_1''}(L_0^-,L_0'^-)
\end{equation}
After folding and gluing, this composition can be described by a count of rigid strips as in Figure \ref{figure:shortfoldedcompose}, where the fixed boundary condition $L_0\times L_0'$ is replaced by a (positively) varying boundary condition from $\tilde L_0^-\times L_0'^-$ to $\tilde L_0\times L_0'$ in the counter clockwise direction and the boundary condition $L_1\times L_1'$ is replaced by $\tilde L_1\langle -r\rangle\times \phi^{-r}(L_1')$ for an integer $r$ depending on the input (the positive isotopy from $\tilde L_0^-\times L_0'^-$ to $\tilde L_0\times L_0'$ is just the product of isotopies from $\tilde L_0^-$ to $\tilde L_0$ and $L_0'^-$ to $L_0'$). This is the same as count in Figure \ref{figure:varyingcompose} (as before the green asterisk only rigidifies the count, so its location does not matter).
More precisely, one puts pull-back of $\Om^2(T_\phi)$-type data to the input end and translation invariant product type data to the output end. We do not claim that (\ref{eq:varyingcompose}) is invertible; however, it admits a ``pre-inverse'':
\begin{equation}\label{eq:preinv}
\fM_{L_1''}(L_0^-,L_0'^-)=\bigoplus_{r\in\Z}\Om_{T_0}(L_0^-,L_1)^r\otimes \Om_{M}(L_0'^-,\phi^{-r}(L_1'))\to \Om^2(L_0''^{--},L_1'')
\end{equation}
for any positive isotopy $L_0''^{--}<L_0'^-$. This pre-inverse is given by the same count of strips with the type of data on both ends reversed, and with upper boundary varying from $\tilde L_0^{-}\times L_0'^{-}$ to $\tilde L_0^{--}\times L_0'^{--}$. See upper right strip in Figure \ref{figure:preinvandcomposition}. Moreover, the left and right compositions of (\ref{eq:varyingcompose}) and (\ref{eq:preinv}) can also be described by a similar count of strips; however, one needs to choose the almost complex structures to be of the same type on both ends. For instance, the composition map
\begin{equation}\label{eq:componeway}
\Om^2(L_0'',L_1'')\to \Om^2(L_0''^{--},L_1'')
\end{equation}
can be described by a strip count as in Figure \ref{figure:preinvandcomposition} (as the bottom strip after gluing). 
\begin{figure}
	\centering
	\includegraphics[height=3.5cm]{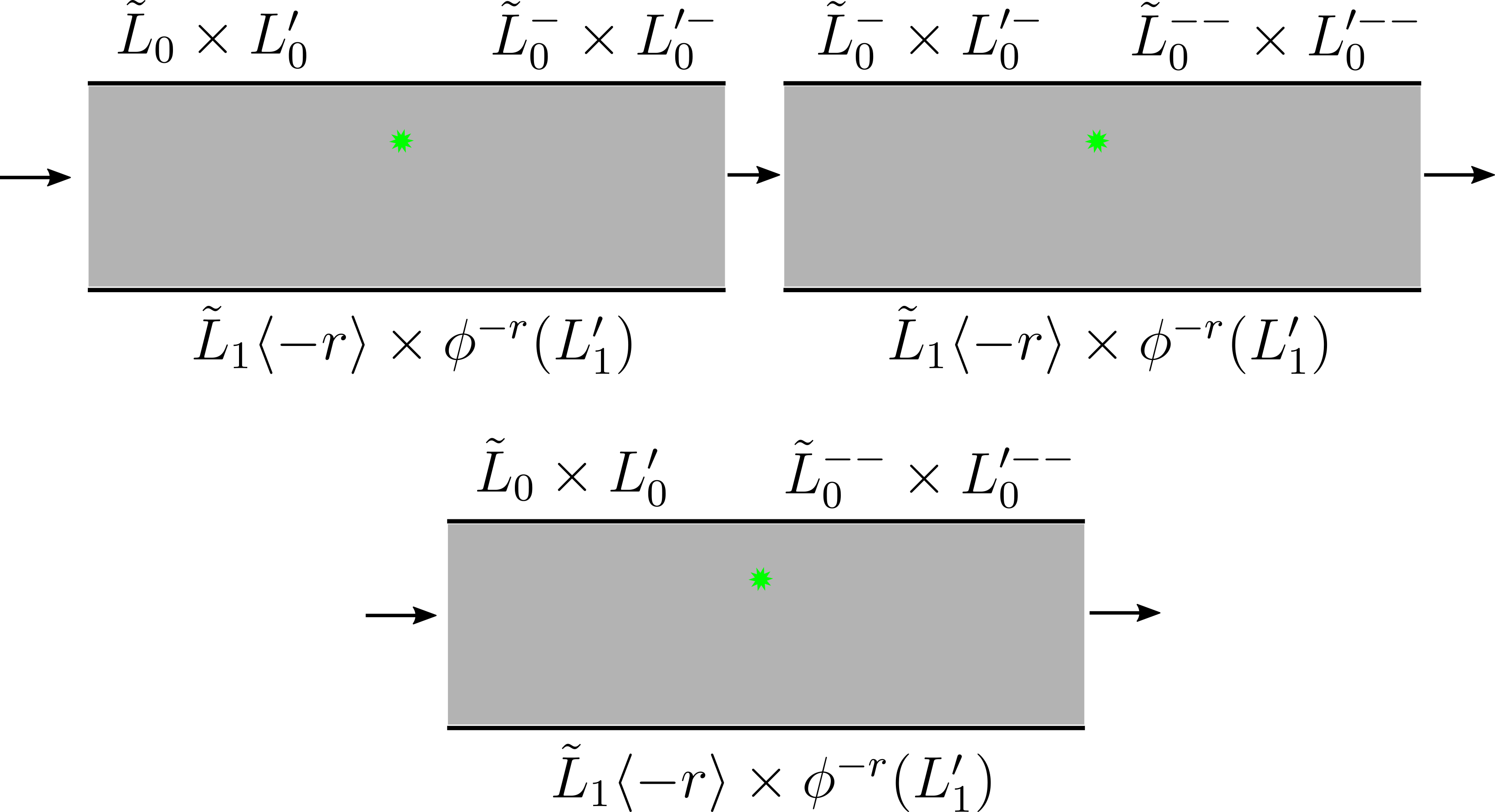}
	\caption{The composition (\ref{eq:componeway}) before and after gluing}
	\label{figure:preinvandcomposition}
\end{figure}
Moreover, the family of almost complex structures can be homotoped to $\Om^2(T_\phi)$-type while keeping conditions (\ref{fd:om2})-(\ref{fd:annulusproduct}) which ensure compactness. Therefore, (\ref{eq:componeway}) is actually homotopic to multiplication by a continuation element $L_0''^{--}\to L_0''$. A similar conclusion holds for the other composition 
\begin{equation}
\bigoplus_{r\in\Z}\Om_{T_0}(L_0,L_1)^r\otimes  \Om_{M}(L_0',\phi^{-r}(L_1'))\to \bigoplus_{r\in\Z}\Om_{T_0}(L_0^{--},L_1)^r\otimes \Om_{M}(L_0'^{--},\phi^{-r}(L_1'))
\end{equation}
The composition 
\begin{equation}\label{eq:wrpvaryingcompose}
\cW^2(L_0'',L_1'')\to hom_{\cW-\cW}(\fM^\cW_{L_0''},\fM^\cW_{L_1''})\to \fM^\cW_{L_1''}(L_0^-,L_0'^-)
\end{equation}
of (\ref{eq:funcinduced}) with the insertion of continuation element can be seen as the direct limit of compositions of (\ref{eq:varyingcompose}) as $L_1''\to \infty$. More precisely, by \cite[Lemma 3.37]{GPS1} the cohomology of $\cW^2(L_0'',L_1'')$, $\cW_{T_0}(L_0,L_1)$ and $\cW_M(L_0',L_1')$ can be described as the direct limit of cohomologies of $\Om^2(L_0'',L_1'')$, $\Om_{T_0}(L_0,L_1)$ and $\Om_M(L_0',L_1')$ respectively, where the maps in diagrams are given by left multiplication with continuation elements. For instance, \begin{equation}
H^*(\cW^2(L_0'',L_1'') )=\lim\limits_{L_1''\to\infty} H^*(\Om^2(L_0'',L_1''))
\end{equation}
Moreover, these maps commute in cohomology with the composition of (\ref{eq:varyingcompose}) and with (\ref{eq:preinv}). This can be seen using explicit geometric descriptions. For instance, consider the diagram
\begin{equation}\label{eq:somediag}
\xymatrix{\Om^2(L_0'',L_1'') \ar[r]\ar[d]& \fM_{L_1''}(L_0^-,L_0'^-)\ar[d] \\ \Om^2(L_0'',L_1''^+) \ar[r]& \fM_{L_1''^+}(L_0^-,L_0'^-)}
\end{equation}
where the horizontal arrows are given by the composition of (\ref{eq:varyingcompose}); therefore, admit a geometric description as in Figure \ref{figure:varyingcompose} (for the bottom arrow, one replaces the labeling $\tilde L_1\langle -r\rangle \times \phi^{-r}(L_1')$ by $\tilde L_1^+\langle -r\rangle \times \phi^{-r}(L_1'^+)$ to be precise). The vertical arrows are composition with the respective continuation elements and admit a geometric description similar to Figure \ref{figure:cont2}. We implicitly use the identification
\begin{equation}
\fM_{L_1''}(L_0^-,L_0'^-)=\bigoplus_{r\in\Z}\Om_{T_0}(L_0^-,L_1)^r\otimes \Om_{M}(L_0'^-,\phi^{-r}(L_1'))
\end{equation} and its analogue for $L_1''^+$ when composing with the continuation elements of $L_1\to L_1^+$ and $L_1'\to L_1'^+$. After gluing, the two possible compositions 
\begin{equation}
\xymatrix{\Om^2(L_0'',L_1'')\to \fM_{L_1''^+}(L_0^-,L_0'^-)}
\end{equation}
in diagram (\ref{eq:somediag}) can be described by a count of rigid strips with the same boundary conditions and with homotopic choices of almost complex structures; therefore, (\ref{eq:somediag}) commutes in cohomology. The compatibility of the pre-inverse (\ref{eq:preinv}) with multiplication by continuation elements follows from the same idea.

As a result, the ``pre-inverse'' induces a map in the other direction
\begin{equation}\label{eq:wrppreinv}
\fM^\cW_{L_1''}(L_0^-,L_0'^-)\to \cW^2(L_0''^{--},L_1'')
\end{equation}
Since, the composition of (\ref{eq:wrpvaryingcompose}) with (\ref{eq:wrppreinv}) can now be seen as a direct limit of multiplication with continuation elements, we conclude it is invertible (since continuation elements are invertible in the limit). Same holds in the other direction as well. Therefore, the composition of (\ref{eq:wrpvaryingcompose}) is invertible. Moreover, since the localization $\fM_{L_1''}^\cW$ of the truncated version of twisted Yoneda bimodule is quasi-isomorphic to the twisted Yoneda bimodule, insertion of continuation element is now invertible. Therefore, 
\begin{equation}
\cW^2(L_0'',L_1'')\to hom_{\cW-\cW}(\fM^\cW_{L_0''},\fM^\cW_{L_1''})
\end{equation}
is also invertible; hence, $\fM^\cW$ as a functor from $\cW^2$ to $Bimod_{tw}(\cW(T_0),\cW(M) )$ is fully faithful in cohomology. 

The essential image of $\fM^\cW$ is spanned by $\fM^\cW_{L_1''}$, which are quasi-isomorphic to twisted Yoneda bimodules; therefore, $\cW^2$ is quasi-equivalent to the ``twisted tensor product'' of $\cW(T_0)$ and $\cW(M)$. Combined with Note \ref{note:w2equalsw}, this implies:
\begingroup
\def\thethm{\ref{mainthmsymp}}
\begin{thm}
	$\cW(T_\phi)$ is quasi-equivalent to twisted tensor product of $\cW(T_0)$ and $\cW(M)$.
\end{thm}
\addtocounter{thm}{-1}
\endgroup

\section{Examples of symplectic manifolds satisfying Assumption \ref*{assumption:symp} and applications}\label{sec:examples}
In this section, we give a large class of examples satisfying Assumption \ref{assumption:symp}. We search for examples among Liouville manifolds with periodic Reeb flow since it is easier to compute the Conley-Zehnder indices. More specifically, we will confine ourselves to complements of smooth ample divisors. For Assumption \ref{assumption:symp}, we need
\begin{enumerate}
	\item Vanishing first and second Betti numbers
	\item Reeb orbits with sufficiently large degree
\end{enumerate}
Let us start by addressing (1):
\begin{lem}\label{lembet}
Let $X$ be a smooth and projective variety and $D\subset X$ be a smooth, connected hypersurface that is given as a transverse hyperplane section of a projective embedding. 
Further assume:
\begin{enumerate}
\item $b_1(X)=b_1(D)=0$
\item $b_2(X)=b_2(D)=1$
\end{enumerate}	Let $M=X\setminus ND\simeq X\setminus D$, where $ND$ is a tubular neighborhood of $D$. Then, $b_1(M)=b_2(M)=0$.
\end{lem}
\begin{proof}
	First note $H^*(X,D)\cong H^*(M,\partial M)$ by excision. Consider the long exact sequence 
	\begin{equation}
	H^0(X,D)\to H^0(X)\xrightarrow{\cong} H^0(D) 
	\rightarrow H^1(X,D)\to H^1(X)\xrightarrow{\cong} H^1(D) \atop
	\rightarrow H^2(X,D)\to H^2(X)\hookrightarrow H^2(D) \dots
	\end{equation}
($H^2(X)\to H^2(D)$ is injective since $H^2(X)$ is one dimensional and this map carries a K\"{a}hler class on $X$ to one on $D$). Thus, $H^1(M,\partial M)=H^1(X,D)=0$, and $H^2(M,\partial M)=H^2(X,D)=0$. This implies $H^1(M)\to H^1(\partial M)$ is an isomorphism, and $H^2(M)\to H^2(\partial M)$ is injective by a similar long exact sequence. Hence, it is sufficient to prove $ H^1(\partial M)= H^2(\partial M)=0$.
	
	Consider Serre spectral sequence for the fibration $S^1\hookrightarrow  \partial M\rightarrow D$ given by \begin{equation}
	E_2^{pq}=H^p(D,\{H^q(S^1) \})\Rightarrow H^{p+q}(\partial M)
	\end{equation}
	where $\{H^q(S^1) \}$ denotes the local system formed by the cohomology of each fiber. $\partial M\to D$ is the circle bundle of a complex line bundle; thus, it is an oriented bundle and $\{H^1(S^1) \}$ is the trivial local system (i.e. constant sheaf). Same holds for $\{H^0(S^1) \}$ easily. 
	This implies $p=1,q=0$ and $p=1,q=1$ terms in the $E_2$-page vanish; thus, $H^1(\partial M)$ and $H^2(\partial M)$ can be obtained as the cohomology of 
	\begin{equation}
	H^0(D,H^1(S^1) )\xrightarrow{d_2}H^2(D,H^0(S^1) ) 
	\end{equation}
	where $d_2$ is the differential of $E_2$-page. Both groups are of rank $1$, and to finish the proof we only need $d_2\neq 0$.
	
	By assumption, there exist a projective embedding $X\hookrightarrow \mathbb P^N$ and a hyperplane $\mathcal H\subset \mathbb P^N$ such that $X\pitchfork\mathcal H$ and $D=X\cap \mathcal H$. Then, $N_{\mathcal H}\cong \Om_{\mathbb P^N}(\mathcal{H})|_{\mathcal H}$ restricts to $N_D\cong\Om_X(D)|_D$. Let $S_{\mathcal H}$ and $S_D\cong \partial M$ denote the corresponding circle bundles. There is a similar spectral sequence for $S_{\mathcal H}$ as well and we have a diagram
	\begin{equation}\label{blah2}
	\xymatrix{H^0(D,H^1(S^1) )\ar[r]^{d_2}&H^2(D,H^0(S^1) )\\
		H^0(\mathcal H,H^1(S^1) )\ar[r]^{d_2}\ar[u]&H^2(\mathcal H,H^0(S^1) )\ar[u]  }
	\end{equation}
	by naturality. In previous considerations, we can replace $X$ by $\mathbb P^N$ and $D$ by $\mathcal H$, and conclude 
	\[H^0(\mathcal H,H^1(S^1) )\xrightarrow{d_2} H^2(\mathcal H,H^0(S^1) ) \] computes $H^1(S_{\mathcal{H}})$ and $H^2(S_{\mathcal{H}})$. But we know $H^1(\mathbb{P}^N\setminus \mathcal H)\cong H^1(S_{\mathcal H})$ (we proved $H^1(M)\cong H^1(\partial M)$, apply this to the case $X=\mathbb{P}^N,D=\mathcal H$). Hence, $H^1(S_{\mathcal H})=0$. A simpler way to see this is: topologically $\Om_{\mathbb P^N}(\mathcal{H})|_{\mathcal H}\cong \Om_{\mathbb P^N}(-\mathcal{H})|_{\mathcal H}$ which is just the circle bundle of the tautological bundle over $\mathcal H$. It is easy to see this is homeomorphic to $$S^{2N-1}\to S^{2N-1}/S^1$$ hence the total space has vanishing $H^1$.
	
	In summary, the lower horizontal arrow in (\ref{blah2}) cannot vanish. The right vertical arrow cannot vanish since it carries at least one K\"ahler class to one on $D$. Left vertical arrow is an isomorphism by the connectedness of $D$. All the groups in (\ref{blah2}) are $1$-dimensional; thus, the composition does not vanish and neither does the upper horizontal $d_2$. This completes the proof.
\end{proof}
To make sure $SH^1(M)=SH^2(M)=0$, we need Reeb orbits to be of sufficiently large degree. More precisely, we will use the spectral sequence in \cite{biasedview} whose $E_1$-page is given by:
\begin{equation}\label{eq:paulspectral}
E^{pq}_1=\begin{cases}
  H^{q}(M)&p=0 \\ H^{p(1-\mu)+q}(\partial M)&p<0\\ 0&p>0
\end{cases}
\end{equation}
and which converges to symplectic cohomology of $M$ (note that we made a degree shift on $SH^*(M)$ by $n$ so that $SH^*(M)\cong HH^*(\cW(M))$ by \cite[Theorem 1.1]{sheelthesis}). 
Here, $\mu\in 2\Z$ is a Conley-Zehnder type index defined in \cite{biasedview}. 

Assume we are in the setting of Lemma \ref{lembet} and $K_X=\Om(mD)$. Then we have:
\begin{lem}\label{lem:mu}
$\mu=-2m-2$.
\end{lem}
\begin{proof}
Recall how $\mu$ is defined: given a Liouville domain $M$ with $1$-periodic Reeb flow at its contact boundary, choose a trivialization of $K_M$. Let $x$ be a Reeb orbit on $\partial M$. We obtain a trivialization of $x^*TM$ as a symplectic bundle. Hence, the Reeb flow defines a path in $Sp(2n)$ and the class of this path in $\pi_1(Sp(2n))\cong \Z$ is the number $\mu/2$. 

To compute $\mu$, identify a neighborhood of $D$ with a neighborhood of zero section of $\mathcal N_D$ (the holomorphic normal bundle of $D$). Let $d\in D$ and let $F$ denote a fiber of $\mathcal N_D$. Then, $T_X|_F\cong F\times (T_{D,d}\oplus \mathcal N_{D,d})=:E$ as a symplectic bundle. In other words, it is the trivial bundle with fiber $T_{D,d}\oplus \mathcal N_{D,d}$. Under this trivialization, the circle action induced by Reeb vector field is 
\begin{equation}
\xymatrix{F\times (T_{D,d}\oplus \mathcal N_{D,d})\ar[r]& F\times (T_{D,d}\oplus \mathcal N_{D,d})\\ (\alpha,v,v')\ar@{|->}[r]& (z^{-1}\alpha,v,z^{-1}v')   }
\end{equation}
If we trivialize using a section $\Omega$ of $K_X$, the section can be chosen to have vanishing order $m$ along $D$. Hence, the dual section has vanishing order $(-m)$ along $D$. Therefore, the map 
\begin{equation}
\xymatrix{f:SF\times (T_{D,d}\oplus \mathcal N_{D,d}) \ar[r] & SF\times (T_{D,d}\oplus \mathcal N_{D,d})=E|_{SF}
	\\ (\alpha,v,v') \ar@{|->}[r]& (\alpha,v,\alpha^{-m} v')}
\end{equation} is the new trivialization (symplectic trivialization obtained by using $\Omega$). Here, $SF$ is the unit circle in $F$. The right hand side is considered to be the restricted bundle, and the left hand side is considered to be a trivial bundle, and the trivialization map is the framing (if $E'$ is a vector bundle over $F'$, then a trivialization is a map $F'\times V\to E'$ for a vector space $V$). The $S^1$-action is by $z^{-1}$ on the right hand side. In other words, \begin{equation}z:(\alpha,v,\alpha^{-m}v')\mapsto (z^{-1}\alpha,v,z^{-1}\alpha^{-m}v')
\end{equation}
or \begin{equation}
f(\alpha,v,v')\mapsto f(z^{-1}\alpha,v,z^{-m-1}v')
\end{equation}
More diagrammatically
\begin{equation}
\xymatrix{SF\times (T_{D,d}\oplus \mathcal N_{D,d}) \ar[r]\ar[d] & E|_{SF} \ar[d] \\SF\times (T_{D,d}\oplus \mathcal N_{D,d}) \ar[r] & E|_{SF} }
\;\;\;\; 
\xymatrix{  (\alpha,v,v') \ar@{|->}[r]^f\ar@{|->}[d]& (\alpha,v,\alpha^{-m} v')\ar@{|->}[d] \\  (z^{-1}\alpha,v,z^{-m-1}v') & (z^{-1}\alpha,v,z^{-1}\alpha^{-m} v')\ar@{|->}[l]_{f^{-1}  } }
\end{equation}
Hence, the path in $U(n)\subset Sp(2n)$ induced by the circle action is \begin{equation}
\xymatrix{S^1\ar[r]& U(T_{D,d}\oplus \mathcal N_{D,d})\\ z\ar@{|->}[r]&  id_{T_{D,d}}\oplus z^{-m-1}id_{\mathcal N_{D,d}} }
\end{equation}
If we compose this map with $\det$, we obtain a map of degree $-m-1$. Thus, $\mu/2=-m-1$, and $\mu=-2m-2$.
\end{proof}
Combining the spectral sequence \eqref{eq:paulspectral} and Lemmas \ref{lembet} and \ref{lem:mu}, we obtain:
\begin{cor}
Assume $m>0$. Then, $SH^*(M)$ vanishes for $*<0$ or $*=1,2$ and it is 1 dimensional for $*=0$.
\end{cor}
We also have:
\begin{lem}
Assume $m>0$. Then, $\cW(M)$ is proper in each degree, i.e. $HW(L,L')$ is finite dimensional in each degree and bounded below for any pair of objects of $\cW(M)$.
\end{lem}
\begin{proof}
Consider the generating subcategory of $\cW(M)$ spanned by cocores. One can arrange the cocores to be cylindrical; hence, their intersections with the contact boundary (with periodic Reeb flow) are Legendrian submanifolds. Let $L_0$ and $L_1$ be two such Lagrangians. The generators of $CW(L_0,L_1)$ are given by 
\begin{enumerate}
\item finitely many chords in the interior
\item finitely many chords in the contact end of length less than $1$. Note each such chord lives on a unique Reeb orbit because of periodicity
\item chords obtained by concatenating a chord  of length less than $1$ with the Reeb orbit it lives on $k$ times (where $k\in\Z_{\geq 0}$)
\end{enumerate}
Let $x$ be an orbit living on a Reeb orbit, and assume $y$ is obtained by concatenating the Reeb orbit $k$-times. A straightforward calculation shows $deg(y)=deg(x)-k\mu$, where $\mu$ is as before (i.e. as in \cite{biasedview}). By Lemma \ref{lem:mu}, $\mu=-2m-2<0$. Thus, the degree of $y$ grows as one increases $k$. In other words, there are only finitely many generators of $CW(L_0,L_1)$ of degree less than $d$ (for every $d$). This completes the proof.
\end{proof}
Combining the results of this section, we have:
\begin{prop}\label{prop:example}
Assume the pair $(X,D)$	satisfies the assumptions of Lemma \ref{lembet} and $K_X\cong \Om(mD)$ such that $m>0$. Then $M$ (the Liouville domain corresponding to $X\setminus D$) satisfies Assumption \ref{assumption:symp}.
\end{prop}
\begin{cor}\label{cor:hypersurface}
Let $X$ be a smooth hypersurface in $\C\mathbb P^{n+1}$ (for $n\geq 4$) of degree at least $n+3$ (i.e. of general type) and $D$ be a transverse hyperplane section. Let $\wh M=X\setminus D$. Then $M$ satisfies Assumption \ref{assumption:symp}. 
\end{cor}
\begin{proof}
This follows from Lefschetz hyperplane theorem and Proposition \ref{prop:example}.
\end{proof}
\begin{rk}As commented in the Section \ref{sec:intro2}, powers of Dehn twists act non-trivially on $\cW(M^{2n})$, when $n>1$; hence giving us applications of the Theorem. However, the least trivial examples are when $\phi$ is (pseudo-)isotopic to identity relative to $\partial M$. We are not aware of such examples when $n=dim(M)/2$ is odd, but powers of Dehn twists give such examples when $n$ is even. Indeed, the order of a Dehn twist in mapping class group divides $4|\Phi_{2n+1}|$, where $\Phi_{2n+1}$ is the group of homotopy spheres of dimension $2n+1$ (see \cite{krylov1},\cite{krylovkauffman}).
\end{rk}
%
Now, we will show that $T_\phi$ and $T_0\times M$ cannot be distinguished by their symplectic cohomology for a large class of examples provided by Proposition \ref{prop:example}. More precisely:
\begin{lem}\label{lem:samesh}
Let $(X,D)$ be as in Proposition \ref{prop:example} and $\phi$ be an even power of a Dehn twist along a spherical Lagrangian in $M$. Assume $n=dim_\C(X)=dim_{\mathbb R}(X)/2$ is even. Then, $SH^*(T_\phi)\cong SH^*(T_0\times M)$ as vector spaces, if $m+1>n$. In particular, this holds if $X$ is an hypersurface in $\C\mathbb{P}^{n+1}$ ($n\geq 4$) of degree larger than $2n+1$.
\end{lem}
\begin{proof}
First, note that one can recover $SH^*(T_\phi)$ as a vector space from $SH^*(M)$ and action of $\phi$ on $SH^*(M)$. This follows for instance by combining \cite[Prop 5.13]{ownpaperalg}, \cite[Theorem 1.1]{sheelthesis}, and Theorem \ref{mainthmsymp}. Hence, it is sufficient to show that $\phi$ acts trivially on $SH^*(M)$ if $m+1>n$.

One obtains the spectral sequence (\ref{eq:paulspectral}) by using the length filtration on the Reeb orbits. Notice that $\phi$ acts trivially on $p=0$ terms (i.e. on $H^*(M)$) by Picard-Lefschetz formula, and it acts trivially on $p<0$ terms since it is compactly supported (and continuation maps defining $\phi$ action on $SC^*(M)$ are length decreasing).

Since $M$ is a Weinstein domain of dimension $2n$, it has the homotopy type of an $n$-dimensional CW complex, and its cohomology is supported in degree $0,\dots,n$. Hence, $p=0^{th}$ column of (\ref{eq:paulspectral}) is supported in degrees $0,\dots,n$. Similarly, the cohomology of $\partial M$ is supported in degrees $0,\dots,2n-1$; therefore, for $p<0$, $(p,q)$ term can be non-zero only if 
\begin{equation}
2n-1\geq (1-\mu)p+q\geq 0
\end{equation}
which is equivalent to 
\begin{equation}
2n-1+\mu p\geq p+q\geq \mu p
\end{equation}
In other words, $p^{th}$ column is supported in degree $\mu p,\dots, 2n-1+\mu p$. By assumption, $\mu(p-1)>(2n-1+\mu p)+1$; hence, terms of $(p-1)^{th}$ column and $p^{th}$ column do not interact. Same holds with $(-1)^{th}$ column and $0^{th}$ column as well. 
Hence, the spectral sequence degenerates in $E_1$-page and the action of $\phi$ on each term is trivial. This implies that $\phi$ acts trivially on $SH^*(M)$ (in summary, one can filter the complex $SC^*(M)$ by length and the action of $\phi$ is trivial on the cohomology of associated graded. Moreover, orbits of different length differ at least by degree $2$, implying the desired result). 
\end{proof}
\begin{rk}
Presumably, when the degree of the hypersurface $X$ is sufficiently large, $SH^*(T_\phi)$ and $SH^*(T_0\times M)$ agree as BV-algebras as well. Indeed, we strongly believe that for any finite set of $BV_\infty$-operations, one can increase the degree of the hypersurface to produce examples where $SH^*(T_\phi)$ and $SH^*(T_0\times M)$ are isomorphic with an isomorphism respecting these operations (for instance, one can produce examples where symplectic cohomologies are the same as $A_n$-algebras). The rationale is that we believe one can recover $SH^*(T_\phi)$ as an $A_n$-algebra (or whichever finite set of operations we are considering), from the $A_n$-algebra $SH^*(M)$ and the action of $\phi$ as an $A_n$-algebra map; therefore, as long as $\phi$ acts trivially as an $A_n$-algebra map, $SH^*(T_\phi)$ and $SH^*(T_0\times M)$ should coincide. In the exotic examples above, $\phi$ acts trivially on $H^*(M)$ and it acts as identity on the contact boundary. By assuming the degree of $X$ is sufficiently large, we can create a degree gap between the Reeb chords of positive length and the cohomology of $M$ that eliminates the possibility of interaction among corresponding elements of the symplectic cohomology under finitely many higher maps of $\phi$. In other words, $\phi$ acts trivially as an $A_n$-algebra map, and $SH^*(T_\phi)$ and $SH^*(T_0\times M)$ coincide.

We believe one can still recover the $BV_\infty$-structure on $SH^*(T_\phi)$ out of $SH^*(M)$ and $\phi$, but we do not know how to produce examples where $\phi$ acts trivially, as a $BV_\infty$-map.
\end{rk}

\appendix
\section{Proof of Theorem \ref*{mainthmsymp} using the gluing formula for wrapped Fukaya categories}\label{sec:appendixgluing}
One can give an alternative proof of Theorem \ref{mainthmsymp} using the gluing formula in \cite{GPS2}. The proof is easy after the algebraic setup given in Section \ref{sec:twtensor}, and we sketch this proof in this appendix. A notational remark: in \cite{GPS1} Liouville sectors are defined with their infinite ends; however, we omit the completions from the notation throughout this section similar to Section \ref{subsec:extragr} (i.e. we write $M$ instead of $\wh{M}$, $T_0$ instead of $\wh{T_0}$ etc.).

Recall the notation from Section \ref{subsec:extragr}: $T$ denotes the $1$-handle that is shown in yellow in Figure \ref{figure:handletorus} and $N$ denotes $\overline{T_0 \setminus T}$ (more precisely, one has to consider the completions). See also Figure \ref{figure:cuttori}. As a sector, $T$ is isomorphic to $T^*[0,1]$ and $\cW(T)\simeq \C$. Similarly, $N$ is equivalent to a cylinder with one stop on each boundary component and $\cW(N)$ is derived equivalent to $D^b(Coh(\mathbb{P}^1) )$. To calculate $\cW(\wh{T_\phi})$, decompose $T_\phi$ into two sectors $T\times M$ and $N\times M$. In other words, $T_\phi= T\times M\cup N\times M$ and these subsectors intersect on $(T^*[0,1]\times M)\sqcup (T^*[0,1]\times M)=T^*[0,1]\times(M\sqcup M)$. Since $M$ is a Weinstein domain, $M\sqcup M$ and the horizontal completions of $T\times M$, $N\times M$ are Weinstein. Moreover, $T_\phi$ is a Liouville domain (as opposed to a more general Liouville sector). Therefore, the assumptions of \cite[Theorem 1.20]{GPS2} are satisfied, and one has a homotopy push-out diagram similar to (\ref{eq:gluediagsymp}):
\begin{equation}\label{eq:gluediagtorussymp}
\xymatrix{ \cW(N\times M)\ar[r]& \cW(T_\phi)\\ \cW(T^*[0,1]\times M)\coprod \cW(T^*[0,1]\times M)\ar[u]\ar[r]& \cW(T^*[0,1]\times M)\simeq \cW(T\times M)\ar[u]  }
\end{equation}
Hence, one has a homotopy coequalizer diagram 
\begin{equation}\label{eq:inclwrapcop}
\cW(T^*[0,1]\times M)\rightrightarrows \cW(N\times M)\to \cW(T_\phi)
\end{equation}
The map $\cW(N\times M)\to \cW(T_\phi)$ is induced by the inclusion and the maps $\cW(T^*[0,1]\times M)\rightrightarrows \cW(N\times M)$ are induced by $j_0\times 1_M$ and $j_1\times \phi$ (recall $j_0,j_1$ were used to denote both the inclusion maps from $T$ into $N$ shown in Figure \ref{figure:cuttori} and the functors induced by these inclusions).
By the K\"unneth theorem \cite[Theorem 1.5]{GPS2}, $\cW(T^*[0,1]\times M)\simeq \cW(T^*[0,1])\otimes\cW( M)$ and $\cW(N\times M)\simeq \cW(N)\otimes \cW(M)$. Note that the Weinstein property of $M$ is needed for the K\"unneth map to be essentially surjective as well. Under these quasi-equivalences (\ref{eq:inclwrapcop}) can be identified with the diagram 
\begin{equation}\label{eq:inclwrapcop2}
\cW(T^*[0,1])\otimes \cW(M)\rightrightarrows \cW(N)\otimes \cW(M)\to \cW(T_\phi)
\end{equation}
where the arrows are $j_0\otimes 1_{\cW(M)}$, $j_1\otimes \phi$. 
Recall that $\cW(T^*[0,1])\simeq \C$ and $\cW(N)\simeq \Om(\mathbb{P}^1)_{dg}$ (see Example \ref{exmp:wt=c} and Example \ref{exmp:wn=p1}). Moreover, 
$j_0$ and $j_1$ turn into $i_0$ and $i_\infty$ under these identifications. In summary, we have a homotopy coequalizer diagram 
\begin{equation}\label{eq:coeqsympl}
\cW(M)\rightrightarrows \Om(\mathbb{P}^1)_{dg}\otimes\cW(M)  \to \cW(T_\phi) 
\end{equation}
where the arrows $\cW(M)\rightrightarrows \Om(\mathbb{P}^1)_{dg}\otimes \cW(M)$ are given by $i_0\otimes 1_{\cW(M)}$ and $i_\infty\otimes \phi$. As the situation is symmetric, one can swap $i_0$ and $i_\infty$ or replace $\phi$ by $\phi^{-1}$ (different identifications may lead to this).

Let us now describe $M_\phi$ as a similar homotopy pushout. Let $\A\simeq\cW(M)$. Recall that the diagram (\ref{eq:acommdiag}) and the quasi-isomorphism (\ref{eq:hocolimtot0}) from the homotopy coequalizer to $\Om(\Tt_0)_{dg}$ are strictly $\tr$ equivariant. Hence, there exists a quasi-equivalence 
\begin{equation}
\big(hocolim\big(Pt_\infty\rightrightarrows \Om(\mathbb{P}^1\times \Z)_{dg}\big)\otimes\A\big)\#\Z \xrightarrow{\simeq} \big(\Om(\Tt_0)_{dg}\otimes \A\big)\#\Z=M_\phi
\end{equation}
Following \cite{GPS2}, we described the homotopy coequalizer as a localization of the Grothendieck construction (see (\ref{eq:hocolimbilbil})). Hence, $M_\phi$ is equivalent to $((C^{-1}\cG r)\otimes\A)\#\Z$. It is easy to see that localization commutes with tensoring with $\A$, i.e. 
\begin{equation}
(C^{-1}\cG r)\otimes\A\simeq (C\otimes 1)^{-1}(\cG r\otimes\A)
\end{equation}
where $C\otimes 1$ is the set of morphisms $\{(c\otimes 1_{L'}):c\in C,L'\in ob(\A) \}$ (in the absence of strict units, choose a $\phi$-equivariant set of cohomological units). Moreover, as $C$ is $\tr$-invariant, localization commutes with smash product as well. Hence, 
\begin{equation}\label{eq:amtlocalgroth}
M_\phi\simeq (C\otimes 1)^{-1}\big((\cG r\otimes\A\big)\#\Z)
\end{equation}
It is easy to see that $\cG r\otimes \A$ is the Grothendieck construction for the diagram 
\begin{equation}
Pt_\infty\otimes \A\rightrightarrows \Om(\mathbb P^1\times\Z)_{dg}\otimes \A
\end{equation}
and $(\cG r\otimes\A\big)\#\Z$ is the Grothendieck construction for 
\begin{equation}\label{eq:coeqwithA}
(Pt_\infty\otimes \A)\#\Z\rightrightarrows (\Om(\mathbb P^1\times\Z)_{dg}\otimes \A)\#\Z
\end{equation}
Hence, by (\ref{eq:amtlocalgroth}), $M_\phi$ is the homotopy coequalizer of the diagram (\ref{eq:coeqwithA}). $\Z$-action is still by $\tr\otimes \phi$; however, translation carries components of $Pt_\infty$, resp. $\mathbb{P}^1\times \Z$ to different components. Hence, 
\begin{equation}\label{eq:identinfsmash}
(Pt_\infty\otimes \A)\#\Z\simeq \A\text{ and }(\Om(\mathbb P^1\times\Z)_{dg}\otimes \A)\#\Z\simeq \Om(\mathbb{ P}^1)_{dg}\otimes\A
\end{equation}
If there were no $\A$ in (\ref{eq:coeqwithA}), the arrows would become $i_0$ and $i_\infty$ under the identification (\ref{eq:identinfsmash}), as remarked in Section \ref{subsec:extragr}. On the other hand, as $\Z$-action is given by $\tr\otimes\phi$, the arrows in (\ref{eq:coeqwithA}) become different under the identification (\ref{eq:identinfsmash}). More precisely, one of them becomes $i_0\otimes 1_\A$ and the other one becomes $i_\infty\otimes \phi$. Hence, we have a coequalizer diagram 
\begin{equation}\label{eq:coeqalgtorus}
\A\rightrightarrows\Om(\mathbb{P}^1)_{dg}\otimes \A\to M_\phi
\end{equation}
where the arrows $\A\rightrightarrows\Om(\mathbb{P}^1)_{dg}\otimes \A$ are given by $i_0\otimes 1_\A$ and $i_\infty\otimes\phi$. Notice under different identifications of $Pt_\infty\#\Z\simeq \C$ and $\Om(\mathbb{P}^1\times \Z)\#\Z\simeq \Om(\mathbb{P}^1 )$, these arrows could turn into $i_0\otimes \phi^{-1}$ and $i_1\otimes 1_\A$ either.

By (\ref{eq:coeqsympl}) and (\ref{eq:coeqalgtorus}), both $\cW(T_\phi)$ and $M_\phi$ are the homotopy coequalizers of equivalent diagrams. Hence, they are equivalent, completing the second proof of Theorem \ref{mainthmsymp}.
\begin{note}\label{note:finalnote}
One can see the commuting of smash products and localization in two ways: the first is writing explicit zigzags using the definition in \cite{GPS2}. More precisely, let $\B$ be a dg category with a strict $\Z$ action and let $C$ be a $\Z$-invariant set of morphisms. Then, by adding cones to $\B$ (and extending the action), the problem turns into showing that applying $\#\Z$ and taking the quotient by a $\Z$-invariant subcategory commute (i.e. $(\B/\B_0)\#\Z\simeq (\B\#\Z)/(\B_0\#\Z)$), which can be achieved using the explicit model in \cite{lyubaquot}, and \cite{partiallywrapped}. The hom-complexes for $(\B\#\Z)/(\B_0\#\Z)$ may look larger. To show it is equivalent to $(\B/\B_0)\#\Z$, one has to first extend the category $\B$ to a quasi-equivalent category by adding objects $(g,b)$ (for all $g\in\Z, b\in ob(\B_0)$) equivalent to $gb$. Then, the quotient of extended categories (with objects added to $\B_0$ as well) is quasi-equivalent to $\B/\B_0$. Smash product with this quasi-equivalent category gives $(\B\#\Z)/(\B_0\#\Z)$. 

The second way is to see $\B\#\Z$ as another colimit. Namely, consider the diagram of categories given by one category, $\B$, and endofunctors $g\in\Z$. Then, the corresponding Grothendieck construction (as in \cite{thomasongroth}) is exactly $\B\#\Z$. In this situation, one does not need to localize with respect to corresponding set of morphisms, as they are already invertible, and one can easily show the colimit property. 
Then, it is easy to see that $(C^{-1}\B)\#\Z$ and $C^{-1}(\B\#\Z)$ can be characterized by the same universal property.
\end{note}



\bibliographystyle{alpha}
\bibliography{biblioforspefibre}	
\end{document}